\newif\ifarxiv
\arxivtrue          

\ifarxiv
\documentclass[11pt]{article}
\usepackage[T1]{fontenc}
\usepackage{lmodern}
\usepackage[margin=1in]{geometry}
\usepackage[numbers,sort&compress]{natbib}
\else
\documentclass{article}
\usepackage{iclr2027_conference,times}
\fi

\usepackage{amsmath}
\usepackage{xcolor}
\ifarxiv
\definecolor{darkblue}{rgb}{0,0,.5}
\usepackage[
breaklinks=true,
pdfstartview=FitH,
pagebackref=true,
colorlinks=true,
linkcolor=darkblue,
citecolor=darkblue,
urlcolor=darkblue
]{hyperref}
\else
\usepackage{hyperref}
\fi

\usepackage{url}
\usepackage{subcaption}

\ifarxiv
\AtBeginDocument{
	\setlength{\abovedisplayskip}{4pt}
	\setlength{\belowdisplayskip}{4pt}
	\setlength{\abovedisplayshortskip}{2pt}
	\setlength{\belowdisplayshortskip}{2pt}
	\setlength{\jot}{1pt}
}

\allowdisplaybreaks
\fi

\ifarxiv
\usepackage[title]{appendix}
\fi

\title{Improved KKT Complexity for First-Order Bilevel Optimization under Weak Lower-Level Convexity}

\ifarxiv
\author{
	Jan Harold Alcantara%
	\thanks{\url{janharold.alcantara@riken.jp}. 
		Center for Advanced Intelligence Project, RIKEN, Tokyo, Japan.}
	\qquad
	Masahiro Inoue%
	\thanks{\url{inoue-masahiro@g.ecc.u-tokyo.ac.jp}. 
		Department of Mathematical Informatics, 
		Graduate School of Information Science and Technology, 
		University of Tokyo, Tokyo, Japan.}
	\qquad
	Akiko Takeda%
	\thanks{\url{takeda@mist.i.u-tokyo.ac.jp}. 
		Department of Mathematical Informatics, 
		Graduate School of Information Science and Technology, 
		University of Tokyo, Tokyo, Japan, and 
		Center for Advanced Intelligence Project, RIKEN, Tokyo, Japan.}
}
\date{\today}
\else
\author{
	Jan Harold Alcantara\\
	Center for Advanced Intelligence Project\\
	RIKEN\\
	\texttt{janharold.alcantara@riken.jp}\\
	\And
	Masahiro Inoue\\
	Graduate School of Information Science and Technology\\
	The University of Tokyo\\
	\texttt{inoue-masahiro@g.ecc.u-tokyo.ac.jp}\\
	\AND
	Akiko Takeda\\
	Graduate School of Information Science and Technology\\
	The University of Tokyo\\
	Center for Advanced Intelligence Project, RIKEN\\
	\texttt{takeda@mist.i.u-tokyo.ac.jp}
}
\fi

\usepackage{amsthm,amssymb}
\usepackage{xcolor}
\usepackage{mathtools}
\usepackage{algorithm}
\usepackage{algorithmic}
\usepackage{enumitem}
\usepackage{booktabs}
\usepackage{multirow}

\mathtoolsset{showonlyrefs} 

\newtheorem{theorem}{Theorem}
\newtheorem{corollary}[theorem]{Corollary}
\newtheorem{proposition}[theorem]{Proposition}
\newtheorem{lemma}[theorem]{Lemma}

\theoremstyle{definition}
\newtheorem{assumption}[theorem]{Assumption}
\newtheorem{definition}[theorem]{Definition}
\newtheorem{example}[theorem]{Example}

\theoremstyle{remark}
\newtheorem{remark}[theorem]{Remark}

\newcommand{\norm}[1]{\left\|#1\right\|}

\newcommand{\argmin}{\operatorname*{arg\,min}}
\newcommand{\dist}{\operatorname{dist}}

\newcommand{\R}{\mathbb{R}}
\newcommand{\alert}[1]{{#1}}
\begin{document}

\maketitle

\begin{abstract}
	We study deterministic first-order bilevel optimization under weak
	lower-level convexity, allowing nonconvex lower-level objectives and
	without assuming strong convexity, the Polyak--\L{}ojasiewicz condition,
	or an error-bound property. We consider a \alert{$\delta$-relaxed Moreau-gap constraint, with $\delta>0$, for} the lower-level stationarity condition and propose an inexact
	variable-smoothing penalty method (IVSP) for computing its approximate
	Karush--Kuhn--Tucker (KKT) points. \alert{For any fixed relaxation level $\delta$,} under a standard extended
	no-nonzero-abnormal-multiplier constraint qualification (ENNAMCQ), we
	prove finite stabilization of the adaptive penalty parameter and an
	overall $\widetilde O(\varepsilon^{-3})$ first-order complexity for
	computing an $\varepsilon$-KKT point. \alert{Notably, we give verifiable sufficient conditions for ENNAMCQ
		covering convex lower-level objectives without nonconstant affine
		segments (including the strictly convex case), and a class of
		nonconvex sample-reweighting models.}
	\alert{The positive relaxation avoids the intrinsic constraint-qualification
		degeneracy of the exact Moreau-gap constraint while achieving an
		$\mathcal O(\sqrt{\delta})$ lower-level near-stationarity guarantee.}
	Numerical experiments on synthetic and real-world
		bilevel learning problems illustrate the practical performance of IVSP.
\end{abstract}

\section{Introduction}

Bilevel optimization has found broad applications in hyperparameter
optimization~\citep{pedregosa2016hyperparameter,franceschi2018bilevel},
meta-learning~\citep{franceschi2018bilevel,rajeswaran2019meta},
neural architecture search~\citep{darts}, and data reweighting or
hypercleaning~\citep{ren2018learning}.
We study bilevel problems of the form
\begin{equation*}
	\min_{x\in X,\,y\in Y} f(x,y)
	\qquad
	\text{s.t.}\qquad
	y\in 
	\argmin_{u\in Y} g(x,u).
\end{equation*}
In this paper, $X\subset\mathbb R^n$ and $Y\subset\mathbb R^m$ are
nonempty closed convex sets, $f$ and $g$ are smooth, and $g(x,\cdot)$ is uniformly weakly
convex. Thus, the lower-level (LL) problem may be nonconvex and need not be
strongly convex, satisfy the Polyak--\L{}ojasiewicz (PL) condition, or
admit a unique solution.

The modeling flexibility of bilevel optimization comes, however, with a
substantial computational challenge: the upper-level (UL) objective depends
implicitly on a solution of the lower-level problem. Classical hypergradient methods typically differentiate through the
lower-level optimization process or use approximate implicit differentiation,
which can require Hessian/Jacobian-vector products or linear-system solves
\citep{ghadimi2018approximation,grazzi2020iteration,ji2021bilevel}. These costs have motivated a growing line of Hessian-free and
fully first-order methods~\citep{BOME,F2SA,MEHA,f2ba}. Despite this progress,
the strongest nonasymptotic guarantees remain concentrated in
well-conditioned lower-level regimes, such as strong convexity, PL, or
related error-bound conditions~\citep{F2SA,hjfbio,f2ba}.

Much less is known once such lower-level regularity is removed. For
convex, possibly nonsmooth lower-level problems, \citet{lumei} develop first-order penalty methods with approximate
weak-KKT guarantees and a
$\widetilde O(\varepsilon^{-4})$ complexity. For weakly convex/nonconvex lower levels,
MEHA~\citep{MEHA} and PNGBiO~\citep{PNGBiO} develop deterministic
first-order methods based on Moreau envelope penalty reformulations.
Their guarantees concern stationarity of a penalized reformulation
together with lower-level feasibility rather than the full KKT system
of the corresponding constrained problem. Requiring both
$O(\varepsilon)$ stationarity and $O(\varepsilon)$ Moreau-gap
feasibility gives an $O(\varepsilon^{-4})$ complexity. TSP~\citep{lu2025tsp}
instead directly targets approximate KKT points of a Moreau envelope
constrained reformulation using a stochastic primal--dual method.

A common route in the literature is to replace or relax the
lower-level condition by nonlinear constraints, leading to a
nonlinear-programming formulation of the bilevel problem. This
viewpoint is particularly natural when the lower-level solution need
not be unique, as may occur without strong convexity or PL-type
regularity, so that a smooth reduced hyperobjective is generally
unavailable. For such constrained formulations, the KKT conditions
constitute the canonical first-order optimality system of nonlinear
programming, encompassing stationarity, feasibility, and
complementarity; see, e.g., \citet{nocedal2006numerical}. Existing
bilevel methods nevertheless target several different solution
notions, including stationarity of penalized reformulations, weak
KKT-type conditions, and KKT systems of constrained reformulations;
see Table~\ref{tab:complexity-comparison}. These distinctions become particularly important in the weakly convex
lower-level regime, where both the attainable solution guarantee and
its deterministic first-order complexity remain substantially less
understood.

This motivates the following question:
\vspace{-1.2mm}
\begin{center}
	\emph{Can a deterministic first-order method achieve a genuine KKT
		guarantee for bilevel optimization with a weakly convex lower
		level at an improved complexity?}
\end{center}
\vspace{-3mm}
\subsection{Our approach and contributions}

Our approach is to replace the lower-level condition by a 
$\delta$-relaxed Moreau-gap constraint, with $\delta>0$, yielding a smooth nonlinear-programming
relaxation whose KKT system is our target solution concept. \alert{The positive relaxation is essential to this KKT viewpoint:
because the Moreau gap is nonnegative, the exact constraint is
intrinsically degenerate and fails standard constraint qualifications
at every feasible point. A fixed $\delta>0$ avoids this intrinsic degeneracy and permits a
standard constraint-qualification-based KKT analysis, while retaining
an $\mathcal O(\sqrt{\delta})$ lower-level near-stationarity guarantee;
see Section~\ref{sec:moreau-gap} and Appendix~\ref{app:moreau-relaxation}.}

We solve the relaxed problem using a smoothed exact-penalty method with
adaptive penalty updates and inexact Moreau evaluations. Our analysis
is carried out under ENNAMCQ (Definition~\ref{def:ENNAMCQ}), which
yields finite stabilization of the penalty parameter. \alert{We establish verifiable sufficient conditions
	for ENNAMCQ covering convex
	lower-level objectives without nonconstant affine segments (such as strictly and strongly convex cases) and a class of nonconvex
	sample-reweighting models; these conditions apply to our few-shot and
	data hyper-cleaning problems (Appendix~\ref{app:ennamcq-sufficient}).}  For a suitable
proximal parameter, the Moreau subproblems are strongly convex even
when the lower-level objective is nonconvex. Our analysis requires
these subproblems to be solved only to polynomially decreasing
accuracy, so linearly convergent first-order inner solves add only a
logarithmic overhead.

Our main contributions are as follows.
\vspace{-2mm}
\begin{itemize}[leftmargin=*,itemsep=1pt,topsep=3pt]
	
	\item \textbf{A KKT-oriented deterministic first-order method.}
	We develop the \emph{inexact variable-smoothing penalty method}
	(IVSP), which uses only first-order information and projections,
	accommodates inexact Moreau evaluations, and admits either a fixed
	stepsize or an inexact backtracking scheme.
	
	\item \textbf{Improved KKT complexity under weak lower-level
		convexity.}
	\alert{For any fixed relaxation level $\delta>0$,} under ENNAMCQ, the adaptive penalty parameter stabilizes after
	finitely many updates (Lemma~\ref{lem:penalty-stabilization}), and
	IVSP computes an $\varepsilon$-KKT point \alert{of the $\delta$-relaxed Moreau-gap formulation} with an overall
	$\widetilde O(\varepsilon^{-3})$ first-order complexity
	(Corollary~\ref{cor:overall-complexity}). These guarantees do
	not require strong convexity, PL, or an error-bound property of the
	lower-level problem. 
	
	\item \textbf{Empirical evaluation.}
		We compare IVSP with representative deterministic first-order
		bilevel methods on synthetic and real-world learning problems and
		observe competitive numerical performance.	
\end{itemize}
\vspace{-2mm}
\subsection{Related works}

\begin{table}[t]
	\vspace{-2mm}
	\centering
	\caption{{\small Comparison of representative deterministic first-order methods
			for bilevel optimization. \alert{The horizontal division separates methods
				relying on SC/PL lower-level regularity from those allowing convex or
				weakly convex lower-level problems without PL. SC/PL provides stronger
				control of lower-level errors and stability, while SC additionally
				guarantees uniqueness; such properties can be exploited to obtain
				sharper complexity guarantees.} Complexity is reported using a
			norm-based $\varepsilon$-stationarity/KKT criterion whenever applicable.}}
	\label{tab:complexity-comparison}
	\small
	\setlength{\tabcolsep}{3.2pt}
	\begin{tabular}{lcccc}
		\toprule
		Method & LL assumption & Approach & Guarantee & Complexity \\
		\midrule
		BOME~\citep{BOME}
		& PL & barrier & KKT$^\ast$
		& $O(\varepsilon^{-6})$--$O(\varepsilon^{-8})$ \\
		F$^2$SA~\citep{F2SA}
		& SC & penalty & hyperobj.\ stat.
		& $\widetilde O(\varepsilon^{-3})$ \\
		V-PBGD~\citep{VPBGD}
		& PL & penalty & pen. stat.
		& $\widetilde O(\varepsilon^{-2})$ \\
		SLM~\citep{SLM}
		& PL + reg. & primal-dual & KKT (fixed-$\delta$ relax.)
		& $O(\varepsilon^{-2})$ \\
		F$^2$BA~\citep{f2ba}
		& SC & penalty & hyperobj.\ stat.
		& $\widetilde O(\varepsilon^{-2})$ \\
		\midrule
		Lu--Mei~\citep{lumei}
		& convex & penalty/minimax & w-KKT
		& $\widetilde O(\varepsilon^{-4})$ \\
		MEHA~\citep{MEHA}
		& WC & ME penalty
		& pen. stat.\ + feas.
		& $O(\varepsilon^{-4})$ \\
		PNGBiO~\citep{PNGBiO}
		& WC & ME penalty
		& pen. stat.\ + feas.
		& $O(\varepsilon^{-4})$ \\
		\textbf{IVSP (ours)}
		& {WC} & {ME-gap sm.\ pen.}
		& {KKT (fixed-$\delta$ relax.)}
		& $\widetilde O(\varepsilon^{-3})$ \\
		\bottomrule
	\end{tabular}
	\vspace{1mm}
	\begin{minipage}{0.98\linewidth}
		\footnotesize
		SC/WC = strongly/weakly convex; PL = Polyak--\L{}ojasiewicz; ME = Moreau envelope.
		$\widetilde O(\cdot)$ suppresses logarithmic factors.
		The target formulations, solution concepts, and rate conventions differ
		across methods; see Appendix~\ref{app:complexity-comparison} for details. 
	\end{minipage}
	\vspace{-8mm}
\end{table}

Strong convexity, the Polyak--\L{}ojasiewicz (PL) condition, and related
error-bound properties provide quantitative control of lower-level
optimization errors, which facilitates the approximation of lower-level
responses and hypergradients. Under strong convexity, the lower-level solution is unique and classical
approaches exploit implicit or iterative differentiation. More broadly,
under SC/PL-type lower-level regularity, recent methods seek to avoid
second-order computations through penalty, primal--dual,
finite-difference, value-function, and other fully first-order
constructions.
Representative methods include BOME~\citep{BOME}, F$^2$SA~\citep{F2SA},
Prox-F$^2$BA~\citep{proxf2ba}, SLM~\citep{SLM},
HJFBiO~\citep{hjfbio}, F$^2$BA~\citep{f2ba},
PBGD-Free~\citep{pbgdfree}, and SGHA~\citep{sgha}.
In the strongly convex setting, fully first-order methods can attain the
near-optimal $\widetilde O(\varepsilon^{-2})$ complexity for
hyperobjective stationarity~\citep{f2ba}. Our interest is in regimes
where strong convexity, PL, and error-bound conditions are
unavailable.

Within this broader literature, relatively few methods directly target
KKT or KKT-type solution concepts. BOME~\citep{BOME} introduces a
KKT-type residual under the PL condition, although its criterion does
not correspond to the full KKT system of a constrained bilevel
reformulation. \citet{lumei} develop first-order penalty
methods with approximate weak-KKT guarantees for convex, possibly
nonsmooth lower-level problems. SLM~\citep{SLM} establishes KKT
guarantees for a positive value-function-gap formulation under the PL
condition and additional regularity assumptions, while
TSP~\citep{lu2025tsp} directly targets approximate KKT points of a
Moreau envelope constrained formulation using a stochastic
primal--dual method.

Turning specifically to weakly convex lower-level objectives, the
Moreau envelope reformulation introduced by \citet{gao2026moreau} has
led to several first-order methods. MEHA~\citep{MEHA} develops a
single-loop Hessian-free method, while PNGBiO~\citep{PNGBiO} extends
this approach to generalized smoothness; both establish stationarity
guarantees for penalized Moreau envelope reformulations together with
lower-level feasibility. Our method instead combines a positive
Moreau-gap constrained formulation with a smoothed exact penalty,
allows inexact proximal computations, and establishes deterministic
first-order complexity for computing approximate KKT points under weak
lower-level convexity. Since the solution concepts in these works
differ, Table~\ref{tab:complexity-comparison} reports both the
guarantee and its corresponding complexity.

\section{Proposed method}
\label{sec:method}
\subsection{Moreau-gap reformulation}
\label{sec:moreau-gap}

\alert{Let $g(x,\cdot)$ be $\sigma$-weakly convex for every $x\in X$, where
$\sigma\geq0$. That is, $g(x,\cdot) + \frac{\sigma}{2}\norm{\cdot}^2$ is convex for any $x\in X$.} Fix $\rho>0$ such that $\rho\sigma<1$. Define
\[
\psi_\rho(z;x,y)
:=g(x,z)+\frac{1}{2\rho}\|z-y\|^2,
\qquad
g_\rho(x,y)
:=\min_{z\in Y}\psi_\rho(z;x,y),
\]
and let $z_\rho(x,y)$ denote the unique minimizer. Following the
Moreau envelope reformulation of \citet{gao2026moreau}, we consider
\begin{equation}
	\min_{w\in C} f(w)
	\qquad\text{s.t.}\qquad
	R(w):=g(w)-g_\rho(w)-\delta\leq0,
	\label{eq:relaxed-ME}
\end{equation}
where $w=(x,y)$, $C:=X\times Y$, and $\delta>0$. Since $\psi_\rho(\cdot;w)$ is
$(\rho^{-1}-\sigma)$-strongly convex,
\[
g(w)-g_\rho(w)
\geq
\frac{1-\rho\sigma}{2\rho}
\|y-z_\rho(w)\|^2,
\]
and the optimality condition of the Moreau subproblem gives
\[
g(w)-g_\rho(w)=0
\iff
y=z_\rho(w)
\iff
0\in \nabla_y g(w)+N_Y(y),
\]
where $N_Y$ denotes the normal cone to $Y$. Thus, under weak convexity, the exact Moreau-gap constraint
characterizes lower-level first-order stationarity
\citep{MEHA}; when $g(x,\cdot)$ is convex, it is equivalent to
lower-level global optimality and recovers the original bilevel
constraint~\citep{gao2026moreau}.

As in \citep{lu2025tsp}, we use a positive relaxation $\delta>0$, which is important both for regularity
and approximation. Since $g-g_\rho\geq0$, every feasible point of the
exact constraint $g-g_\rho\leq0$ is a minimizer of the gap over $C$;
hence standard constraint qualifications such as MFCQ and NNAMCQ
necessarily fail there
\citep{lin2014simple,ye2023difference,gao2026moreau,bai2026agils}.
On the other hand, every feasible point of \eqref{eq:relaxed-ME}
satisfies
\[
\|y-z_\rho(w)\|
\leq
\sqrt{\frac{2\rho\delta}{1-\rho\sigma}}
~\text{and}~
\dist\!\left(
0,\nabla_y g(x,z_\rho(w))+N_Y(z_\rho(w))
\right)
\leq
\sqrt{\frac{2\delta}{\rho(1-\rho\sigma)}}.
\]
Hence, $\delta$ directly controls lower-level near stationarity at
order $O(\sqrt{\delta})$. Further properties of the positive
relaxation, including its behavior as $\delta\downarrow0$, are given
in Appendix~\ref{app:moreau-relaxation}.

\subsection{Overview of the smoothed exact-penalty method}
\label{sec:overview_smoothedmodel}
Our approach to solving \eqref{eq:relaxed-ME} begins with the exact-penalty
model
\begin{equation*}
	\min_{w\in C}\;
	f(w)+\frac{1}{\gamma}[R(w)]_+,
\end{equation*}
where $[t]_+:=\max\{t,0\}$ and $\gamma>0$ is the inverse penalty
parameter. Thus, decreasing $\gamma$ strengthens the penalty. Under Assumption~\ref{assume:A} and $\rho\sigma<1$, the residual $R$
is continuously differentiable with
\begin{equation}
	\nabla R(w)=
	\begin{pmatrix}
		\nabla_x g(x,y)-\nabla_x g(x,z_\rho(w))\\[1mm]
		\nabla_y g(x,y)-\rho^{-1}(y-z_\rho(w))
	\end{pmatrix};
	\label{eq:grad-R}
\end{equation}
see Appendix~\ref{app:moreau-properties}. The exact-penalty model remains nonsmooth because of $[\cdot]_+$. We
therefore replace it by a smooth approximation $\phi_\mu$ with smoothing
parameter $\mu>0$ and consider
\begin{equation}
	\min_{w\in C}
	M_{\gamma,\mu}(w)
	:=\gamma f(w)+G_\mu(w),
	\qquad
	G_\mu(w):=\phi_\mu(R(w)).
	\label{eq:smoothed-penalty}
\end{equation}
We use a standard family of smooth approximations satisfying the conditions
in Appendix~\ref{app:smoothing-properties}; see \citep{beck2012smoothing}. In particular,
$\phi_\mu$ approximates $[\cdot]_+$ to accuracy $O(\mu)$,
$0\leq\phi_\mu'\leq1$, and $\phi_\mu'$ is $O(\mu^{-1})$-Lipschitz.
A standard example is the softplus
$\phi_\mu(t)=\mu\log(1+\exp(t/\mu))$.

By the chain rule,
$
\nabla G_\mu(w)
=\phi_\mu'(R(w))\nabla R(w),
$
and $\nabla G_\mu$ is $O(\mu^{-1})$-Lipschitz; see
Appendix~\ref{app:smoothing-properties}. Hence, an exact projected-gradient
step for \eqref{eq:smoothed-penalty} takes the familiar form
\begin{equation*}
	w_L^k
	=
	P_C\!\left(
	w^k-\frac{\mu}{L}
	\big[\gamma\nabla f(w^k)+\nabla G_{\mu}(w^k)\big]
	\right),
\end{equation*}
where $P_C$ denotes the Euclidean projection onto $C$.
The difficulty is that computing this step requires the Moreau proximal
point $z_\rho(w^k)$ due to \eqref{eq:grad-R}, which is generally unavailable exactly. This motivates an inexact Moreau evaluation together with explicit
control of the resulting value and gradient errors.

\subsection{Inexact Moreau envelope evaluations}
\label{sec:moreau-oracle}

For a prescribed accuracy $\zeta\geq0$, let $\widetilde z\in Y$ be an
approximation of $z_\rho (w)$ satisfying
\begin{equation}
	0\leq
	\psi_\rho(\widetilde z;x,y)-g_\rho(x,y)
	\leq\zeta.
	\label{eq:oracle-gap}
\end{equation}
Define the corresponding inexact residual by
$
\widetilde R(w)
\coloneqq
g(x,y)-\psi_\rho(\widetilde z;x,y)-\delta.
$
Then
\begin{equation}
	\widetilde R(w)
	\leq
	R(w)
	\leq
	\widetilde R(w)+\zeta.
	\label{eq:residual-error-bound}
\end{equation}

Likewise, let $\widetilde\nabla R(w)$ be obtained from
\eqref{eq:grad-R} by replacing $z_\rho(w)$ with $\widetilde z$, and set
$
\widetilde\nabla G_\mu(w)
\coloneqq
\phi_\mu'(\widetilde R(w))
\widetilde\nabla R(w).
$
Under Assumption~\ref{assume:A},
\begin{equation}
	\bigl\|
	\widetilde\nabla G_\mu(w)-\nabla G_\mu(w)
	\bigr\|
	=
	O\!\left(
	\sqrt{\zeta}+\frac{\zeta}{\mu}
	\right).
	\label{eq:oracle-gradient-error}
\end{equation}
Explicit constants and detailed error estimates are given in
Appendix~\ref{app:oracle-error}. The bounds
\eqref{eq:residual-error-bound}--\eqref{eq:oracle-gradient-error}
are the only properties of the inexact Moreau evaluations needed in
the outer method.

\alert{\textbf{Choice of inner solver.}
For a suitable proximal parameter, each Moreau subproblem is strongly
convex. Standard first-order methods can therefore exploit strong
convexity to obtain linear convergence. In particular, projected
gradient~\citep{Beck17} has this property under our assumptions and is
the method used in our implementation. More generally, the complexity
argument only requires an inner solver that computes a
$\zeta$-accurate solution in $O(\log(1/\zeta))$ iterations uniformly
over the proximal subproblems. A computable residual-based stopping
rule is given in Appendix~\ref{app:complete-algorithm}.}

\subsection{Inexact variable-smoothing penalty method}
\label{sec:algorithm}

We now combine the inexact Moreau evaluations with projected-gradient
steps for~\eqref{eq:smoothed-penalty}. The outer variable-smoothing and adaptive-penalty mechanism of IVSP is
adapted from sESQM~\citep{xu2026smoothing}, which was developed for
single-level nonconvex optimization with nonlinear inequality constraints
and applies one projected-gradient step to a smoothed exact penalty at each
iteration while updating the smoothing and penalty parameters. In our bilevel setting, however,
the Moreau-gap residual and its gradient depend on the generally
unavailable proximal point $z_\rho(w)$. IVSP therefore incorporates
the inexact Moreau evaluations of Section~\ref{sec:moreau-oracle} and
controls the resulting value and gradient errors in the outer method.

Let $\{\mu_k\}$ be a positive sequence with
$\mu_k\downarrow0$, let $\{\zeta_k\}$ be the prescribed proximal
accuracies, and let $\{\widehat\gamma_t\}_{t\geq0}$ be a positive
inverse-penalty grid with $\widehat\gamma_t\downarrow0$. We set
$\gamma_k\coloneqq\widehat\gamma_{t_k}$, where $\{t_k\}$ is
nondecreasing. Given a $\zeta_k$-accurate proximal point $\widetilde z^k$ at $w^k$,
form the corresponding $\widetilde R^k = \widetilde{R}(w^k)$ and
$\widetilde\nabla G^k
\coloneqq\phi_{\mu_k}'(\widetilde R^k)\widetilde\nabla R(w^k)$. For a parameter $L>0$, define
\begin{equation}
	w_L^k
	=
	P_C\!\left(
	w^k-\frac{\mu_k}{L}
	\left[
	\gamma_k\nabla f(w^k)
	+\widetilde\nabla G^k
	\right]
	\right).
	\label{eq:inexact-outer-step}
\end{equation}
The complete method is summarized in Algorithm~\ref{alg:main}.


\textbf{Choice of $L$ and inexact line search.} We consider two strategies for selecting $L$.
In Strategy I (fixed $L$), a sufficiently large constant $L$ is used
throughout; a sufficient choice $L_{\rm safe}$ is given in
Lemma~\ref{lem:descent}. This avoids line search but requires an
a priori estimate of $L_{\rm safe}$. Strategy II instead selects $L$
adaptively by backtracking and does not require $L_{\rm safe}$ to be
known in advance. Specifically, starting from
$L_{k,0}\in[L_{\min},L_{\max}]$, we increase $L$ until an inexact
Armijo condition is satisfied. For each trial $w_L^k$, compute a
$\zeta_{k+1}$-accurate proximal point $\widetilde z_{k,L}^+$ and the
corresponding residual $\widetilde R_{k,L}^+$. We accept the trial if
\begin{align}
	&\gamma_k f(w_L^k)
	+\phi_{\mu_k}
	\bigl(\widetilde R_{k,L}^++\zeta_{k+1}\bigr)
	\leq
	\gamma_k f(w^k)
	+\phi_{\mu_k}(\widetilde R^k)
	-\frac{c_1}{2\mu_k}\|w_L^k-w^k\|^2
	+\tau_{k,L},
	\label{eq:robust-armijo}
\end{align}
where $\tau_{k,L}\geq0$ is an explicit error allowance induced by the
inexact Moreau evaluation, rather than an additional tuning parameter;
its definition is given in
Appendix~\ref{app:lem-descent}. If the
condition fails, we increase $L$ (e.g., $L\leftarrow2L$) and recompute
the trial point and its Moreau evaluation. In practice,
$L_{k,0}$ may be initialized using a safeguarded Barzilai--Borwein
spectral estimate~\citep{barzilai1988two}.

\textbf{Adaptive penalty update.} After obtaining $w^{k+1}$ and
$\widetilde R^{k+1}$, we strengthen the penalty whenever
\begin{equation}
	\widetilde R^{k+1}>2a_\phi\mu_k
	\qquad\text{and}\qquad
	\|w^{k+1}-w^k\|
	\leq c_2\mu_k\gamma_k,
	\label{eq:penalty-update-conditions}
\end{equation}
where $a_\phi>0$ is the smoothing-approximation constant defined in
Appendix~\ref{app:smoothing-properties}. The first condition detects
violation beyond the smoothing scale, which cannot be caused by the
inexact Moreau evaluation since
$\widetilde R^{k+1}\leq R(w^{k+1})$; the second requires a small outer step. Thus, if a significant violation persists
despite a small step, we move to the next value of the inverse-penalty
grid $\{\widehat\gamma_t\}$.

\begin{algorithm}[t]
	\caption{Inexact Variable-Smoothing Penalty Method (IVSP)}
	\label{alg:main}
	\begin{algorithmic}[1]
		\STATE Choose $w^0\in C$, $\{\mu_k\}\downarrow0$,
		$\{\zeta_k\}\downarrow0$, $\{\widehat\gamma_t\}\downarrow0$,
		and set $t_0=0$. Compute a $\zeta_0$-accurate proximal
		solution $\widetilde z^0$ at $w^0$ and form
		$\widetilde R^0$.
		\FOR{$k=0,1,\ldots$}
		\STATE Set $\gamma_k=\widehat\gamma_{t_k}$ and form
		$\widetilde\nabla G^k$ from
		$\widetilde z^k$ and $\widetilde R^k$.
		\STATE Choose $L_k$ by Strategy I (fixed $L$) or
		Strategy II (backtracking via~\eqref{eq:robust-armijo}),
		obtaining $w^{k+1}=w_{L_k}^k$ defined in \eqref{eq:inexact-outer-step} and a
		$\zeta_{k+1}$-accurate proximal solution
		$\widetilde z^{k+1}$ at $w^{k+1}$.
		\STATE Form $\widetilde R^{k+1}$ and set
		$t_{k+1}=t_k+1$ if
		\eqref{eq:penalty-update-conditions} holds; otherwise,
		set $t_{k+1}=t_k$.
		\ENDFOR
	\end{algorithmic}
\end{algorithm}

\section{Theoretical guarantees}
\label{sec:theory}

\subsection{Assumptions and stationarity measures}
\label{sec:assumptions}

We impose the following standard assumptions. 

\begin{assumption}
	\label{assume:A}
	\begin{itshape}
		The sets $X\subset\mathbb R^n$ and $Y\subset\mathbb R^m$ are
		nonempty, compact, and convex. The functions $f$ and $g$ are
		continuously differentiable on neighborhoods of their respective
		domains, and $\nabla f$ and $\nabla g$ are $L_f$- and
		$L_g$-Lipschitz continuous, respectively, on $C=X\times Y$.
	\end{itshape}
\end{assumption}

By compactness,
$f_{\inf}\coloneqq\min_{w\in C}f(w)$ is well defined and
finite. The Lipschitz continuity of $\nabla g$ implies that $g(x,\cdot)$ is weakly convex uniformly over
$x\in X$; see, e.g., \citet[Lemma~2.64]{bauschke2017convex}.
We denote by $\sigma\geq0$ a uniform weak-convexity modulus
(and one may always take $\sigma=L_g$), and fix
$\rho>0$ such that $\rho\sigma<1$. 
%
Since our target is the constrained relaxation
\eqref{eq:relaxed-ME}, we measure convergence through its
KKT residuals; cf.~\citet{lu2025tsp}. This differs from
penalty-stationarity criteria used, for example, in PBGD and
MEHA~\citep{shen2025penalty,VPBGD,MEHA}, whose stationary systems depend
explicitly on a prescribed penalty weight. By contrast, the KKT multiplier is determined jointly with the primal
variables, while stationarity, feasibility, and complementarity are
controlled separately. 

\begin{definition}[Approximate KKT point]
	\label{def:approx-KKT}
	For $w\in C$ and $\lambda\geq0$, define
	\begin{align*}
		\Phi_\rho^{\rm s}(w,\lambda)
		&\coloneqq
		\dist\!\left(
		0,\nabla f(w)+\lambda\nabla R{}(w)+N_C(w)
		\right), \notag\\
		\Phi_\rho^{\rm f}(w)
		&\coloneqq[R{}(w)]_+,
		\qquad
		\Phi_\rho^{\rm c}(w,\lambda)
		\coloneqq|\lambda R{}(w)|.
	\end{align*}
	We call $w$ an
	$(\varepsilon_s,\varepsilon_f,\varepsilon_c)$-KKT point of
	\eqref{eq:relaxed-ME} if there exists $\lambda\geq0$
	such that
	$\Phi_\rho^{\rm s}(w,\lambda)\leq\varepsilon_s$,
	$\Phi_\rho^{\rm f}(w)\leq\varepsilon_f$, and
	$\Phi_\rho^{\rm c}(w,\lambda)\leq\varepsilon_c$.
	When the three tolerances equal $\varepsilon$, we simply call $w$
	an $\varepsilon$-KKT point. The case $\varepsilon=0$ gives the
	standard KKT conditions.
\end{definition}

To control the adaptive penalty parameter, we impose the
\emph{extended no-nonzero-abnormal-multiplier constraint
	qualification} (ENNAMCQ). Conditions of this type are standard in
value-function reformulations of bilevel programs
\citep{lin2014simple,ye2023difference,alcantara2026smoothing}, and
have also been used for relaxed Moreau envelope reformulations;
see, e.g., \citet{gao2026moreau}. ENNAMCQ extends the classical NNAMCQ
used in nonlinear programming~\citep{ye2001multiplier} by imposing
the same nonstationarity condition also at infeasible points. For
the smooth single-constraint problem~\eqref{eq:relaxed-ME}, it also
coincides with the constraint qualification used in
sESQM~\citep{xu2026smoothing}. This condition is the key ingredient in proving finite
stabilization of the penalty parameter, which in turn is essential
for the complexity analysis.

\begin{definition}[ENNAMCQ]
	\label{def:ENNAMCQ}
	We say that ENNAMCQ holds for
	\eqref{eq:relaxed-ME} if, for every $w\in C$ with
	$R{}(w)\geq0$, $0\notin\nabla R{}(w)+N_C(w).$
\end{definition}


By compactness, ENNAMCQ yields a uniform separation from constraint
stationarity over all active or infeasible points; see
Lemma~\ref{lem:ennamcq-margin}. 
We provide verifiable sufficient conditions in
	Appendix~\ref{app:ennamcq-sufficient} under which ENNAMCQ holds
	automatically. For a positive Moreau-gap relaxation,
	\citet{gao2026moreau} established automatic ENNAMCQ under strict
	lower-level convexity in a more general setting with an $x$-dependent
	lower-level feasible set and additional \textit{joint }weak-convexity
	assumptions. In our setting, strict convexity of
	$g(x,\cdot)$ over $Y$ guarantees ENNAMCQ. We further extend
	this sufficient condition beyond strict convexity: for a convex
	lower-level problem, it suffices that every affine segment of
	$g(x,\cdot)$ on $Y$ be constant. We further establish a sufficient condition for nonconvex sample-reweighting models, covering our few-shot and data hyper-cleaning settings.
	
%
\subsection{Basic convergence properties}
\label{sec:basic-convergence}

We first establish two properties underlying the complexity analysis: descent of the proposed step and finite stabilization of the adaptive penalty parameter.
We define the Lyapunov function
\begin{equation*}
	Q_k
	\coloneqq
	\gamma_k\bigl(f(w^k)-f_{\inf}\bigr)
	+G_{\mu_k}(w^k)+c_\phi\mu_k,
	\label{eq:Lyapunov}
\end{equation*}
where $c_\phi$ is the constant in the smoothing bound
\eqref{eq:phi-changing-mu}.

\begin{lemma}[Well-definedness and descent]
	\label{lem:descent}
	Under Assumption~\ref{assume:A}, there exists a constant
	$L_{\rm safe}>0$, independent of $k$, such that the following
	hold.
	\begin{enumerate}[itemsep=1pt]
		\item[(i)]
		Every trial parameter $L\geq L_{\rm safe}$ satisfies the
		inexact Armijo condition~\eqref{eq:robust-armijo}.
		Consequently, any fixed $L\geq L_{\rm safe}$ is admissible
		under Strategy I, while the backtracking procedure in
		Strategy II terminates finitely at every iteration.
		Moreover, under Strategy II, the accepted parameters $L_k$
		and the number of backtracking trials are uniformly bounded.
		
		\item[(ii)]
		Under either strategy, defining $L_k\coloneqq L$ for Strategy I, we have
		\begin{equation}
			Q_{k+1}
			\leq
			Q_k
			-\frac{c_1}{2\mu_k}\norm{d^k}^2
			+\tau_{k,L_k},\quad \forall k\geq 0.\notag 
		\end{equation} 
	\end{enumerate}
\end{lemma}

An explicit expression for
$L_{\rm safe}$, together with uniform bounds on $L_k$
and the number of backtracking trials, is given in
Appendix~\ref{app:lem-descent}. 

\begin{lemma}[Finite penalty stabilization]
	\label{lem:penalty-stabilization}
	Suppose ENNAMCQ holds for~\eqref{eq:relaxed-ME},
	the proximal accuracies satisfy $\zeta_k=o(\mu_k)$, and the
	inverse-penalty grid satisfies
	$\widehat\gamma_t\leq M_\gamma\widehat\gamma_{t+1}$
	for some $M_\gamma>1$ and every $t\geq0$.
	Then there exists $\underline\gamma>0$ such that
	$\gamma_k\geq\underline\gamma$ for all $k$. Consequently, only
	finitely many penalty updates occur, and $\gamma_k$ is eventually
	constant.
\end{lemma}



\subsection{Complexity guarantees}
\label{sec:complexity}

We now establish the convergence rates of the proposed method under
polynomial smoothing and proximal-accuracy schedules. 

\begin{theorem}[Outer-iteration complexity]
	\label{thm:complexity}
Suppose Assumption~\ref{assume:A} and ENNAMCQ hold, and suppose
that the inverse-penalty grid satisfies
$\widehat\gamma_t\leq M_\gamma\widehat\gamma_{t+1}$ for some
$M_\gamma>1$ and every $t\geq0$. Let $\{\mu_k\}$ be positive and
nonincreasing with \alert{$\mu_k=\Theta((1+k)^{-r})$} for some $r\in(0,1)$,
and let $\{\zeta_k\}$ satisfy
$\zeta_k\leq\bar\zeta\mu_k^q$ for some $\bar\zeta>0$ and $q>1/r$.
Let $\{w^k\}$ be generated by Algorithm~\ref{alg:main} and, for each
$k\geq0$, set
$\widetilde\lambda_{k+1}
\coloneqq
\gamma_k^{-1}\phi_{\mu_k}'(\widetilde R^{k+1})$.
	Then, for every sufficiently large $K$, there exists
	$\widehat k\in[\lceil K/2\rceil,K]$ such that
\begin{equation}
	\begin{aligned}
		\Phi_\rho^{\rm s}(w^{\widehat k+1},\widetilde\lambda_{\widehat k+1})
		&=O\!\left(K^{-(1-r)/2}\right),\qquad
		\Phi_\rho^{\rm f}(w^{\widehat k+1})=O(K^{-r}),\\
		\Phi_\rho^{\rm c}(w^{\widehat k+1},\widetilde\lambda_{\widehat k+1})
		&=O(K^{-r}).
	\end{aligned}
	\label{eq:main-rates}
\end{equation}
	Consequently, within
		$
		K
		=
		O\!\left(
		\max\left\{
		\varepsilon_s^{-2/(1-r)},
		\varepsilon_f^{-1/r},
		\varepsilon_c^{-1/r}
		\right\}
		\right)
	$
	outer iterations, we obtain an
	$(\varepsilon_s,\varepsilon_f,\varepsilon_c)$-KKT point of the $\delta$-relaxed Moreau-gap formulation~\eqref{eq:relaxed-ME}. 
\end{theorem}

The rates in~\eqref{eq:main-rates} show that smaller $r$ favors stationarity, whereas larger $r$ favors feasibility and complementarity. Balancing $(1-r)/2=r$ gives $r=1/3$. Thus, for any $q>3$, Algorithm~\ref{alg:main} reaches an $\varepsilon$-KKT iterate within $O(\varepsilon^{-3})$ outer iterations. Accounting additionally for the logarithmic cost of the inexact Moreau
evaluations yields the
following overall first-order complexity.

\begin{corollary}[Overall first-order complexity]
	\label{cor:overall-complexity}
	Suppose the conditions of Theorem~\ref{thm:complexity} hold and,
	in addition, choose $\zeta_k=\Theta(\mu_k^q)$. Suppose each
	required Moreau subproblem is solved by projected gradient to
	objective-gap accuracy $\zeta_k$. Let $B_{\rm ls}$ denote a uniform bound on the
	number of trial Moreau evaluations per outer iteration, with
	$B_{\rm ls}=1$ for Strategy I. Then the overall
	first-order/projection complexity of IVSP through outer iteration
	$K$ is $
		O\!\left(
		B_{\rm ls}\,
		K\log K
		\right).$
	In particular, for the balanced choice $r=1/3$, IVSP computes an
	$\varepsilon$-KKT point of the $\delta$-relaxed Moreau-gap formulation~\eqref{eq:relaxed-ME} with overall first-order/projection
	complexity
	$
		O\!\left(
		B_{\rm ls}\,
		\varepsilon^{-3}
		\log\frac{1}{\varepsilon}
		\right)
		=
		\widetilde O(\varepsilon^{-3}).$
\end{corollary}

\section{Experiments}
\label{sec:experiments}

In this section, we evaluate IVSP on a synthetic nonconvex problem and two
real-world bilevel learning tasks: data
hyper-cleaning and few-shot learning.
We compare IVSP with representative deterministic first-order bilevel methods,
including PNGBiO~\citep{PNGBiO}, MEHA~\citep{MEHA}, SLM~\citep{SLM},
F$^2$SA~\citep{F2SA}, V-PBGD~\citep{VPBGD}, and BOME~\citep{BOME}.
For IVSP, we use the blockwise smoothing schedule
$\mu_k=\mu_0\bigl(1+\lfloor k/B\rfloor\bigr)^{-r}$,
which satisfies the smoothing condition of
Theorem~\ref{thm:complexity}.
Detailed problem formulations, implementation settings, and
algorithmic hyperparameter values, including $\mu_0$, $B$, and $r$ are provided in Appendix~\ref{app:experimental-details}.

\subsection{Synthetic numerical experiment}
\label{sec:synthetic-nonconvex}

We first consider the smooth nonconvex toy problem used in \citep{MEHA}:
\begin{equation}
    \min_{\substack{x\in[-100,100]\\
                    y\in[-100,100]^n}}
    \ (x-a)^2 + \|y-ae-c\|^2 
    \quad\text{s.t.}\quad
    y\in\argmin_{z\in[-100,100]^n}
    \sum_{i=1}^n \sin(x+z_i-c_i),
\notag 
\end{equation}
where $e$ is the all-ones vector.
The LL objective is smooth, nonconvex, and
$1$-weakly convex.
We set $n=1000$, $a=2$, $c_i=2$, $x^0=-6$, and $y^0=0$.
The optimal solution is
$x^*=((1-n)a+nC)/(1+n)$ and $y_i^*=C+c_i-x^*$,
where $C \in \argmin_{c\in\{-\pi/2+2k\pi:\,k\in\mathbb Z\}} |c-2a|.$

Figure~\ref{fig:synthetic-nonconvex} reports the normalized errors in
$x$ and $y$ and the KKT residual
$\max\{\Phi_\rho^{\rm s},\Phi_\rho^{\rm f},\Phi_\rho^{\rm c}\}$ for IVSP.
Since the competing methods target different reformulations and solution concepts, we report the IVSP KKT measure only for IVSP. IVSP achieves the
fastest reduction and smallest final errors in both variables.

\begin{figure}[h]
	\centering
	\includegraphics[width=.8\linewidth]{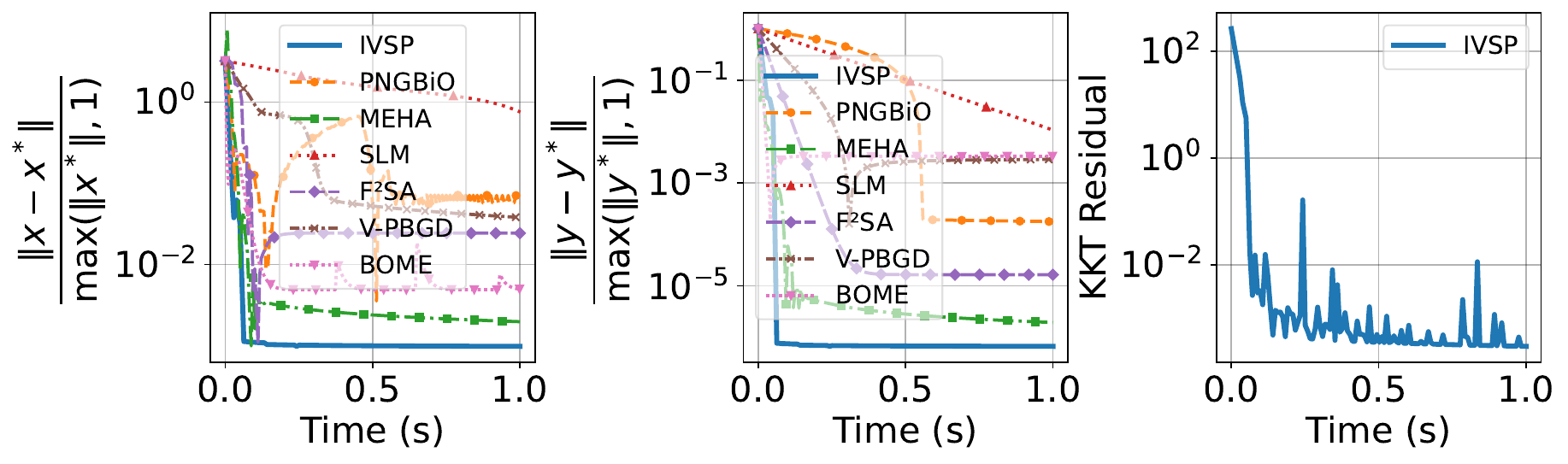}
	\caption{Convergence comparison on the  nonconvex problem with $n=1000$ measured by $\|x-x^*\|/\max\{\|x^*\|,1\}$, $\|y-y^*\|/\max\{\|y^*\|,1\}$,
		and the KKT residual
		$\max\{\Phi_\rho^{\mathrm s},\Phi_\rho^{\mathrm f},\Phi_\rho^{\mathrm c}\}$.}
	\label{fig:synthetic-nonconvex}
	\vspace{-3mm}
\end{figure}

\subsection{Real-world applications}
\label{sec:real-world}


\begin{figure}[t]
    \centering
    \begin{subfigure}[t]{\textwidth}
        \centering
        \includegraphics[width=.8\linewidth]{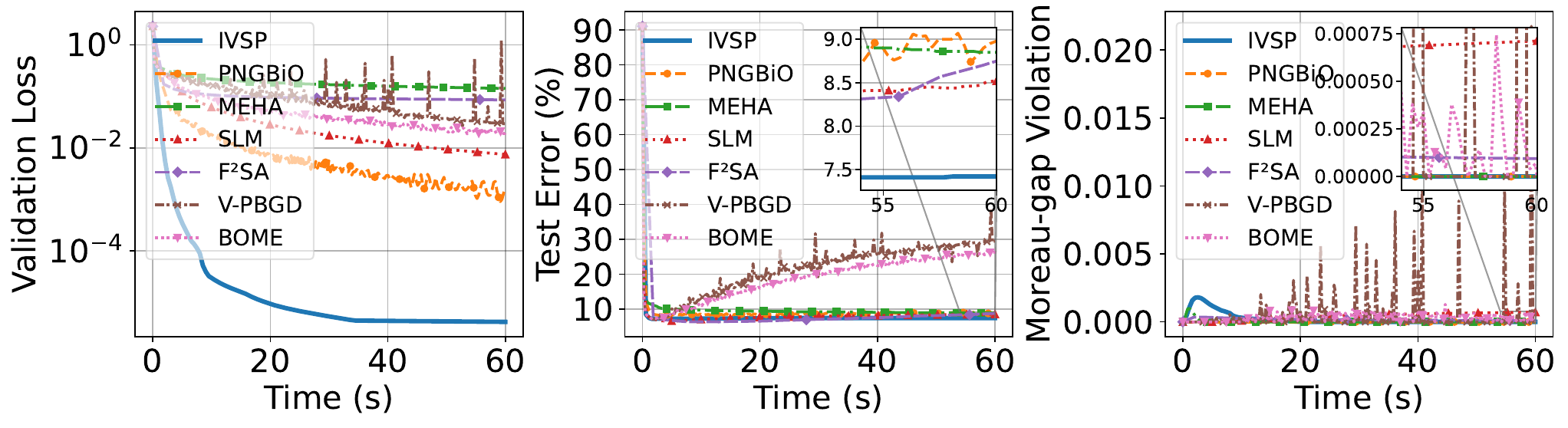}
        \caption{Data hyper-cleaning: MNIST}
        \label{fig:hypercleaning-mnist}
    \end{subfigure}

    \vspace{0.5em}

    \begin{subfigure}[t]{\textwidth}
        \centering
        \includegraphics[width=.8\linewidth]{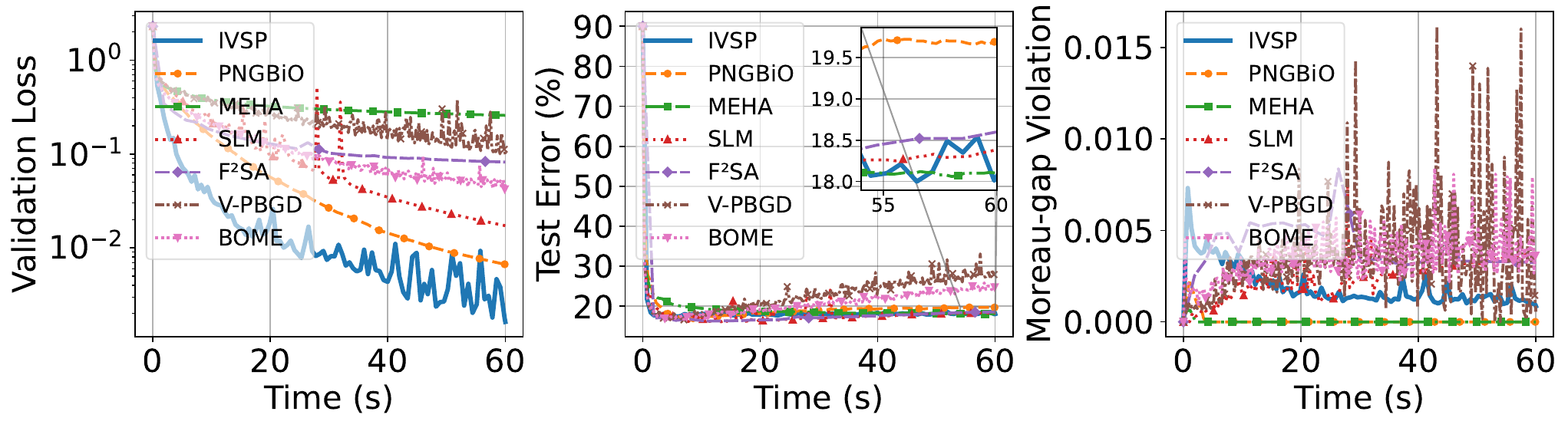}
        \caption{Data hyper-cleaning: FashionMNIST}
        \label{fig:hypercleaning-fashion-mnist}
    \end{subfigure}
    \caption{Comparison results for data hyper-cleaning on MNIST and FashionMNIST. The third panels report the estimated Moreau-gap
constraint violation.}
    \label{fig:hypercleaning}
    \vspace{-3mm}
\end{figure}

\noindent \textbf{Data hyper-cleaning.}
We consider data hyper-cleaning~\citep{SLM, MEHA, PNGBiO} on
MNIST~\citep{MNIST} and FashionMNIST~\citep{FashionMNIST}, where the
UL variables assign weights to the training samples, half of which have corrupted labels, and the LL model is a two-layer fully connected network with sigmoid activation,
yielding a smooth nonconvex LL problem.
\alert{This problem falls within the sigmoid-reweighted classification model
of Corollary~\ref{cor:ennamcq-hypercleaning}, so ENNAMCQ holds for our
choice $\delta=10^{-3}$.}

As shown in Figure~\ref{fig:hypercleaning}, IVSP reduces the validation loss more rapidly than the
compared methods and attains the lowest validation loss and lowest test error on both datasets.
The estimated Moreau-gap constraint violation, which estimates $[R(w)]_+$, is driven essentially
to zero on MNIST and remains small on FashionMNIST, although PNGBiO
and MEHA reach zero violation there, consistent with their
unrelaxed Moreau-gap feasibility target.

\noindent \textbf{Few-shot learning.}
We consider 10-way 1-shot classification on Omniglot~\citep{Omniglot}
following the bilevel meta-learning setting of \citet{MEHA}, with the goal of assessing whether the learned representation transfers effectively to unseen tasks.
The UL variable parameterizes a shared Conv-4 feature extractor, while the LL variables are task-specific linear classifiers, yielding a convex
but not strongly convex LL problem.
\alert{Proposition~\ref{prop:ennamcq-fewshot} shows that ENNAMCQ holds
for every $\delta>0$.}
Hyperparameters are selected using only meta-validation query loss.

As shown in Figure~\ref{fig:fewshot-omniglot}, IVSP attains the lowest validation and test query losses as well as the lowest test
classification error.
The similar validation and test behavior suggests that the selected
configuration generalizes to the held-out test classes in this experiment.

\begin{figure}[h]
	\centering
	\includegraphics[width=.8\linewidth]{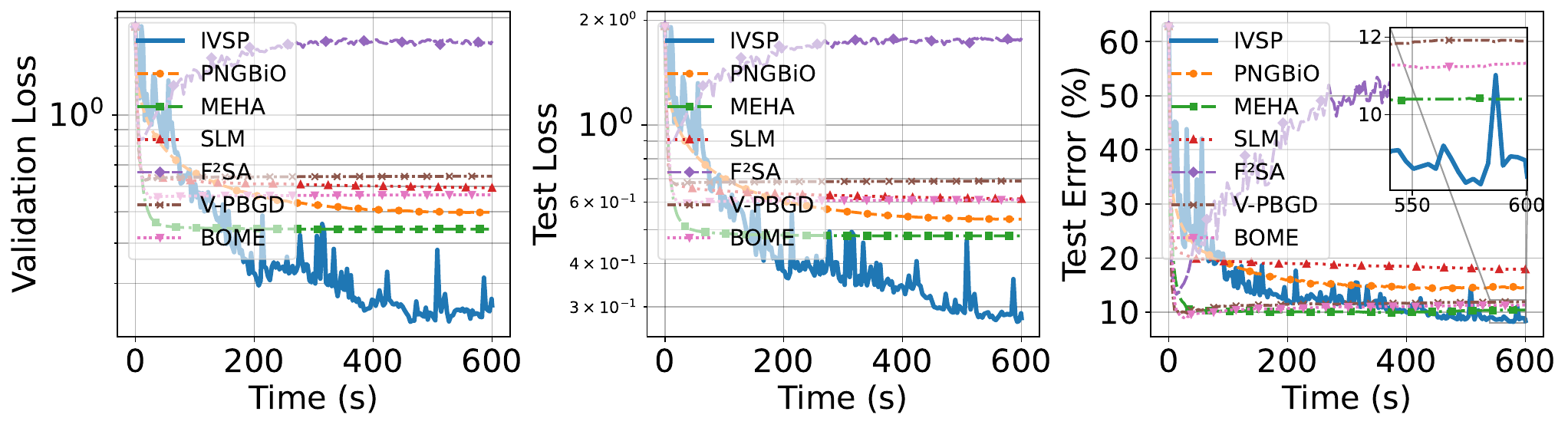}
	\caption{{ Comparison results for few-shot learning on Omniglot.}}
	\label{fig:fewshot-omniglot}
	\vspace{-3mm}
\end{figure}
\section{Conclusion}

We developed IVSP, a deterministic first-order method for bilevel optimization
with weakly convex lower-level objectives, using a positive Moreau-gap
relaxation and a smoothed exact-penalty approach. Under ENNAMCQ, IVSP achieves
finite penalty stabilization and an overall
$\widetilde O(\varepsilon^{-3})$ first-order complexity for computing an
$\varepsilon$-KKT point, without strong convexity, PL, or error-bound
assumptions. Numerical experiments illustrate its competitive performance.
Future work includes improving the complexity and further weakening
the ENNAMCQ-type conditions required for finite penalty stabilization.

\ifarxiv 
\subsection*{AI Use Statement}
\fi 
ChatGPT (OpenAI), using the GPT-5.5 and GPT-5.6 Sol models over the course of the project, was used as an auxiliary tool in parts of the theoretical and experimental workflow, as well as in manuscript preparation. In the theoretical work, ChatGPT was used to assist with checking calculations, derivations, proof arguments, and error terms arising from inexact computations. For the numerical experiments, it was used to assist with code implementation and debugging, as well as with developing parameter-tuning strategies. ChatGPT was also used to help identify potentially relevant literature and to assist with restructuring, editing, and polishing parts of the manuscript for clarity and readability. All AI-assisted mathematical arguments, references, code, parameter choices, and reported results were reviewed and verified by the authors. The authors take full responsibility for the final content of the manuscript.

%

\ifarxiv
 \begin{appendices}
\else
\bibliographystyle{iclr2027_conference}
\bibliography{ivsp_bibfile}
\newpage
\appendix 
\fi

\section{Mathematical Preliminaries}
\subsection{Moreau envelope properties}
\label{app:moreau-properties}

We first collect some consequences of Assumption~\ref{assume:A}
that will be used throughout the analysis. Since $\nabla g$ is
$L_g$-Lipschitz continuous on $C=X\times Y$, there exist finite
constants $L_{yx},L_{xy},L_{yy}>0$ such that
\begin{align}
	\|\nabla_y g(x_1,y)-\nabla_y g(x_2,y)\|
	&\leq L_{yx}\|x_1-x_2\|, \label{eq:Lyx}\\
	\|\nabla_x g(x,y_1)-\nabla_x g(x,y_2)\|
	&\leq L_{xy}\|y_1-y_2\|, \label{eq:Lxy}\\
	\|\nabla_y g(x,y_1)-\nabla_y g(x,y_2)\|
	&\leq L_{yy}\|y_1-y_2\|. \label{eq:Lyy}
\end{align}
In particular, one may take each of these constants no larger
than $L_g$. Recall that $\sigma\geq0$ denotes a uniform
weak-convexity modulus of $g(x,\cdot)$ and that $\rho\sigma<1$.
For convenience, define
\begin{equation}
	m_\rho\coloneqq\rho^{-1}-\sigma>0,
	\qquad
	\kappa_\rho\coloneqq\frac{1}{1-\rho\sigma}.
	\label{eq:mrho-kapparho}
\end{equation}

\begin{proposition}[Moreau envelope regularity]
	\label{prop:moreau-regularity}
	Suppose that Assumption~\ref{assume:A} holds, and let $m_\rho, \kappa_\rho$ be given by \eqref{eq:mrho-kapparho}. The following statements hold.
	\begin{enumerate}
		\item[(i)]
		For every $w=(x,y)\in C$, the function
		$\psi_\rho(\cdot;x,y)$ is $m_\rho$-strongly convex on $Y$.
		Consequently, $z_\rho(w)$ is well defined and unique.
		
		\item[(ii)]
		The Moreau envelope $g_\rho$ is continuously differentiable on
		$C$, with
		\begin{equation}
			\nabla_x g_\rho(x,y)
			=
			\nabla_x g(x,z_\rho(w)),
			\qquad
			\nabla_y g_\rho(x,y)
			=
			\rho^{-1}\bigl(y-z_\rho(w)\bigr).
			\label{eq:app-grad-moreau}
		\end{equation}
		
		\item[(iii)]
		For any $w_i=(x_i,y_i)\in C$, $i=1,2$,
		\begin{equation}
			\|z_\rho(w_1)-z_\rho(w_2)\|
			\leq
			\kappa_\rho
			\left(
			\rho L_{yx}\|x_1-x_2\|
			+\|y_1-y_2\|
			\right).
			\label{eq:prox-lip}
		\end{equation}
		In particular, $z_\rho$ is Lipschitz continuous on $C$.
		
		\item[(iv)]
		The residual $R{}$ is continuously differentiable, with
		\begin{equation}
			\nabla R{}(w)
			=
			\begin{pmatrix}
				\nabla_x g(x,y)
				-\nabla_x g(x,z_\rho(w))
				\\[1mm]
				\nabla_y g(x,y)
				-\rho^{-1}(y-z_\rho(w))
			\end{pmatrix}.
			\label{eq:app-grad-R}
		\end{equation}
		Moreover, there exist constants $L_R,B_R>0$ such that
		\begin{equation}
			\|\nabla R{}(w_1)-\nabla R{}(w_2)\|
			\leq L_R\|w_1-w_2\|,
			\qquad \text{and} \qquad 
			\|\nabla R{}(w)\|\leq B_R
			\label{eq:LR-BR}
		\end{equation}
		for all $w,w_1,w_2\in C$.
	\end{enumerate}
\end{proposition}

\begin{proof}
	Fix $w=(x,y)\in C$. Since $g(x,\cdot)$ is
	$\sigma$-weakly convex,
	$
	z\mapsto g(x,z)+\frac{\sigma}{2}\|z\|^2
	$
	is convex. Moreover,
	\[
	\frac{1}{2\rho}\|z-y\|^2-\frac{\sigma}{2}\|z\|^2
	=
	\frac{m_\rho}{2}\|z\|^2
	-\rho^{-1}\langle y,z\rangle
	+\frac{1}{2\rho}\|y\|^2.
	\]
	It follows that $\psi_\rho(\cdot;x,y)$ is
	$m_\rho$-strongly convex on $Y$. Since $Y$ is compact and
	$\psi_\rho(\cdot;x,y)$ is continuous, a minimizer exists, while strong
	convexity gives uniqueness; see \cite[Theorem 5.25]{Beck17}. This proves (i).
	
	The differentiability and gradient formulas in (ii) follow from
	Danskin's theorem; see \citet[Theorem 4.13 and Remark 4.14]{bonnans2000perturbation}. 
	
	We next prove (iii). Let $w_i=(x_i,y_i)$ and
	$z_i=z_\rho(w_i)$ for $i=1,2$. The optimality conditions of the
	two proximal subproblems give
	\begin{align*}
		\left\langle
		\nabla_y g(x_1,z_1)+\rho^{-1}(z_1-y_1),
		z_2-z_1
		\right\rangle&\geq0,\\
		\left\langle
		\nabla_y g(x_2,z_2)+\rho^{-1}(z_2-y_2),
		z_1-z_2
		\right\rangle&\geq0.
	\end{align*}
	Adding these inequalities yields
	\begin{align}
		\rho^{-1}
		\langle y_1-y_2,z_1-z_2\rangle \geq \left\langle
		\nabla_y g(x_1,z_1)-\nabla_y g(x_2,z_2),
		z_1-z_2
		\right\rangle
		+\rho^{-1}\|z_1-z_2\|^2		.
		\label{eq:prox-VI}
	\end{align}
	On the other hand, 
	\begin{align*}
		\left\langle
		\nabla_y g(x_1,z_1)-\nabla_y g(x_2,z_2),
		z_1-z_2
		\right\rangle
		& =
		\left\langle
		\nabla_y g(x_1,z_1)-\nabla_y g(x_1,z_2),
		z_1-z_2
		\right\rangle
		\\
		&\qquad+
		\left\langle
		\nabla_y g(x_1,z_2)-\nabla_y g(x_2,z_2),
		z_1-z_2
		\right\rangle
		\\
		&\quad\geq
		-\sigma\|z_1-z_2\|^2
		-
		L_{yx}\|x_1-x_2\|\|z_1-z_2\|.
	\end{align*}
	For the first term above, we used the $\sigma$-weak convexity of $g(x_1,\cdot)$. For the second term, we used Cauchy--Schwarz along with
	\eqref{eq:Lyx}. 
	
	Together with
	\eqref{eq:prox-VI}, we have
	\begin{align*}
		&
		-\sigma\|z_1-z_2\|^2
		-
		L_{yx}\|x_1-x_2\|\|z_1-z_2\|
		+
		\rho^{-1}\|z_1-z_2\|^2\leq
		\rho^{-1}
		\langle y_1-y_2,z_1-z_2\rangle.
	\end{align*}
	Hence,
	\begin{align*}
		(\rho^{-1}-\sigma)\|z_1-z_2\|^2
		&\leq
		L_{yx}\|x_1-x_2\|\|z_1-z_2\|+
		\rho^{-1}
		\langle y_1-y_2,z_1-z_2\rangle
		\\
		&\leq
		L_{yx}\|x_1-x_2\|\|z_1-z_2\|+
		\rho^{-1}\|y_1-y_2\|\|z_1-z_2\|,
	\end{align*}
	where the last inequality again follows from Cauchy--Schwarz. If
	$z_1=z_2$, the claim is immediate. Otherwise, dividing by
	$\|z_1-z_2\|$ yields
	\[
	(\rho^{-1}-\sigma)\|z_1-z_2\|
	\leq
	L_{yx}\|x_1-x_2\|
	+
	\rho^{-1}\|y_1-y_2\|.
	\]
	Therefore,
	\begin{align*}
		\|z_1-z_2\|
		&\leq
		\frac{
			L_{yx}\|x_1-x_2\|
			+\rho^{-1}\|y_1-y_2\|
		}{
			\rho^{-1}-\sigma
		}
		\\
		&=
		\frac{
			\rho L_{yx}\|x_1-x_2\|
			+\|y_1-y_2\|
		}{
			1-\rho\sigma
		}
		\\
		&=
		\kappa_\rho
		\left(
		\rho L_{yx}\|x_1-x_2\|
		+\|y_1-y_2\|
		\right),
	\end{align*}
	which proves~\eqref{eq:prox-lip}.
	
	Formula~\eqref{eq:app-grad-R} follows directly from
	\eqref{eq:app-grad-moreau} and the definition of $R{}$.
	We next verify the Lipschitz continuity of $\nabla R{}$.
	By~\eqref{eq:prox-lip} and Cauchy--Schwarz, there exists
	$K_z>0$ such that
	\[
	\|z_\rho(w_1)-z_\rho(w_2)\|
	\leq
	K_z\|w_1-w_2\|,
	\qquad w_1,w_2\in C.
	\]
	Hence, the map $
	T(w)\coloneqq(x,z_\rho(w))
	$
	is Lipschitz on $C$. Indeed,
	\[
	\|T(w_1)-T(w_2)\|^2
	=
	\|x_1-x_2\|^2
	+
	\|z_\rho(w_1)-z_\rho(w_2)\|^2
	\leq
	(1+K_z^2)\|w_1-w_2\|^2.
	\]
	
	Since $\nabla g$ is $L_g$-Lipschitz on $C$, it follows that
	$
	w\mapsto\nabla g(x,z_\rho(w))
	$
	is Lipschitz on $C$.
	Therefore, each term in~\eqref{eq:app-grad-R} is Lipschitz
	continuous in $w$, and hence there exists $L_R>0$ such that
	\[
	\|\nabla R{}(w_1)-\nabla R{}(w_2)\|
	\leq
	L_R\|w_1-w_2\|,
	\qquad
	w_1,w_2\in C.
	\] Finally, $C$ is compact and $\nabla R{}$ is continuous, so
	$
	B_R\coloneqq\max_{w\in C}\|\nabla R{}(w)\|<\infty.
	$
	This completes the proof of (iv).
\end{proof}

\subsection{Smoothing estimates}
\label{app:smoothing-properties}

We use a standard family of smooth approximations of the plus
function~\citep{beck2012smoothing}. For every $\mu>0$, let
$\phi_\mu$ be convex, continuously differentiable, and
nondecreasing, and suppose that there exist constants
$a_1,a_2,b_\phi,c_\phi\geq0$ such that
\begin{align}
	0& \leq \phi_\mu'(t)\leq1, \label{eq:phi-derivative-bound}\\
	[t]_+-a_1\mu
	&\leq
	\phi_\mu(t)
	\leq
	[t]_++a_2\mu,\label{eq:phi-approx}\\
	|\phi_\mu'(t)-\phi_\mu'(s)|
	&\leq
	\frac{b_\phi}{\mu}|t-s|, \label{eq:phi-derivative-lip}\\
	\phi_{\mu_1}(t)
	&\leq
	\phi_{\mu_0}(t)
	+c_\phi(\mu_0-\mu_1),
	\qquad
	0<\mu_1\leq\mu_0.
	\label{eq:phi-changing-mu}\\
		a_\phi&\coloneqq a_1+a_2.	\label{eq:a-phi}
\end{align}

The next lemma makes precise the $O(1/\mu)$ smoothness property of
$G_\mu$ used in Section~\ref{sec:overview_smoothedmodel}.

\begin{lemma}
	\label{lem:G-lip}
	Under Assumption~\ref{assume:A} and the above conditions on $\phi_\mu$, the function
	$G_\mu=\phi_\mu\circ R{}$ is continuously differentiable and $
	\nabla G_\mu(w)
	=
	\phi_\mu'(R{}(w))\nabla R{}(w).
	$
	Moreover, for every $\mu\in(0,\mu_0]$,
	\begin{equation}
		\|\nabla G_\mu(w)-\nabla G_\mu(\bar w)\|
		\leq
		\frac{L_G}{\mu}\|w-\bar w\|,
		\qquad w,\bar w\in C, \notag
	\end{equation}
	where $L_G\coloneqq b_\phi B_R^2+\mu_0L_R,$ with $B_R$ and $L_R$ given in~\eqref{eq:LR-BR}.
	Furthermore, whenever $0<\mu_1\leq\mu_0$,
	\begin{equation}
		G_{\mu_1}(w)
		\leq
		G_{\mu_0}(w)+c_\phi(\mu_0-\mu_1),
		\qquad w\in C.
		\label{eq:G-changing-mu}
	\end{equation}
\end{lemma}


\begin{proof}
	The gradient formula follows from the chain rule. The remaining estimates
	also follow as special cases of \citet[Lemma~2.3(i),(iii)]{xu2026smoothing};
	we provide a direct proof here for completeness. For
	$w,\bar w\in C$, 
	\begin{align*}
		\|\nabla G_\mu(w)-\nabla G_\mu(\bar w)\|
		&\leq
		|\phi_\mu'(R(w))-\phi_\mu'(R(\bar w))|
		\|\nabla R(w)\|+
		|\phi_\mu'(R(\bar w))|
		\|\nabla R(w)-\nabla R(\bar w)\|.
	\end{align*}
	By Proposition~\ref{prop:moreau-regularity},
	$\|\nabla R\|\leq B_R$ and $\nabla R$ is $L_R$-Lipschitz.
	Hence
	$
	|R(w)-R(\bar w)|
	\leq B_R\|w-\bar w\|.
	$
	Using $0\leq\phi_\mu'\leq1$ by \eqref{eq:phi-derivative-bound} and the
	$(b_\phi/\mu)$-Lipschitz continuity of $\phi_\mu'$ by \eqref{eq:phi-derivative-lip} gives
	\[
	\|\nabla G_\mu(w)-\nabla G_\mu(\bar w)\|
	\leq
	\left(
	\frac{b_\phi B_R^2}{\mu}+L_R
	\right)
	\|w-\bar w\|.
	\]
	Since $\mu\leq\mu_0$, the right-hand side is bounded by
	$(L_G/\mu)\|w-\bar w\|$.
	
	Finally, \eqref{eq:phi-changing-mu} directly gives
	\[
	G_{\mu_1}(w)
	=
	\phi_{\mu_1}(R{}(w))
	\leq
	\phi_{\mu_0}(R{}(w))
	+c_\phi(\mu_0-\mu_1),
	\]
	which proves~\eqref{eq:G-changing-mu}. 
\end{proof}

The following estimates correspond to bounds used in the proof of
\citet[Lemma~3.2(ii)--(iii)]{xu2026smoothing}; we give a direct proof
under our notation for completeness.
\begin{lemma}[Smoothing derivative estimates]
	\label{lem:phi-products}
	Let $a_\phi$ be given by \eqref{eq:a-phi} and let $\theta>1$.
	For every $\mu>0$ and $t\in\mathbb R$, the following hold:
	\begin{enumerate}
		\item[(i)] If $t>\theta a_\phi\mu$, then
		$
		\phi_\mu'(t)\geq\frac{\theta-1}{\theta};
		$
		\item[(ii)] If $t\leq\theta a_\phi\mu$, then
		$
		|\phi_\mu'(t)t|
		\leq\theta a_\phi\mu.
		$
	\end{enumerate}
\end{lemma}

\begin{proof}
	For~(i), convexity of $\phi_\mu$ gives
	\begin{equation}
		t\phi_\mu'(t)
		\geq
		\phi_\mu(t)-\phi_\mu(0).
		\label{eq:convexity-phi_mu}
	\end{equation}
	For $t>0$, the bound
	\eqref{eq:phi-approx} yields
	$
	\phi_\mu(t)\geq t-a_1\mu,
	$ and
	$\phi_\mu(0)\leq a_2\mu.
	$
	Therefore,
	\[
	\phi_\mu'(t)
	\geq
	1-\frac{(a_1+a_2)\mu}{t}
	=
	1-\frac{a_\phi\mu}{t}.
	\]
	If $t>\theta a_\phi\mu$, then
	$
	\frac{a_\phi\mu}{t}<\frac{1}{\theta},
	$
	and hence
	the claim of part~(i). 
	
	For~(ii), first suppose
	$0\leq t\leq\theta a_\phi\mu$. Since
	$0\leq\phi_\mu'(t)\leq1$, we have
	$
	|\phi_\mu'(t)t|
	\leq t
	\leq\theta a_\phi\mu.
	$
	
	Now suppose $t<0$. By~\eqref{eq:convexity-phi_mu},
	$
	(-t)\phi_\mu'(t)
	\leq
	\phi_\mu(0)-\phi_\mu(t).
	$
	Since $[t]_+=0$ for $t<0$, \eqref{eq:phi-approx} gives
	$
	\phi_\mu(0)\leq a_2\mu,
	$ and
	$\phi_\mu(t)\geq-a_1\mu.
	$
	Consequently,
	$
	(-t)\phi_\mu'(t)
	\leq
	(a_1+a_2)\mu
	=
	a_\phi\mu.
	$
	Since $t<0$ and $\phi_\mu'(t)\geq0$,
	\[
	|\phi_\mu'(t)t|
	=
	(-t)\phi_\mu'(t)
	\leq
	a_\phi\mu
	\leq
	\theta a_\phi\mu.
	\]
	This proves~(ii).
\end{proof}

\subsection{Error bounds for inexact Moreau evaluations}
\label{app:oracle-error}

We next give the explicit error bounds underlying
\eqref{eq:oracle-gradient-error}. Define
\begin{equation}
	L_{a,\rho}
	\coloneqq
	\sqrt{L_{xy}^2+\rho^{-2}},
	\quad
	D_\rho(\zeta)
	\coloneqq
	L_{a,\rho}\sqrt{\frac{2\zeta}{m_\rho}}, \quad \text{and} 	\quad E_{\mu,\rho}(\zeta)
	\coloneqq
	D_\rho(\zeta)
	+\frac{b_\phi B_R}{\mu}\zeta, \label{eq:constants-for-error-bounds}
\end{equation}
with $B_R$ given in \eqref{eq:LR-BR}.
\begin{lemma}[Errors from an inexact Moreau evaluation]
	\label{lem:inexact-errors}
	Let $w=(x,y)\in C$, $\zeta\geq0$, and let
	$\widetilde z\in Y$ satisfy
	\begin{equation}
		0\leq
		\psi_\rho(\widetilde z;w)-g_\rho(w)
		\leq\zeta. \label{eq:zeta-accurate-solution}
	\end{equation}
	Define
	\begin{equation}
			\widetilde g_\rho(w)
		\coloneqq
		\psi_\rho(\widetilde z;w),
		\qquad
		\widetilde R{}(w)
		\coloneqq
		g(w)-\widetilde g_\rho(w)-\delta,
		\label{eq:approximate-constraint}
	\end{equation}
	and
	\begin{equation}
		\widetilde\nabla R{}(w)
		\coloneqq
		\begin{pmatrix}
			\nabla_x g(x,y)-\nabla_x g(x,\widetilde z)
			\\[1mm]
			\nabla_y g(x,y)-\rho^{-1}(y-\widetilde z)
		\end{pmatrix}.
		\label{eq:app-inexact-grad-R}
	\end{equation}
	Then the following estimates hold:
	\begin{enumerate}
		\item[(i)] $	\|\widetilde z-z_\rho(w)\|
		\leq
		\sqrt{\frac{2\zeta}{m_\rho}}.$
		
		\item[(ii)] We have 
		\begin{equation}
			\widetilde R{}(w)
			\leq
			R{}(w)
			\leq
			\widetilde R{}(w)+\zeta.
			\label{eq:app-residual-error}
		\end{equation}
		Consequently,
		\begin{equation}
			\phi_\mu(\widetilde R{}(w))
			\leq
			G_\mu(w)
			\leq
			\phi_\mu(\widetilde R{}(w)+\zeta)
			\label{eq:merit-error-bound}
		\end{equation}
		and
		\begin{equation}
			0
			\leq
			\phi_\mu(\widetilde R{}(w)+\zeta)
			-\phi_\mu(\widetilde R{}(w))
			\leq\zeta.
			\label{eq:phi-value-error}
		\end{equation}
		
		\item[(iii)] $		\|
		\widetilde\nabla R{}(w)-\nabla R{}(w)
		\|
		\leq
		D_\rho(\zeta),$ where $D_\rho(\zeta)$ is given in~\eqref{eq:constants-for-error-bounds}.
		
		\item[(iv)]
		With
		$
		\widetilde\nabla G_\mu(w)
		=
		\phi_\mu'(\widetilde R{}(w))
		\widetilde\nabla R{}(w),
		$
		we have $	\|
		\widetilde\nabla G_\mu(w)-\nabla G_\mu(w)
		\|
		\leq
		E_{\mu,\rho}(\zeta),$ where $E_{\mu,\rho}(\zeta)$ is defined in~\eqref{eq:constants-for-error-bounds}. 
	\end{enumerate}
\end{lemma}

\begin{proof}
	Since $\psi_\rho(\cdot;w)$ is $m_\rho$-strongly convex and
	$z_\rho(w)$ is its minimizer, we have from \cite[Theorem 5.25]{Beck17} that 
\begin{equation}
		\psi_\rho(\widetilde z;w)
	\geq
	\psi_\rho(z_\rho(w);w)
	+
	\frac{m_\rho}{2}
	\|\widetilde z-z_\rho(w)\|^2.
	\label{eq:strongconvexity-psi_rho}
\end{equation}
	Combining this with \eqref{eq:zeta-accurate-solution} and noting that $	\psi_\rho(z_\rho(w);w) = g_\rho(w)$
	give~(i).
	
	For part~(ii), let
	$
	e
	\coloneqq
	\widetilde g_\rho(w)-g_\rho(w)
	$. Note that $e\in [0,\zeta]$ due to \eqref{eq:zeta-accurate-solution} and that $e = R(w) - \widetilde{R}{}(w)$ by noting \eqref{eq:approximate-constraint}. Hence, \eqref{eq:app-residual-error}~holds. Since $\phi_\mu$ is
	nondecreasing, this immediately yields~\eqref{eq:merit-error-bound}.
	Moreover, $0\leq\phi_\mu'\leq1$, and therefore $\phi_\mu$ is
	$1$-Lipschitz. Hence~\eqref{eq:phi-value-error} follows.
	
For~(iii), subtract
	\eqref{eq:app-grad-R} from~\eqref{eq:app-inexact-grad-R}:
	\[
	\widetilde\nabla R{}(w)-\nabla R{}(w)
	=
	\begin{pmatrix}
		\nabla_x g(x,z_\rho(w))
		-\nabla_x g(x,\widetilde z)
		\\[1mm]
		\rho^{-1}(\widetilde z - z_\rho(w))
	\end{pmatrix}.
	\]
	Using~\eqref{eq:Lxy} and part~(i),
	\begin{align*}
		\|
		\widetilde\nabla R{}(w)-\nabla R{}(w)
		\|
		&\leq
		\sqrt{L_{xy}^2+\rho^{-2}}\,
		\|\widetilde z-z_\rho(w)\|\leq
		L_{a,\rho}\sqrt{\frac{2\zeta}{m_\rho}}
		=
		D_\rho(\zeta),
	\end{align*}
	which proves~(iii).
	
	Finally, denote
	$\widetilde R=\widetilde R{}(w)$ and $R=R{}(w)$.
	Then
	\begin{align*}
		\|
		\widetilde\nabla G_\mu(w)-\nabla G_\mu(w)
		\| & =
		\|
		\phi_\mu'(\widetilde R)
		\widetilde\nabla R{}(w)
		-
		\phi_\mu'(R)\nabla R{}(w)
		\|\\
		&\quad\leq
		|\phi_\mu'(\widetilde R)|
		\|
		\widetilde\nabla R{}(w)-\nabla R{}(w)
		\|+
		|\phi_\mu'(\widetilde R)-\phi_\mu'(R)|
		\|\nabla R{}(w)\|.
	\end{align*}
	Together with $0\leq\phi_\mu'\leq1$,
	the $(b_\phi/\mu)$-Lipschitz continuity of $\phi_\mu'$,
	\eqref{eq:app-residual-error}, and
	$\|\nabla R{}(w)\|\leq B_R$, we obtain
	\[
	\|
	\widetilde\nabla G_\mu(w)-\nabla G_\mu(w)
	\|
	\leq
	D_\rho(\zeta)
	+\frac{b_\phi B_R}{\mu}\zeta
	=
	E_{\mu,\rho}(\zeta).
	\]
	This proves~(iv).
\end{proof}

\subsection{Properties of the positive Moreau-gap relaxation}
\label{app:moreau-relaxation}

This subsection clarifies the role of the positive relaxation parameter
$\delta>0$ in~\eqref{eq:relaxed-ME} and the relation between
the Moreau-gap constraint and lower-level stationarity. 

\paragraph{Lower-level interpretation of the relaxation.} Define
\begin{equation}
	H_\rho(x,y)
	\coloneqq
	g(x,y)-g_\rho(x,y),
	\qquad
	\mathcal F_\delta
	\coloneqq
	\{(x,y)\in C:H_\rho(x,y)\leq\delta\}.
	\label{eq:moreau-feasible-sets}
\end{equation}
Thus, $\mathcal F_\delta$ is the feasible set of
\eqref{eq:relaxed-ME}, while $\mathcal F_0$ corresponds to
the exact Moreau-gap constraint. We also define
\begin{equation}
	\widetilde S(x)
	\coloneqq
	\left\{
	y\in Y:
	0\in\nabla_y g(x,y)+N_Y(y)
	\right\}
	\label{eq:widetilde-S}
\end{equation}
as the lower-level stationary-point mapping. 

The Moreau envelope reformulation was introduced by
\citet{gao2026moreau}. Under convexity of the lower-level problem,
the exact Moreau-gap constraint is equivalent to lower-level global
optimality~\citep{gao2026moreau}. For weakly convex lower-level
objectives, the exact constraint instead characterizes lower-level
first-order stationarity; see \citet{MEHA}.
The proposition below recalls this equivalence and additionally gives
quantitative bounds relating the Moreau gap to lower-level
near-stationarity.

\begin{proposition}[Stationarity properties of the Moreau gap]
	\label{prop:moreau-gap-stationarity}
	Suppose that $g(x,\cdot)$ is $\sigma$-weakly convex and
	$\rho\sigma<1$. Let $H_\rho$ and $\mathcal F_\delta$ be given by
	\eqref{eq:moreau-feasible-sets}. Then, for every $(x,y)\in C$,
	\begin{equation}
		H_\rho(x,y)
		\geq
		\frac{1-\rho\sigma}{2\rho}
		\|y-z_\rho(x,y)\|^2.
		\label{eq:gap-prox-bound-app}
	\end{equation}
	Moreover,
	\begin{equation}
		\dist\!\left(
		0,
		\nabla_y g(x,z_\rho(x,y))
		+N_Y(z_\rho(x,y))
		\right)
		\leq
		\frac{1}{\rho}\|y-z_\rho(x,y)\|
		\leq
		\sqrt{
			\frac{2H_\rho(x,y)}
			{\rho(1-\rho\sigma)}
		}.
		\label{eq:prox-stationarity-gap-bound}
	\end{equation}
	In particular,
	\begin{equation}
		H_\rho(x,y)=0
		\quad\Longleftrightarrow\quad
		y=z_\rho(x,y)
		\quad\Longleftrightarrow\quad
		0\in\nabla_y g(x,y)+N_Y(y).
		\label{eq:exact-gap-equivalence-app}
	\end{equation}
	Hence,
	\begin{equation}
			\mathcal F_0
		=
		\{(x,y)\in C:y\in\widetilde S(x)\},
		\label{eq:F0}
	\end{equation}
	where $\widetilde{S}$ is given by~\eqref{eq:widetilde-S}.
\end{proposition}

\begin{proof}
	By strong convexity of $\psi_\rho(\cdot; x,y)$, as in \eqref{eq:strongconvexity-psi_rho}, we have 
	\[
	\psi_\rho(y;x,y)
	-
	\psi_\rho(z_\rho(x,y);x,y)
	\geq
	\frac{m_\rho}{2}
	\|y-z_\rho(x,y)\|^2.
	\]
	Since
	$\psi_\rho(y;x,y)=g(x,y)$ and
	$\psi_\rho(z_\rho(x,y);x,y)=g_\rho(x,y)$,
	this gives~\eqref{eq:gap-prox-bound-app}.
	
	Since $z_\rho(x,y)$ minimizes $\psi_\rho(\cdot,x,y) = g(x,\cdot) + \frac{1}{2\rho} \norm{\cdot - y }^2,$ we have 
	\begin{equation}
		\rho^{-1}(y-z_\rho(x,y)) \in
		\nabla_y g(x,z_\rho(x,y)) + N_Y(z_\rho(x,y)), \label{eq:z-rho-optimality}
	\end{equation}
	which proves the first inequality in
	\eqref{eq:prox-stationarity-gap-bound}; the second follows from
	\eqref{eq:gap-prox-bound-app}.
	
	Finally,~\eqref{eq:gap-prox-bound-app} implies that
	$H_\rho(x,y)=0$ only if $y=z_\rho(x,y)$, while the converse is
	immediate since $g_\rho(x,y) = g(x,z_\rho(x,y)) + \frac{1}{2\rho} \norm{z_\rho(x,y) - y}^2$. By the optimality
	condition of the proximal subproblem,
	$y=z_\rho(x,y)$ holds if and only if
	$0\in\nabla_y g(x,y)+N_Y(y)$. Thus, \eqref{eq:exact-gap-equivalence-app} holds. 
\end{proof}

The preceding result also gives a direct interpretation of the
feasibility residual used in our approximate KKT condition.

\begin{corollary}[Interpretation of approximate feasibility]
	\label{cor:kkt-lower-level-stationarity}
	Let $w=(x,y)\in C$ satisfy
	$
	[R(w)]_+\leq\varepsilon_f,
	$ where
	$R(w)=H_\rho(w)-\delta.
	$
	Then
$
	H_\rho(w)\leq\delta+\varepsilon_f,
	$
	and consequently
	\begin{align}
		\|y-z_\rho(x,y)\|
		&\leq
		\sqrt{
			\frac{2\rho(\delta+\varepsilon_f)}
			{1-\rho\sigma}
		}, \notag \\
		\dist\!\left(
		0,
		\nabla_y g(x,z_\rho(x,y))
		+N_Y(z_\rho(x,y))
		\right)
		&\leq
		\sqrt{
			\frac{2(\delta+\varepsilon_f)}
			{\rho(1-\rho\sigma)}
		}.\notag 
	\end{align}
\end{corollary}

\begin{proof}
	Since $R(w)\leq[R(w)]_+$, we obtain
	$
	H_\rho(w)=R(w)+\delta
	\leq\delta+\varepsilon_f.
	$
	The two estimates then follow from
	Proposition~\ref{prop:moreau-gap-stationarity}.
\end{proof}

\paragraph{Why a positive relaxation is needed for KKT analysis.}
The exact Moreau-gap constraint has an inherent constraint-qualification
degeneracy. Recall that, for a feasible point $w$ of a smooth
inequality $R(w)\leq0$ over the closed convex set $C$, the
\emph{no nonzero abnormal multiplier constraint qualification}
(NNAMCQ) requires that
\[
0\notin \nabla R(w)+N_C(w),
\]
whenever $R(w)=0$.
For a single smooth inequality, this condition is equivalent to the
Mangasarian--Fromovitz constraint qualification (MFCQ) relative to
$C$ \citep{jourani1994constraint}. 

\begin{proposition}[Degeneracy of the exact Moreau-gap constraint]
	\label{prop:exact-gap-degeneracy}
	Under Assumption~\ref{assume:A}, NNAMCQ (equivalently, MFCQ), fails
	at every feasible point of the exact Moreau-gap formulation
	$H_\rho\leq0$. On the other hand, for every $\delta>0$, the
	relaxed constraint $H_\rho\leq\delta$ admits a strictly feasible
	point.
\end{proposition}

\begin{proof}
	Since $H_\rho\geq0$ on $C$, every
	$w\in\mathcal F_0$ is a global minimizer of $H_\rho$ over $C$.
	The first-order necessary condition therefore gives
	$
	0\in\nabla H_\rho(w)+N_C(w),
	$
	so NNAMCQ, equivalently MFCQ, fails. For the second claim, fix any $x\in X$ and choose
	$y^\star\in\argmin_{y\in Y}g(x,y)$. Then
	$0\in\nabla_y g(x,y^\star)+N_Y(y^\star)$, and
	Proposition~\ref{prop:moreau-gap-stationarity} gives
	$H_\rho(x,y^\star)=0<\delta$.
\end{proof}

Thus, the positive relaxation removes the unavoidable degeneracy of
the exact formulation in the sense that strict feasibility becomes
possible. Strict feasibility alone, however, does not guarantee the
extended constraint qualification required in our convergence
analysis. Sufficient conditions to guarantee ENNAMCQ are explored next in Appendix~\ref{app:ennamcq-sufficient}. 

\paragraph{Behavior as $\delta\downarrow0$.}
The positive Moreau-gap relaxation is consistent with the exact
constraint as the relaxation level vanishes. Indeed, since
$H_\rho$ is continuous and $C$ is compact, the sets
$\mathcal F_\delta=\{w\in C:H_\rho(w)\leq\delta\}$ are nonempty
and compact, and
\[
\mathcal F_0
\subseteq
\mathcal F_{\delta_1}
\subseteq
\mathcal F_{\delta_2}
\qquad
(0\leq\delta_1\leq\delta_2),
\qquad
\bigcap_{\delta>0}\mathcal F_\delta
=
\mathcal F_0.
\]
Consequently,
$d_{\rm H}(\mathcal F_\delta,\mathcal F_0)\to0$ as
$\delta\downarrow0$, where, for nonempty compact sets $A$ and $B$,
\[
d_{\rm H}(A,B)
\coloneqq
\max\left\{
\sup_{a\in A}\dist(a,B),
\sup_{b\in B}\dist(b,A)
\right\}.
\]
By~\eqref{eq:F0}, $\mathcal F_0$ corresponds to lower-level
stationary points and, in general, need not be the feasible set of
the original bilevel problem. When $\widetilde S(x)=S(x) \coloneqq \argmin_{y\in Y} g(x,y)$ for every
$x\in X$, in particular under lower-level convexity, the above
Hausdorff convergence is to the original bilevel feasible set.
This is consistent with the approximation result of
\citet[Proposition~2]{gao2026moreau}, who show in the convex
lower-level setting that sufficiently small positive relaxations
admit local minimizers arbitrarily close to the bilevel solution set. Nevertheless, this geometric consistency does not imply
uniform algorithmic complexity as $\delta\downarrow0$; see
Example~\ref{ex:eta-delta-deterioration} and the discussion following
it. Accordingly, throughout our complexity analysis, $\delta>0$ is
fixed, and the resulting bounds are not claimed to be uniform as
$\delta\downarrow0$.

\section{Sufficient conditions for ENNAMCQ}
\label{app:ennamcq-sufficient}

This section gives several sufficient conditions under which the
ENNAMCQ assumption used in the main text holds. These
conditions cover lower-level problems arising in many important bilevel optimization problems.

\subsection{Convex lower-level problems without nonconstant affine segments}
For fixed $x\in X$, define the extended-valued lower-level objective
\begin{equation}
	h_x(y):=g(x,y)+\iota_Y(y),
	\label{eq:h_x}
\end{equation}
where $\iota_Y$ denotes the indicator function of $Y$.

We consider the following property:
\begin{equation}
	\label{eq:flat-affine-property}
	\begin{split}
		&\text{if }y_0,y_1\in Y
		\text{ and }
		t\mapsto g\bigl(x,(1-t)y_0+ty_1\bigr)
		\text{ is affine on }[0,1],\\
		&\text{then }
		g\bigl(x,(1-t)y_0+ty_1\bigr)
		\text{ is constant on }[0,1].
	\end{split}
\end{equation}
In other words, affine segments are allowed, but only when the
lower-level objective is constant along them. This condition is
strictly weaker than strict convexity.

\begin{proposition}[Automatic ENNAMCQ without nonconstant affine segments]
	\label{prop:ennamcq-flat-affine}
	Let $\delta>0$, and suppose the standing smoothness assumptions hold.
	Fix $\bar x\in X$. Suppose that $g(\bar x,\cdot)$ is convex on $Y$ and
	satisfies \eqref{eq:flat-affine-property}. Then
	$
	0\notin \nabla R(\bar x,y)+N_C(\bar x,y)
	$
	for every $y\in Y$ such that $R(\bar x,y)\ge0$.
	Consequently, if \eqref{eq:flat-affine-property} holds for every $x\in X$, then ENNAMCQ
	holds on $X\times Y$.
	
	In particular, 	if $g(x,\cdot)+\iota_Y$ is strictly convex for all $x\in X$, then ENNAMCQ holds on $X\times Y$. 
\end{proposition}

\begin{proof}
	Fix $y\in Y$ with $R(\bar x,y)\ge0$ and suppose, to the contrary, that
	\begin{equation}
		\label{eq:flat-affine-ennamcq-failure}
		0\in \nabla R(\bar x,y)+N_C(\bar x,y).
	\end{equation}
	Let $
	z:=z_\rho(\bar x,y),$ and  $p:=\rho^{-1}(y-z).$
	Since $C=X\times Y$, taking the $y$-component of
	\eqref{eq:flat-affine-ennamcq-failure} and using
	$
\nabla _y R(\bar x,y)
	=
	\nabla_y g(\bar x,y)-\rho^{-1}(y-z)
	$ from \eqref{eq:app-grad-R}
	gives
	\begin{equation}
		\label{eq:flat-affine-common-subgradient-y}
		p\in \nabla_y g(\bar x,y)+N_Y(y)
		=\partial h_{\bar x}(y),
	\end{equation}
	where $h_{\bar x}$ is defined as in \eqref{eq:h_x}. 
	On the other hand, the optimality condition for the proximal subproblem
	defining $z$ gives
	$
	0\in
	\nabla_y g(\bar x,z)
	+\rho^{-1}(z-y)
	+N_Y(z),
	$
	and therefore
	\begin{equation}
		\label{eq:flat-affine-common-subgradient-z}
		p\in \nabla_y g(\bar x,z)+N_Y(z)
		=\partial h_{\bar x}(z).
	\end{equation}
	
	Since $h_{\bar x}$ is convex,
	\eqref{eq:flat-affine-common-subgradient-y}--\eqref{eq:flat-affine-common-subgradient-z}
	yield
	\begin{align*}
			h_{\bar x}(z)
		& \ge
		h_{\bar x}(y)+\langle p,z-y\rangle,  \\
			h_{\bar x}(y)
		& \ge
		h_{\bar x}(z)+\langle p,y-z\rangle.
	\end{align*}
	Thus, the two inequalities must both hold with equality. Hence
	\begin{equation}
		\label{eq:flat-affine-subgradient-equality}
		h_{\bar x}(z)
		=
		h_{\bar x}(y)+\langle p,z-y\rangle.
	\end{equation}
	For $t\in[0,1]$, let $y_t=(1-t)y+tz$. By convexity,
	\begin{equation}
		\label{eq:flat-affine-convexity-upper}
		h_{\bar x}(y_t)
		\le
		(1-t)h_{\bar x}(y)+t h_{\bar x}(z).
	\end{equation}
	Using $p\in\partial h_{\bar x}(y)$ and
	\eqref{eq:flat-affine-subgradient-equality},
	\[
	\begin{split}
		h_{\bar x}(y_t)
		&\ge
		h_{\bar x}(y)
		+t\langle p,z-y\rangle=
		(1-t)h_{\bar x}(y)+t h_{\bar x}(z).
	\end{split}
	\]
	Together with \eqref{eq:flat-affine-convexity-upper}, this shows that
	$h_{\bar x}$, and hence
	$g(\bar x,\cdot)$, is affine on the segment $[y,z]$. By assumption \eqref{eq:flat-affine-property}, it must therefore be constant on this
	segment, so
	$
	h_{\bar x}(y)=h_{\bar x}(z).
	$
	Combining this with \eqref{eq:flat-affine-subgradient-equality} gives
	\[
	0
	=
	\langle p,y-z\rangle
	=
	\rho^{-1}\|y-z\|^2,
	\]
	and so $y=z$. Therefore
	$
	H_\rho(\bar x,y)=0
	$
	and consequently
	$
	R(\bar x,y)=-\delta<0,
	$
	contradicting $R(\bar x,y)\ge0$. This completes the proof. 
\end{proof}

We provide an application to a non-strictly-convex lower-level problem used in few-shot learning; see Appendix~\ref{app:experimental-details}.

\begin{proposition}[ENNAMCQ for few-shot learning]
	\label{prop:ennamcq-fewshot}
	Let $T,m,p\in\mathbb N$. For each task $i=1,\ldots,T$, let
	$
	S_i=\{(u_{i,r},q_{i,r})\}_{r=1}^{n_i},
	$
	with $q_{i,r}\in\{1,\ldots,m\}$, be a finite support set.
	For fixed $x\in X$ and a finite dataset $D$, let
	$F_{x,D}(u)\in\mathbb R^p$ denote the feature vector produced for
	$u$ by the shared feature extractor when its batch-dependent
	statistics are computed from $D$. For each $i$ and $r$, set
	$
	F_{i,r}(x)\coloneqq F_{x,S_i}(u_{i,r}).
	$ Let
	$(W_i,b_i)\in\mathbb R^{m\times p}\times\mathbb R^m$ be the
	task-specific linear classifier, and set
	$
	y=((W_i,b_i))_{i=1}^T\in Y,
	$
	where $Y\subseteq
	\prod_{i=1}^T(\mathbb R^{m\times p}\times\mathbb R^m)$ is convex.
	Define
	$
	\ell(s,q)
	:=
	\log\left(\sum_{j=1}^m e^{s_j}\right)-s_q,
	$ for $
	s=(s_1,\ldots,s_m)\in\mathbb R^m,
	$
	and for all $
	y=((W_i,b_i))_{i=1}^T\in Y,
	$ 
	\begin{equation*}
		g(x,y)
		:=
		\frac1T\sum_{i=1}^T\frac1{n_i}
		\sum_{r=1}^{n_i}
\ell\bigl(W_iF_{i,r}(x)+b_i,q_{i,r}\bigr).
	\end{equation*}
	Then, for every fixed $x\in X$, $g(x,\cdot)$ is convex on $Y$
	and satisfies \eqref{eq:flat-affine-property}. Consequently, ENNAMCQ holds for every
	$\delta>0$.
\end{proposition}
\begin{proof}
	Fix $x$ and $y_0,y_1\in Y$. For $\nu=0,1$, write
	$
	y_\nu=\bigl((W_i^{[\nu]},b_i^{[\nu]})\bigr)_{i=1}^T,
	$
	let $y_t=(1-t)y_0+ty_1=\bigl((W_i(t),b_i(t))\bigr)_{i=1}^T$, where
	$
	W_i(t):=(1-t)W_i^{[0]}+tW_i^{[1]},$ and
	$b_i(t):=(1-t)b_i^{[0]}+tb_i^{[1]},
	$
	and write
	$\Delta W_i:=W_i^{[1]}-W_i^{[0]}$,
	$\Delta b_i:=b_i^{[1]}-b_i^{[0]}$. For $(u,q)\in S_i$, set
	$d_{i,r}:=\Delta W_iF_{i,r}(x)+\Delta b_i$, and let
	$\varphi(t):=g(x,y_t)$.
	
	For the softmax vector
	$p_j(s):=e^{s_j}/\sum_{k=1}^m e^{s_k}$, a direct calculation gives
	$
	\nabla^2\ell(s,q)=\operatorname{Diag}(p(s))-p(s)p(s)^\top
	$
	and, for any $d\in\mathbb R^m$,
	\[
	d^\top\nabla^2\ell(s,q)d
	=
	\sum_{j=1}^m p_j(s)
	\left(d_j-\sum_{k=1}^m p_k(s)d_k\right)^2
	\ge0.
	\]
	Since $p_j(s)>0$ for every $j$, equality holds if and only if
	$d=\alpha\mathbf1$ for some $\alpha\in\mathbb R$, where $\mathbf 1$ is a vector of ones. Meanwhile, a direct calculation shows
	\[
	\varphi''(t)
	=
	\frac1T\sum_{i=1}^T\frac1{n_i}
	\sum_{r=1}^{n_i}
	d_{i,r}^\top
	\nabla^2\ell\bigl(W_i(t)F_{i,r}(x)+b_i(t),q_{i,r}\bigr)d_{i,r}
	\ge0,
	\]
	so $g(x,\cdot)$ is convex. If $\varphi$ is affine on $[0,1]$,
	then $\varphi''(t)=0$ on $(0,1)$. From the equivalence noted above, we have for each $(u,q)\in S_i$ that
	$d_{i,r}=\alpha_{i,r}\mathbf1$ for some
	$\alpha_{i,r}\in\mathbb R$. Thus,
	\[
	W_i(t)F_{i,r}(x)+b_i(t)
	=
	W_i^{[0]}F_{i,r}(x)+b_i^{[0]}
	+t\alpha_{i,r}\mathbf1.
	\]
Moreover, for any $\beta\in\mathbb R$,
$
\ell(s+\beta\mathbf1,q)
=
\log(e^\beta\sum_j e^{s_j})-s_q-\beta
=
\ell(s,q).
$
Therefore, taking $\beta=t\alpha_{i,r}$,
\[
\ell\bigl(W_i(t)F_{i,r}(x)+b_i(t),q_{i,r}\bigr)
=
\ell\bigl(W_i^{[0]}F_{i,r}(x)+b_i^{[0]},q_{i,r}\bigr)
\]
for every $t\in[0,1]$, and hence
\[
\varphi(t)
=
\frac1T\sum_{i=1}^T\frac1{n_i}
\sum_{r=1}^{n_i}
\ell\bigl(W_i^{[0]}F_{i,r}(x)+b_i^{[0]},q_{i,r}\bigr)
=
\varphi(0).
\]
Thus \eqref{eq:flat-affine-property} holds, and
Proposition~\ref{prop:ennamcq-flat-affine} yields ENNAMCQ. 
\end{proof}

\subsection{Nonconvex lower level problem with sample reweighting}

We next give a sufficient condition for ENNAMCQ for a specific class of nonconvex smooth loss functions that does not require convexity of
the lower-level objective.

\begin{proposition}[ENNAMCQ for sample-reweighting models]
	\label{prop:ennamcq-reweighting}
	Suppose the standing assumptions hold. Let $X=[-M,M]^n$ and let $Y$ be a compact set. Suppose
	\begin{equation}
		\label{eq:reweighting-model}
		g(x,y)
		=
		\frac1n\sum_{i=1}^n\omega(x_i)\ell_i(y),
		\qquad
		\omega\ge0,\quad \omega'>0\ \text{on }[-M,M],
	\end{equation}
	where each $\ell_i$ is continuous on $Y$. Define
	$
	D_\ell
	:=
	\max_{1\le i\le n}
	\left(
	\max_{y\in Y}\ell_i(y)-\min_{y\in Y}\ell_i(y)
	\right).
	$
	If $
		\delta>\omega(-M)D_\ell,$
	then ENNAMCQ holds on $X\times Y$.
\end{proposition}

\begin{proof}
	Suppose ENNAMCQ fails at $(x,y)$ with $R(x,y)\geq 0$,
	and let
	$
	z:=z_\rho(x,y)
	$
	and
	$
	\Delta_i:=\ell_i(y)-\ell_i(z).
	$
	From \eqref{eq:reweighting-model},
	we have $
	\frac{\partial R}{\partial x_i}(x,y)
	=
	\frac{\omega'(x_i)}{n}\Delta_i.
	$
	Noting~\eqref{eq:app-grad-R}, the $x_i$-component of
	$
	0\in\nabla R(x,y)+N_{X\times Y}(x,y)
	$
	therefore gives
	\[
	0\in
	\frac{\omega'(x_i)}{n}\Delta_i
	+N_{[-M,M]}(x_i).
	\]
	Since $\omega'(x_i)>0$, the above together with the definition of normal cone gives 
	\[
	\Delta_i
	=
	\begin{cases}
		0, & -M<x_i<M,\\
		\ge0, & x_i=-M,\\
		\le0, & x_i=M.
	\end{cases}
	\]
	Thus, using $\omega\ge0$,
	\begin{align*}
		g(x,y)-g(x,z)
		&=
		\frac1n\sum_{i=1}^n\omega(x_i)\Delta_i \\
		& = 	\frac1n\sum_{\{i: x_i = M\text{ or }-M\}}\omega(x_i)\Delta_i \\
		&\leq
		\frac{\omega(-M)}{n}
		\sum_{\{i:x_i=-M\}}\Delta_i\\
		& \leq
		\omega(-M)D_\ell.
	\end{align*}
	Consequently,
	\[
	R(x,y)
	=
	g(x,y)-g(x,z)-\frac1{2\rho}\|z-y\|^2 - \delta 
	\le
	\omega(-M)D_\ell -\delta 
	<
	0,
	\]
	where the last inequality holds from the choice of $\delta$. Since $R(x,y)\geq 0$, we arrived at a contradiction. This completes the proof. 
\end{proof}

A relevant specialization of Proposition~\ref{prop:ennamcq-reweighting} is obtained for sigmoid sample weights and multiclass cross-entropy loss. The data hyper-cleaning problems in our experiments are instances of this setting; see Appendix~\ref{app:experimental-details}.
\begin{corollary}[ENNAMCQ for sigmoid-reweighted classification]
	\label{cor:ennamcq-hypercleaning}
	Let $M,B_y>0$, $n,m,d,d_h\in\mathbb N$, $X=[-M,M]^n$, and
	let $u_i\in\mathbb R^d$ and $q_i\in\{1,\ldots,m\}$,
	$i=1,\ldots,n$. Let
	$
	\omega(t):=(1+e^{-t})^{-1}
	$
	denote the sigmoid sample-weight function. For
	$
	y=(A,a,V,b)\in Y\subseteq
	\mathbb R^{d_h\times d}\times\mathbb R^{d_h}
	\times\mathbb R^{m\times d_h}\times\mathbb R^m,
	$
	define
	$
	h_y(u):=V\operatorname{sigm}(Au+a)+b,
	$
	where $\operatorname{sigm}$ denotes the logistic sigmoid applied
	componentwise, and set
	$
	\ell_i(y):=\ell(h_y(u_i),q_i),
	$
	where $\ell$ is defined in
	Proposition~\ref{prop:ennamcq-fewshot}.  Suppose $Y$ is nonempty, compact and convex,
	and every component of $V$ and $b$ lies in $[-B_y,B_y]$, and let $g$
	be given by \eqref{eq:reweighting-model}. If
	$
	\delta>
	\bigl(2B_y(d_h+1)+\log m\bigr)/(1+e^M),
	$
	then ENNAMCQ holds on $X\times Y$.
\end{corollary}

\begin{proof}
	Since $\omega(t)>0$ and
	$\omega'(t)=\omega(t)(1-\omega(t))>0$,
	Proposition~\ref{prop:ennamcq-reweighting} applies.
	Since $0<\operatorname{sigm}(t)<1$ for every $t\in\mathbb R$,
	for every $j=1,\ldots,m$,
	$
	|[h_y(u)]_j|
	\leq
	\sum_{\ell=1}^{d_h}|V_{j\ell}|
	\,|\operatorname{sigm}(Au+a)_\ell|
	+|b_j|
	\leq
	d_hB_y+B_y=B_y(d_h+1).
	$
	Hence, for every $q\in\{1,\ldots,m\}$,
	\[
	0\le
	\ell(h_y(u),q)
	\le
	2B_y(d_h+1)+\log m,
	\]
	and therefore
	$
	D_\ell\le2B_y(d_h+1)+\log m.
	$
	Moreover, $\omega(-M)=(1+e^M)^{-1}$, so
	\[
	\omega(-M)D_\ell
	\le
	\frac{2B_y(d_h+1)+\log m}{1+e^M}
	<
	\delta.
	\]
	The conclusion follows from
	Proposition~\ref{prop:ennamcq-reweighting}.
\end{proof}

\section{Complete algorithm}
\label{app:complete-algorithm}
We first describe the inexact Moreau evaluation used by IVSP. Since the
proximal objective-gap condition \eqref{eq:oracle-gap} involves the
unknown value $g_\rho(x,y)$, we use projected gradient and terminate it
using the computable residual criterion established below. The resulting
routine is given in Algorithm~\ref{alg:ME-evaluation}.

\begin{lemma}[Residual-based stopping rule for the Moreau evaluation]
	\label{lem:ME-residual-stopping}
	Let
	$
	L_{\psi,\rho}\coloneqq L_{yy}+\rho^{-1},
	$
	where $L_{yy}$ is the Lipschitz constant in~\eqref{eq:Lyy}.
	Given $w=(x,y)\in C$, define
	$
	h(\cdot)\coloneqq\psi_\rho(\cdot;x,y).
	$
	Let $u\in Y$, let
	$0<\eta\leq2/L_{\psi,\rho}$,
	and define
	\[
	z^+
	\coloneqq
	P_Y\!\left(u-\eta\nabla h(u)\right)
	\qquad \text{and} \qquad
	r_\eta(u)
	\coloneqq
	\frac{u-z^+}{\eta}.
	\]
	Then
	$
	0
	\leq
	h(z^+)-g_\rho(x,y)
	\leq
	\frac{1}{2m_\rho}\|r_\eta(u)\|^2,
	$ where $m_\rho$ is given by~\eqref{eq:mrho-kapparho}
	Consequently,
	\begin{equation}
		\|r_\eta(u)\|^2\leq2m_\rho\zeta
		\label{eq:ME-stopping-rule}
	\end{equation}
	guarantees that $z^+$ is $\zeta$-accurate in the sense of
	\eqref{eq:oracle-gap}.
\end{lemma}

\begin{proof}
	Let $z_\rho=z_\rho(w)$. By the optimality condition for the
	projection defining $z^+$ and noting that $z_\rho\in Y$,
	\begin{equation}
		\langle
		\nabla h(u),z^+-z_\rho
		\rangle
		\leq
		\langle
		r_\eta(u),z^+-z_\rho
		\rangle.
		\label{eq:projection-inequality}
	\end{equation}
	By the $L_{\psi,\rho}$-smoothness of $h$, we have from
	\cite[Lemma 5.7]{Beck17} that
	\[
	h(z^+)
	\leq
	h(u)
	+
	\langle\nabla h(u),z^+-u\rangle
	+
	\frac{L_{\psi,\rho}}{2}\|z^+-u\|^2.
	\]
	On the other hand, the $m_\rho$-strong convexity of $h$ (Proposition~\ref{prop:moreau-regularity}(i)) gives
	\[
	h(z_\rho)
	\geq
	h(u)
	+
	\langle\nabla h(u),z_\rho-u\rangle
	+
	\frac{m_\rho}{2}\|z_\rho-u\|^2.
	\]
	Subtracting the latter inequality from the former and using
	\eqref{eq:projection-inequality} yields
	\begin{equation}
		h(z^+)-h(z_\rho)
		\leq
		\langle r_\eta(u),z^+-z_\rho\rangle
		+
		\frac{L_{\psi,\rho}}{2}\|z^+-u\|^2
		-
		\frac{m_\rho}{2}\|z_\rho-u\|^2.
		\label{eq:gap}
	\end{equation}
	Since $z^+=u-\eta r_\eta(u)$, the right-hand side of
	\eqref{eq:gap} equals
	\begin{align}
		&
		\langle r_\eta(u),u-z_\rho\rangle
		-\eta\left(
		1-\frac{\eta L_{\psi,\rho}}{2}
		\right)\|r_\eta(u)\|^2
		-\frac{m_\rho}{2}\|u-z_\rho\|^2.
		\label{eq:rhs-upperbound}
	\end{align}
	Meanwhile, Young's inequality yields
	\begin{equation}
		\langle r_\eta(u),u-z_\rho\rangle
		\leq
		\frac{1}{2m_\rho}\|r_\eta(u)\|^2
		+\frac{m_\rho}{2}\|u-z_\rho\|^2.
		\label{eq:young-complexity}
	\end{equation}
	Combining \eqref{eq:young-complexity}, \eqref{eq:gap}, and
	\eqref{eq:rhs-upperbound}, and using
	$\eta\leq2/L_{\psi,\rho}$, gives
	\[
	0\leq h(z^+)-h(z_\rho)
	\leq
	\frac{1}{2m_\rho}\|r_\eta(u)\|^2.
	\]
	Here, the lower bound follows from the optimality of $z_\rho$, and
	$h(z_\rho)=g_\rho(x,y)$.
\end{proof}

\begin{algorithm}[t]
	\small
	\setlength{\abovedisplayskip}{4pt}
	\setlength{\belowdisplayskip}{4pt}
	\caption{Inexact Moreau evaluation
		$\operatorname{ME}(w,z^{\rm init},\zeta)$}
	\label{alg:ME-evaluation}
	
	\noindent\textbf{Require.}
	A point $w=(x,y)\in C$, an initial point
	$z^{\rm init}\in Y$, and a tolerance $\zeta>0$.
	Set
	$
	\eta=L_{\psi,\rho}^{-1},$ 
	$z^0=z^{\rm init},$
	and 
	$j=0$.

	\smallskip
	\noindent\textbf{Step 1.}
	Compute
	$
	z^{j+1}
	=
	P_Y\!\left(
	z^j-\eta
	\left[
	\nabla_y g(x,z^j)
	+\rho^{-1}(z^j-y)
	\right]
	\right)
	$
	and set
	$
	r^j
	\coloneqq
	\frac{z^j-z^{j+1}}{\eta}.
	$
	
	\smallskip
	\noindent\textbf{Step 2.}
	If
	$
	\|r^j\|^2
	\leq
	2m_\rho\zeta,
	$
	set
	$\widetilde z\coloneqq z^{j+1}$
	and go to \textbf{Step 3}.
	Otherwise, set $j\leftarrow j+1$ and return to
	\textbf{Step 1}.
	
	\smallskip
	\noindent\textbf{Step 3.}
	Form
	$
	\widetilde R(w)
	=
	g(x,y)
	-g(x,\widetilde z)
	-\frac{1}{2\rho}\|\widetilde z-y\|^2
	-\delta.
	$
	Return $\widetilde z$ and $\widetilde R(w)$.
\end{algorithm}

\begin{lemma}[Complexity of the residual stopping rule]
	\label{lem:ME-residual-complexity}
	Let $\{z^j\}$ be generated by
	Algorithm~\ref{alg:ME-evaluation} for a fixed
	$w=(x,y)\in C$, and let
	$
	q_\rho
	\coloneqq
	1-\frac{1}{\kappa_{\psi,\rho}},
	$ where $\kappa_{\psi,\rho}\coloneqq \frac{L_{\psi,\rho}}{m_\rho}\in [1,+\infty]$.
	If $\kappa_{\psi,\rho}>1$, then there exists a constant
	$C_{\rm ME}>0$, independent of $\zeta$, such that
	\begin{equation}
		\|r^j\|^2
		\leq
		C_{\rm ME}q_\rho^{\,j},
		\qquad j\geq0.
		\label{eq:ME-residual-rlinear}
	\end{equation}
	Consequently, the stopping condition
	\eqref{eq:ME-stopping-rule} is satisfied after
	$
	O\!\left(
	\kappa_{\psi,\rho}
	\log\frac{1}{\zeta}
	\right)
	$
	iterations. Moreover, when $w$ ranges over $C$ and the
	initializations belong to $Y$, $C_{\rm ME}$ can be chosen
	uniformly.
\end{lemma}

\begin{proof}
	Let $z_\rho=z_\rho(w)$. Note that $h$ is $m_\rho$-strongly convex with
	$L_{\psi,\rho}$-Lipschitz continuous gradient. Applying
	\citet[Theorem~10.29(b)]{Beck17} to
	$h+\iota_Y$, with the constant stepsize
	$\eta=L_{\psi,\rho}^{-1}$, gives
	\begin{equation}
		\|z^j-z_\rho\|^2
		\leq
		q_\rho^{\,j}
		\|z^{\rm init}-z_\rho\|^2,
		\qquad j\geq0.
		\label{eq:distance-to-solution}
	\end{equation}

	Since
	$
	r^j
	=
	\eta^{-1}(z^j-z^{j+1})
	$
	and $\eta=L_{\psi,\rho}^{-1}$, the triangle inequality and
	\eqref{eq:distance-to-solution} yield
	\begin{align*}
		\|r^j\|
		&=
		L_{\psi,\rho}\|z^j-z^{j+1}\|\\
		&\leq
		L_{\psi,\rho}
		\bigl(
		\|z^j-z_\rho\|
		+
		\|z^{j+1}-z_\rho\|
		\bigr)\\
		&\leq
		L_{\psi,\rho}
		(1+\sqrt{q_\rho})
		q_\rho^{\,j/2}
		\|z^{\rm init}-z_\rho\|\\
		&\leq
		2L_{\psi,\rho}
		q_\rho^{\,j/2}
		\|z^{\rm init}-z_\rho\|.
	\end{align*}
	Hence,
	$
	\|r^j\|^2
	\leq
	4L_{\psi,\rho}^2
	\|z^{\rm init}-z_\rho\|^2
	q_\rho^{\,j},
	$
	which proves \eqref{eq:ME-residual-rlinear} with
	\[
	C_{\rm ME}
	=
	4L_{\psi,\rho}^2
	\|z^{\rm init}-z_\rho\|^2.
	\]
	
	For $\kappa_{\psi,\rho}>1$,
	$
	q_\rho^{\,j}
	=
	\left(
	1-\frac{1}{\kappa_{\psi,\rho}}
	\right)^j
	\leq
	\exp\!\left(
	-\frac{j}{\kappa_{\psi,\rho}}
	\right).
	$
	Therefore,
	$\|r^j\|^2\leq2m_\rho\zeta$ after
	$
	O(\kappa_{\psi,\rho}\log(1/\zeta))
	$
	iterations.
	
	Finally, let $D_Y\coloneqq\operatorname{diam}(Y)$. Since
	$z^{\rm init},z_\rho\in Y$,
	$
	\|z^{\rm init}-z_\rho\|\leq D_Y.
	$
	Thus,
	$
	C_{\rm ME}
	\leq
	4L_{\psi,\rho}^2D_Y^2,
	$
	which is independent of $w$, $z^{\rm init}$, and $\zeta$.
	Consequently, $C_{\rm ME}$ can be chosen uniformly over all
	Moreau evaluations.
\end{proof}

Combining Lemmas~\ref{lem:ME-residual-stopping} and
\ref{lem:ME-residual-complexity},
Algorithm~\ref{alg:ME-evaluation} terminates after
$O(\kappa_{\psi,\rho}\log(1/\zeta))$ projected-gradient iterations
and returns a point satisfying \eqref{eq:oracle-gap}. The complete
implementation of IVSP using Algorithm~\ref{alg:ME-evaluation} is
given in Algorithm~\ref{alg:inexact-penalty-smoothing}.

\begin{algorithm}[t!]
	\small
	\setlength{\abovedisplayskip}{4pt}
	\setlength{\belowdisplayskip}{4pt}
	\caption{Detailed implementation of Algorithm~\ref{alg:main}}
	\label{alg:inexact-penalty-smoothing}
	
	\noindent\textbf{Require.}
	Constants $c_1,c_2,a_\phi,\underline{\vartheta}>0$, $\hat{\vartheta}\in (0,2)$, an initial point $w^0\in C$,
	an initial proximal point $z_{\rm init}^0\in Y$,
	positive nonincreasing sequences $(\mu_k)_{k\ge0}$ and
	$(\zeta_k)_{k\ge0}$, and a strictly decreasing positive
	inverse-penalty grid $(\widehat\gamma_t)_{t\ge0}$.
	Choose either:
	\emph{Strategy I}, a fixed parameter $L>0$; or
	\emph{Strategy II}, constants
	$0<L_{\min}\leq L_{\max}<\infty$ for backtracking.
	Set $t_0=0$ and $k=0$.
	
	\smallskip
	\noindent\textbf{Step 0.}
	Compute the initial $\zeta_0$-accurate Moreau evaluation
	$
	(\widetilde z_0^0,
	\widetilde R_0^0)
	=
	\operatorname{ME}
	(w^0,z_{\rm init}^0,\zeta_0).
	$ Form
	\[
	\widetilde\nabla R^0_0
	=
	\begin{pmatrix}
		\nabla_x g(x^0,y^0)-\nabla_x g(x^0,\widetilde z^0_0)\\
		\nabla_y g(x^0,y^0)-\rho^{-1}(y^0-\widetilde z^0_0)
	\end{pmatrix}.
	\]
	
	\smallskip
	\noindent\textbf{Step 1.}
	Set
	$
	\gamma_k=\widehat\gamma_{t_k},
	$ $E_k^0\coloneqq E_{\mu_k,\rho}(\zeta_k),$ and $
	\widetilde\nabla G_k^0
	\coloneqq
	\phi_{\mu_k}'(\widetilde R_k^0)
	\widetilde\nabla R_k^0.
	$
	Under Strategy I, set $L$ to the prescribed fixed value.
	Under Strategy II, choose
	$L_{k,0}\in[L_{\min},L_{\max}]$ and set $L=L_{k,0}$.
	
	\smallskip
	\noindent\textbf{Step 2.}
	Compute the trial point
	\begin{equation}
		\label{eq:inexact-step}
		w_L^k
		=
		P_C\left(
		w^k-\frac{\mu_k}{L}
		\left[
		\gamma_k\nabla f(w^k)
		+\widetilde\nabla G_k^0
		\right]
		\right).
	\end{equation}
	
	\smallskip
	\noindent\textbf{Step 3.}
	Using $\widetilde z_k^0$ as a warm start, compute a
	$\zeta_{k+1}$-accurate Moreau evaluation at the trial point:
	$
	(\widetilde z_{k,L}^+,
	\widetilde R_{k,L}^+)
	=
	\operatorname{ME}
	(w_L^k,\widetilde z_k^0,\zeta_{k+1}).
	$
	Under Strategy II, define
	\[
	\tau_{k,L}\coloneqq
	\frac{\mu_k}{2\vartheta_{k,L}}(E_k^0)^2+\zeta_k+\zeta_{k+1},
	\]
	where $	\vartheta_{k,L}
	\coloneqq
	\max\left\{
	\min\left\{
	\frac{\mu_kE_k^0}{\|w^k_L-w^k\|}, 
	\hat{\vartheta}L
	\right\},
	\underline{\vartheta}
	\right\}$, where the ratio is interpreted as $+\infty$ when
	$w_L^k=w^k$.
	
	\smallskip
	\noindent\textbf{Step 4.}
	Under Strategy I, set
	$
	L_k\coloneqq L,$
	$w^{k+1}\coloneqq w_L^k,
	$
	and go to \textbf{Step 5}.
	Under Strategy II, if
	\begin{align}
		\label{eq:inexact-LS}
		\gamma_k f(w_L^k)
		+
		\phi_{\mu_k}
		(\widetilde R_{k,L}^++\zeta_{k+1})
		&\leq
		\gamma_k f(w^k)
		+
		\phi_{\mu_k}(\widetilde R_k^0) 
		-\frac{c_1}{2\mu_k}
		\|w_L^k-w^k\|^2
		+\tau_{k,L}
	\end{align}
	does not hold, set $L\leftarrow2L$ and return to
	\textbf{Step 2}.
	Otherwise, set
	$
	L_k\coloneqq L,$
	$w^{k+1}\coloneqq w_L^k,
	$
	and go to \textbf{Step 5}.

	\smallskip 
	\noindent\textbf{Step 5.}
	Set
	$
	\widetilde z_k^+
	\coloneqq
	\widetilde z_{k,L_k}^+,$ and
	$\widetilde R_k^+
	\coloneqq
	\widetilde R_{k,L_k}^+.
	$
	Then form
	\[
	\widetilde\nabla R_k^+
	=
	\begin{pmatrix}
		\nabla_x g(x^{k+1},y^{k+1})
		-\nabla_x g(x^{k+1},\widetilde z_k^+)\\
		\nabla_y g(x^{k+1},y^{k+1})
		-\rho^{-1}(y^{k+1}-\widetilde z_k^+)
	\end{pmatrix}.
	\]
	Set
	$
	d^k\coloneqq w^{k+1}-w^k,$
	and 
	$\omega_{k+1}\coloneqq\frac{\|d^k\|}{\mu_k}.
	$
	
	\smallskip
	\noindent\textbf{Step 6.}
	If
	\begin{equation}
		\label{eq:inexact-PU}
		\widetilde R_k^+>2a_\phi\mu_k
		\qquad\text{and}\qquad
		\|d^k\|\leq c_2\mu_k\gamma_k,
	\end{equation}
	set $t_{k+1}=t_k+1$; otherwise, set $t_{k+1}=t_k$.
	
	\smallskip
	\noindent\textbf{Step 7.}
	Retain the accepted Moreau evaluation as the current evaluation
	for the next outer iteration:
	\[
	\widetilde z_{k+1}^0
	\coloneqq
	\widetilde z_k^+,
	\qquad
	\widetilde R_{k+1}^0
	\coloneqq
	\widetilde R_k^+,
	\qquad
	\widetilde\nabla R_{k+1}^0
	\coloneqq
	\widetilde\nabla R_k^+.
	\]
	Update $k\leftarrow k+1$ and return to \textbf{Step 1}.
\end{algorithm}

\section{Proofs of the theoretical results}
\label{app:basic-convergence}
We recall that our Lyapunov function introduced in Section~\ref{sec:basic-convergence} is given by 
\begin{equation}
	Q_k
	=
	\gamma_k(f(w^k)-f_{\inf})
	+
	G_{\mu_k}(w^k)
	+
	c_\phi\mu_k.
	\label{eq:lyapunov}
\end{equation}
\subsection{Lemma~\ref{lem:descent}: Well-definedness of line search and descent inequality}
\label{app:lem-descent}

The following is the complete statement of
Lemma~\ref{lem:descent}, including the constants used in the analysis.

\medskip
\noindent\textbf{Lemma~\ref{lem:descent} (complete statement).}
\textit{Suppose Assumption~\ref{assume:A} holds. Let
	$\hat{\vartheta}\in(0,2)$ and $\underline{\vartheta}>0$.
	For each iteration $k\geq0$ and each trial parameter $L>0$, define
	$
	d_{k,L}
	\coloneqq
	w_L^k-w^k$,
	$E_k^0
	\coloneqq
	E_{\mu_k,\rho}(\zeta_k),$
	and set $	\vartheta_{k,L}
	\coloneqq
	\max\left\{
	\min\left\{
	\frac{\mu_kE_k^0}{\|d_{k,L}\|},
	\hat{\vartheta}L
	\right\},
	\underline{\vartheta}
	\right\},$
	where the ratio is interpreted as $+\infty$ when $d_{k,L}=0$.
	Define $		\tau_{k,L}
	\coloneqq
	\frac{\mu_k}{2\vartheta_{k,L}}(E_k^0)^2
	+\zeta_k+\zeta_{k+1}.$
	Furthermore, let
	\begin{equation}
		L_{\rm safe}
		\coloneqq
		\max\left\{
		\frac{\underline{\vartheta}}{\hat{\vartheta}},
		\frac{
			c_1+\widehat\gamma_0\mu_0L_f+L_G
		}{
			2-\hat{\vartheta}
		}
		\right\}.
		\notag
	\end{equation}
	Then, for every $k\geq0$, the following hold.}

\begin{enumerate}[before=\itshape]
	\item[(i)]
	Every trial parameter $L\geq L_{\rm safe}$ satisfies~\eqref{eq:inexact-LS}, i.e., 
	\begin{align}
		\gamma_kf(w_L^k)
		+
		\phi_{\mu_k}
		\bigl(\widetilde R_{k,L}^++\zeta_{k+1}\bigr)
		&\leq
		\gamma_kf(w^k)
		+
		\phi_{\mu_k}(\widetilde R_k^0)
		\notag\\
		&\quad
		-\frac{c_1}{2\mu_k}
		\|w_L^k-w^k\|^2
		+\tau_{k,L}.
		\label{eq:armijo-safe-app}
	\end{align}
	Consequently, any fixed $L\geq L_{\rm safe}$ is admissible
	under Strategy I. Under Strategy II, the backtracking procedure
	terminates finitely, and the accepted parameter satisfies
	\begin{equation}
		L_k\leq
		\overline L
		\coloneqq
		\max\{L_{\max},2L_{\rm safe}\}.
		\label{eq:Lbar}
	\end{equation}
	Furthermore, the number of trial evaluations at each iteration
	is bounded by
	\begin{equation}
		B_{\rm ls}
		\coloneqq
		1+
		\max\left\{
		0,
		\left\lceil
		\log_2\frac{L_{\rm safe}}{L_{\min}}
		\right\rceil
		\right\}.
		\label{eq:B-ls}
	\end{equation}
	
	\item[(ii)]
	Let $L_k=L$ under Strategy I, and let $L_k$ denote the accepted
	parameter under Strategy II. Set
	$
	w^{k+1}\coloneqq w_{L_k}^k,$ $d^k\coloneqq w^{k+1}-w^k,$ and $\omega_{k+1}\coloneqq\frac{\|d^k\|}{\mu_k}$. Moreover, let $
	\vartheta_k
	\coloneqq
	\vartheta_{k,L_k}$ and $\tau_k
	\coloneqq
	\tau_{k,L_k}.$
	Then, under either strategy,
	\begin{equation}
		Q_{k+1}
		\leq
		Q_k
		-\frac{c_1}{2}\mu_k\omega_{k+1}^2
		+\tau_k.
		\label{eq:Lyapunov-descent-app}
	\end{equation}
\end{enumerate}

\begin{proof}[Proof of Lemma~\ref{lem:descent}]
	Fix an iteration $k$ and a trial parameter $L>0$, and let
	$
	d\coloneqq d_{k,L}=w_L^k-w^k.
	$
	Using the optimality condition for the projected step
	\eqref{eq:inexact-step} and noting that $w^k\in C$, we obtain
	\begin{equation}
		\gamma_k
		\langle\nabla f(w^k),d\rangle
		+
		\langle\widetilde\nabla G_k^0,d\rangle
		\leq
		-\frac{L}{\mu_k}\|d\|^2.
		\notag
	\end{equation}
	Adding and subtracting $\nabla G_{\mu_k}(w^k)$ yields
	\begin{align}
		&
		\gamma_k
		\langle\nabla f(w^k),d\rangle
		+
		\langle\nabla G_{\mu_k}(w^k),d\rangle\leq
		-\frac{L}{\mu_k}\|d\|^2
		+
		\left\langle
		\nabla G_{\mu_k}(w^k)
		-\widetilde\nabla G_k^0,
		d
		\right\rangle.
		\label{eq:VI-exact-gradient}
	\end{align}
	Meanwhile, by the Cauchy--Schwarz inequality,
	Lemma~\ref{lem:inexact-errors}(iv), and Young's inequality,
	\begin{align}
		\left\langle
		\nabla G_{\mu_k}(w^k)
		-\widetilde\nabla G_k^0,
		d
		\right\rangle
		&\leq
		\left\|
		\nabla G_{\mu_k}(w^k)
		-\widetilde\nabla G_k^0
		\right\|\|d\|
		\notag\\
		&\leq
		E_k^0\|d\|
		\notag\\
		&\leq
		\frac{\vartheta_{k,L}}{2\mu_k}\|d\|^2
		+
		\frac{\mu_k}{2\vartheta_{k,L}}(E_k^0)^2.
		\label{eq:grad-cauchy-young}
	\end{align}
	Since $\nabla f$ is $L_f$-Lipschitz and, by
	Lemma~\ref{lem:G-lip}, $\nabla G_{\mu_k}$ is
	$(L_G/\mu_k)$-Lipschitz, we have
	(see \cite[Lemma~5.7]{Beck17})
	\begin{align*}
		f(w^k+d)
		&\leq
		f(w^k)
		+\langle\nabla f(w^k),d\rangle
		+\frac{L_f}{2}\|d\|^2,\\
		G_{\mu_k}(w^k+d)
		&\leq
		G_{\mu_k}(w^k)
		+\langle\nabla G_{\mu_k}(w^k),d\rangle
		+\frac{L_G}{2\mu_k}\|d\|^2.
	\end{align*}
	Multiplying the first inequality by $\gamma_k$, adding the second,
	and using~\eqref{eq:VI-exact-gradient}--%
	\eqref{eq:grad-cauchy-young} gives
	\begin{align}
		&
		\gamma_kf(w^k+d)+G_{\mu_k}(w^k+d)
		\notag\\
		&\quad\leq
		\gamma_kf(w^k)+G_{\mu_k}(w^k)
		-
		\frac{
			2L-\gamma_k\mu_kL_f-L_G-\vartheta_{k,L}
		}{2\mu_k}
		\|d\|^2
		+\chi_{k,L},
		\label{eq:model-descent-inexact}
	\end{align}
	where
	$
	\chi_{k,L}
	\coloneqq
	\frac{\mu_k}{2\vartheta_{k,L}}(E_k^0)^2.
	$
	
	Now, suppose that $L\geq L_{\rm safe}$. By the definition of $L_{\rm safe}$,
	we have $
	L\geq\frac{\underline{\vartheta}}{\hat{\vartheta}},
	$
	and therefore
	$
	\underline{\vartheta}
	\leq
	\hat{\vartheta}L.
	$
	Since $\min\left\{
	\frac{\mu_kE_k^0}{\|d_{k,L}\|}, \hat{\vartheta}L\right\}$ is also bounded above by $\hat{\vartheta}L$, it follows that $\vartheta_{k,L} \leq \hat{\vartheta}L$. 
	Since
	$\gamma_k\leq\widehat\gamma_0$ and
	$\mu_k\leq\mu_0$, we obtain
	\begin{align*}
		2L-\gamma_k\mu_kL_f-L_G-\vartheta_{k,L}
		&\geq
		(2-\hat{\vartheta})L
		-\widehat\gamma_0\mu_0L_f-L_G\\
		&\geq
		c_1,
	\end{align*}
	where the last inequality follows from the definition of
	$L_{\rm safe}$. Hence, using \eqref{eq:model-descent-inexact}, we have for all $L\geq L_{\rm safe}$ that 
	\begin{equation}
		\gamma_kf(w_L^k)+G_{\mu_k}(w_L^k)
		\leq
		\gamma_kf(w^k)+G_{\mu_k}(w^k)
		-\frac{c_1}{2\mu_k}\|w_L^k-w^k\|^2
		+\chi_{k,L}.
		\label{eq:model-descent-safe}
	\end{equation}
	
	We next relate the exact values in
	\eqref{eq:model-descent-safe} to the quantities used in the
	inexact Armijo test. At the trial point $w_L^k$, we obtain the
	$\zeta_{k+1}$-accurate minimizer of $\psi_\rho (\cdot,w_L^k)$. Applying \eqref{eq:phi-value-error} and \eqref{eq:merit-error-bound} in 	Lemma~\ref{lem:inexact-errors}, respectively, we obtain
	\begin{align}
		&
		\gamma_kf(w_L^k)
		+
		\phi_{\mu_k}
		\bigl(\widetilde R_{k,L}^++\zeta_{k+1}\bigr)\leq
		\gamma_kf(w_L^k)
		+
		G_{\mu_k}(w_L^k)
		+\zeta_{k+1}.
		\label{eq:trial-merit-upper}
	\end{align}
	On the other hand, at the current point $w^k$, the $\zeta_k$-accurate solution gives
	\begin{equation}
		\gamma_kf(w^k)+G_{\mu_k}(w^k)
		\leq
		\gamma_kf(w^k)
		+
		\phi_{\mu_k}(\widetilde R_k^0)
		+\zeta_k,
		\label{eq:current-merit-upper}
	\end{equation}
	by using \eqref{eq:merit-error-bound} and \eqref{eq:phi-value-error}, respectively. 
	Combining
	\eqref{eq:model-descent-safe}, \eqref{eq:trial-merit-upper}, and
	\eqref{eq:current-merit-upper}, and noting that
	$
	\tau_{k,L}
	=
	\chi_{k,L}+\zeta_k+\zeta_{k+1},
	$
	yields~\eqref{eq:armijo-safe-app}.
	Thus every $L\geq L_{\rm safe}$ satisfies the inexact Armijo
	condition.
	
	Under Strategy I, the prescribed fixed parameter is therefore
	valid whenever $L\geq L_{\rm safe}$. Under Strategy II, $L$ is
	doubled until the inexact Armijo condition holds, so the
	backtracking procedure terminates finitely. Since
	$L_{k,0}\in[L_{\min},L_{\max}]$, every accepted parameter satisfies \eqref{eq:Lbar}. 
	Moreover, the number of trial evaluations at any iteration is at
	most $	B_{\rm ls}$ as given in~\eqref{eq:B-ls}. This completes the proof of (i).

	We next prove (ii). Note from \eqref{eq:merit-error-bound} that 
	\[
	\phi_{\mu_k}(\widetilde R_k^0)
	\leq
	G_{\mu_k}(w^k) \qquad \text{and} \qquad G_{\mu_k}(w^{k+1})
	\leq
	\phi_{\mu_k}
	\bigl(\widetilde R_k^++\zeta_{k+1}\bigr).
	\]
	Together with \eqref{eq:armijo-safe-app}, we obtain\footnote{We note that under Strategy I,~\eqref{eq:accepted-merit-descent} follows
		directly from~\eqref{eq:model-descent-safe}; in fact, in this
		case the stronger estimate with
		$
		\chi_k
		\coloneqq
		\frac{\mu_k}{2\vartheta_k}(E_k^0)^2
		$
		in place of $\tau_k$ holds.}
	\begin{equation}
		\gamma_kf(w^{k+1})+G_{\mu_k}(w^{k+1})
		\leq
		\gamma_kf(w^k)+G_{\mu_k}(w^k)
		-\frac{c_1}{2\mu_k}\|d^k\|^2
		+\tau_k.
		\label{eq:accepted-merit-descent}
	\end{equation}
	Meanwhile, since $\gamma_{k+1}\leq\gamma_k$ and
	$f(w^{k+1})-f_{\inf}\geq0$,
	\begin{equation}
		\gamma_{k+1}
		\bigl(f(w^{k+1})-f_{\inf}\bigr)
		\leq
		\gamma_k
		\bigl(f(w^{k+1})-f_{\inf}\bigr).
		\label{eq:f-part}
	\end{equation}
	Moreover, inequality~\eqref{eq:G-changing-mu} in
	Lemma~\ref{lem:G-lip} gives
	\begin{equation}
		G_{\mu_{k+1}}(w^{k+1})
		+
		c_\phi\mu_{k+1}
		\leq
		G_{\mu_k}(w^{k+1})
		+
		c_\phi\mu_k.
		\label{eq:G-mu-part}
	\end{equation}
	Combining \eqref{eq:f-part} and \eqref{eq:G-mu-part} with
	\eqref{eq:accepted-merit-descent}, and using
	$\|d^k\|=\mu_k\omega_{k+1}$ and the definition of $Q_k$ in \eqref{eq:lyapunov}, we obtain~\eqref{eq:Lyapunov-descent-app}.
%
\end{proof}

\subsection{Lemma~\ref{lem:penalty-stabilization}: Finite penalty stabilization}
\label{app:penalty-stabilization}

To prove Lemma~\ref{lem:penalty-stabilization}, we need the following two lemmas. 
\begin{lemma}[Uniform ENNAMCQ margin]
	\label{lem:ennamcq-margin}
	Suppose that $C$ is nonempty, compact, and convex, and that
	$R{}$ is continuously differentiable on $C$. If ENNAMCQ holds
	for~\eqref{eq:relaxed-ME} and
	$
	\{w\in C:R{}(w)\geq0\}\neq\varnothing,
	$
	then
	\begin{equation}
		\eta_\delta
		\coloneqq
		\min_{\substack{w\in C\\R{}(w)\geq0}}
		\dist\!\left(
		0,\nabla R{}(w)+N_C(w)
		\right)
		>0.
		\label{eq:eta-delta-app}
	\end{equation}
\end{lemma}

\begin{proof}
	Let
	$
	\eta
	\coloneqq
	\inf_{\substack{w\in C\\R(w)\geq0}}
	\dist\!\left(
	0,\nabla R(w)+N_C(w)
	\right).
	$
	Since $R$ is continuous, the set
	$
	\{w\in C:R(w)\geq0\}
	$
	is compact. Choose a sequence $\{w^j\}$ in this set and
	$n^j\in N_C(w^j)$ such that
	$
	\|\nabla R(w^j)+n^j\|\to\eta.
	$
	By compactness,  we may assume without loss of generality that
	$
	w^j\to\bar w,
	$
	where $\bar w\in C$ and $R(\bar w)\geq0$. It follows that 
	\begin{equation}
		 \eta \leq 	\dist\!\left(
		0,\nabla R(\bar w)+N_C(\bar w)
		\right)
		\label{eq:eta-leq-dist}
	\end{equation}
	
	Meanwhile, since $\nabla R$ is continuous and $C$ is compact,
	$\{\nabla R(w^j)\}$ is bounded. Moreover,
	since $\{\nabla R(w^j)+n^j\}$ is bounded,  $\{n^j\}$ is also bounded. Therefore, there exists a
	subsequence $\{j_r\}$ and some $\bar n$ such that
	$
	n^{j_r}\to\bar n.
	$
	Since $C$ is closed and convex, the normal cone mapping
	$N_C$ has closed graph by
	\citet[Proposition~6.6]{rockafellar1998variational}, and therefore
	$
	\bar n\in N_C(\bar w).
	$ Consequently, 
	\[
	\dist\!\left(
	0,\nabla R(\bar w)+N_C(\bar w)
	\right)
	\leq
	\|\nabla R(\bar w)+\bar n\|
	=
	\lim_{r\to\infty}
	\|\nabla R(w^{j_r})+n^{j_r}\| =
	\eta,
	\]
	where the first equality holds by continuity of $\nabla R$. Together with \eqref{eq:eta-leq-dist}, we have 
	\[
	\eta
	=
	\dist\!\left(
	0,\nabla R(\bar w)+N_C(\bar w)
	\right).
	\]
	Thus the infimum is attained at $\bar w$ so that $\eta_\delta = \eta$. Finally, ENNAMCQ at
	$\bar w$ implies that $\dist\!\left(
	0,\nabla R(\bar w)+N_C(\bar w)
	\right)>0$, and so $\eta_\delta>0$. 
\end{proof}

\begin{lemma}[Stationarity bound at a penalty update]
	\label{lem:update-stationarity}
	Let $\overline L>0$ be given by
	\begin{equation}
		\overline L
		=
		\begin{cases}
			L, & \text{under Strategy I},\\[1mm]
			\max\{L_{\max},2L_{\rm safe}\},
			& \text{under Strategy II},
		\end{cases}
		\label{eq:overline-L}
	\end{equation}
	with $L\geq L_{\rm safe}$. 
	Define
	$
	E_k^0
	\coloneqq
	E_{\mu_k,\rho}(\zeta_k)$,
	$E_k^+
	\coloneqq
	E_{\mu_k,\rho}(\zeta_{k+1}),$
	$
	D_k^+
	\coloneqq
	D_\rho(\zeta_{k+1}),
	$
	and
	$
	A_0
	\coloneqq
	M_f+c_2(\overline L+L_G),
	$
	where $M_f \coloneqq \max_{w\in C} \norm{\nabla f(w)}$.
	If a penalty update is triggered at iteration $k$, then
	\begin{align}
		R{}(w^{k+1})& >0, \qquad \text{and}
		\label{eq:update-exact-positive} \\
		\dist\!\left(
		0,\nabla R{}(w^{k+1})+N_C(w^{k+1})
		\right)
		& \leq
		2A_0\gamma_k
		+2(E_k^0+E_k^+)
		+D_k^+.
		\label{eq:update-stationarity-bound}
	\end{align}
\end{lemma}

\begin{proof}
	Since a penalty update is triggered, we have from the first condition in \eqref{eq:inexact-PU} that
	$
	\widetilde R_k^+>2a_\phi\mu_k>0.
	$
	Thus, by~\eqref{eq:app-residual-error}, we have
	$
	R{}(w^{k+1})
	\geq
	\widetilde R_k^+>0.
	$
	This proves~\eqref{eq:update-exact-positive}.
	
	
	From the optimality condition of~\eqref{eq:inexact-step} with $L=L_k$, there exists
	$n^{k+1}\in N_C(w^{k+1})$ such that
	\begin{equation}
		0
		=
		\gamma_k\nabla f(w^k)
		+\widetilde\nabla G_k^0
		+\frac{L_k}{\mu_k}d^k
		+n^{k+1}.
		\label{eq:accepted-optimality-penalty}
	\end{equation}
	At the accepted point, let
	$
	\widetilde s_k^+
	\coloneqq
	\phi_{\mu_k}'(\widetilde R_k^+),
$ and 
	$\widetilde\nabla G_k^+
	\coloneqq
	\widetilde s_k^+\widetilde\nabla R_k^+.$
	Then~\eqref{eq:accepted-optimality-penalty} gives
	\begin{align}
		\widetilde s_k^+\widetilde\nabla R_k^+
		+n^{k+1}
		&=
		-\gamma_k\nabla f(w^k)
		-\frac{L_k}{\mu_k}d^k
		+\widetilde\nabla G_k^+
		-\widetilde\nabla G_k^0.
		\label{eq:s-nablaR-plus-n}
	\end{align}
Meanwhile, by the triangle inequality,
\begin{align}
	\|
	\widetilde\nabla G_k^+
	-\widetilde\nabla G_k^0
	\|
	&\leq
	\|
	\widetilde\nabla G_k^+
	-\nabla G_{\mu_k}(w^{k+1})
	\|+
	\|
	\nabla G_{\mu_k}(w^{k+1})
	-\nabla G_{\mu_k}(w^k)
	\|
	\notag\\
	&\quad+
	\|
	\nabla G_{\mu_k}(w^k)
	-\widetilde\nabla G_k^0
	\|.
	\notag 
\end{align}
Using Lemma~\ref{lem:inexact-errors}(iv) to bound the first and third terms, and using Lemma~\ref{lem:G-lip} to bound the second term, we obtain 
\[
\|
\widetilde\nabla G_k^+
-\widetilde\nabla G_k^0
\|
\leq
L_G\omega_{k+1}
+E_k^0+E_k^+.
\]
	Hence, together with~\eqref{eq:s-nablaR-plus-n} and noting that $L_k\leq \overline{L}$, 
	\[
	\|
	\widetilde s_k^+\widetilde\nabla R_k^+
	+n^{k+1}
	\|
	\leq
	M_f\gamma_k
	+(\overline L+L_G)\omega_{k+1}
	+E_k^0+E_k^+.
	\]
	Since $N_C(w^{k+1})$ is a cone and
	$\widetilde s_k^+\geq1/2$ by  	Lemma~\ref{lem:phi-products}(i) with $\theta =2$,
	\begin{align}
		\dist\!\left(
		0,\widetilde\nabla R_k^+ +N_C(w^{k+1})
		\right)
		&\leq
		2M_f\gamma_k
		+2(\overline L+L_G)\omega_{k+1}
		+2(E_k^0+E_k^+).
		\notag 
	\end{align}
	The second condition in~\eqref{eq:inexact-PU} gives
	$\omega_{k+1}\leq c_2\gamma_k$, and therefore
	\begin{equation}
		\dist\!\left(
		0,\widetilde\nabla R_k^+ +N_C(w^{k+1})
		\right)
		\leq
		2A_0\gamma_k+2(E_k^0+E_k^+).
		\label{eq:dist-0-newsubgradient}
	\end{equation}
Finally, Lemma~\ref{lem:inexact-errors}(iii) gives
$
\|
\widetilde\nabla R_k^+
-\nabla R(w^{k+1})
\|
\leq D_k^+.
$
Moreover, for every $n\in N_C(w^{k+1})$, the triangle inequality
gives
\begin{align*}
	\|
	\nabla R(w^{k+1})+n
	\|
	&\leq
	\|
	\nabla R(w^{k+1})-\widetilde\nabla R_k^+
	\|
	+
	\|
	\widetilde\nabla R_k^+ +n
	\|.
\end{align*}
Taking the infimum over $n\in N_C(w^{k+1})$  and using \eqref{eq:dist-0-newsubgradient}, we have
\begin{align*}
	&\dist\!\left(
	0,\nabla R(w^{k+1})+N_C(w^{k+1})
	\right)
	\\
	&\quad\leq
	\|
	\nabla R(w^{k+1})-\widetilde\nabla R_k^+
	\|
	+
	\dist\!\left(
	0,\widetilde\nabla R_k^+ +N_C(w^{k+1})
	\right)
	\\
	&\quad\leq
	2A_0\gamma_k
	+2(E_k^0+E_k^+)
	+D_k^+,
\end{align*}
which proves~\eqref{eq:update-stationarity-bound}.
\end{proof}

We now give the complete statement of
Lemma~\ref{lem:penalty-stabilization}, including the explicit
constant appearing in the lower bound on the inverse penalty
parameter.

\medskip
\noindent\textbf{Lemma~\ref{lem:penalty-stabilization}
	(complete statement).}
\begin{itshape}
	Suppose ENNAMCQ holds for~\eqref{eq:relaxed-ME},
	$\zeta_k=o(\mu_k)$, and there exists $M_\gamma>1$ such that
	\begin{equation}
		\widehat\gamma_t
		\leq
		M_\gamma\widehat\gamma_{t+1}
		\qquad
		\forall t\geq0.
		\label{eq:grid-ratio}
	\end{equation}
	Let $\overline L$ be given by \eqref{eq:overline-L}, and define
	$
	A_0
	\coloneqq
	M_f+c_2(\overline L+L_G).
	$ 
	\begin{enumerate}
		\item If
		$
		\{w\in C:R(w)\geq0\}=\varnothing,
		$
		then no penalty update occurs, and
		$
		\gamma_k=\underline{\gamma} \coloneqq \widehat\gamma_0$ for all $k\geq0.
		$
		
		\item 	If
		$
		\{w\in C:R(w)\geq0\}\neq\varnothing,
		$ let $\eta_\delta>0$ be defined by~\eqref{eq:eta-delta-app}.
		Then there exists $k_0\geq0$ such that
		\[
		2\!\left[
		E_{\mu_k,\rho}(\zeta_k)
		+
		E_{\mu_k,\rho}(\zeta_{k+1})
		\right]
		+
		D_\rho(\zeta_{k+1})
		\leq
		\frac{\eta_\delta}{2}
		\qquad
		\forall k\geq k_0.
		\]
		Moreover,
		\begin{equation}
			\gamma_k
			\geq
			\underline\gamma
			\coloneqq
			\min\left\{
			\widehat\gamma_{k_0},
			\frac{\eta_\delta}{4M_\gamma A_0}
			\right\}
			>0
			\qquad
			\forall k\geq0.\notag
		\end{equation}
		
	\end{enumerate}

	In either case, the penalty parameter is updated only finitely many
	times, and $\gamma_k$ is eventually constant.
\end{itshape}
	
\begin{proof}[Proof of Lemma~\ref{lem:penalty-stabilization}]
	If
	$
	\{w\in C:R(w)\geq0\}=\varnothing,
	$
	then $R(w)<0$ for every $w\in C$. Since
	$\widetilde R(w)\leq R(w)$, the first condition in
	\eqref{eq:inexact-PU} can never hold. Hence no penalty update is
	ever triggered, so $t_k=0$ and
	$\gamma_k=\widehat\gamma_0$ for every $k$. Thus the conclusion
	holds with $\underline\gamma=\widehat\gamma_0$.
	
	Suppose now that
	$
	\{w\in C:R(w)\geq0\}\neq\varnothing.
	$
By Lemma~\ref{lem:ennamcq-margin}, $\eta_\delta>0$ is well defined.
		Since $\zeta_k=o(\mu_k)$ and $\mu_k\leq\mu_0$, we have
	$\zeta_k/\mu_k \to 0$ and $\zeta_k\to0$. Together with~\eqref{eq:constants-for-error-bounds},
	we have 
	\[
	D_k^+
	=
	L_{a,\rho}
	\sqrt{\frac{2\zeta_{k+1}}{m_\rho}}
	\to0 \qquad \text{and} \qquad 	E_k^0
	=
	L_{a,\rho}
	\sqrt{\frac{2\zeta_k}{m_\rho}}
	+
	b_\phi B_R\frac{\zeta_k}{\mu_k}
	\to0.
	\]
Moreover, since $\mu_{k+1}\leq\mu_k$,
	$
	\zeta_{k+1}/\mu_k
	\leq
	\zeta_{k+1}/\mu_{k+1}\to0,
	$
	and therefore
	\[
	E_k^+
	=
	L_{a,\rho}
	\sqrt{\frac{2\zeta_{k+1}}{m_\rho}}
	+
	b_\phi B_R\frac{\zeta_{k+1}}{\mu_k}
	\to0.
	\]
	Hence, there exists $k_0$ such that
	\begin{equation}
		2(E_k^0+E_k^+)+D_k^+
		\leq
		\frac{\eta_\delta}{2}
		\qquad
		\forall k\geq k_0.
		\label{eq:eventual-oracle-small}
	\end{equation}
	
	Let $k\geq k_0$ be an iteration at which a penalty update is
	triggered. By Lemma~\ref{lem:update-stationarity},
	$R(w^{k+1})>0$, and hence the definition of $\eta_\delta$ gives
	\[
	\eta_\delta
	\leq
	\dist\!\left(
	0,\nabla R(w^{k+1})+N_C(w^{k+1})
	\right).
	\]
	Combining this with
	\eqref{eq:update-stationarity-bound} and
	\eqref{eq:eventual-oracle-small}, we obtain
	$
	\eta_\delta
	\leq
	2A_0\gamma_k+\eta_\delta/2.
	$
	Thus,
$\gamma_k
		\geq
		\frac{\eta_\delta}{4A_0}.
		$
	Since an update is triggered,
	$t_{k+1}=t_k+1$. By~\eqref{eq:grid-ratio},
	$
	\gamma_k
	=
	\widehat\gamma_{t_k}
	\leq
	M_\gamma\widehat\gamma_{t_k+1}
	=
	M_\gamma\gamma_{k+1}.
	$
	Hence
	\begin{equation}
		\gamma_{k+1}
		\geq
		\frac{\eta_\delta}{4M_\gamma A_0}.
		\label{eq:gamma-after-update}
	\end{equation}
	Since $t_k\leq k$, we have for $k\leq k_0$ that
	$
	\gamma_k
	=
	\widehat\gamma_{t_k}
	\geq
	\widehat\gamma_{k_0}
	\geq
	\underline\gamma.
	$
	For $k\geq k_0$, $\gamma_k$ changes only when a penalty update is
	triggered. If no update occurs at iteration $k$, then
	$\gamma_{k+1}=\gamma_k$. If an update is triggered, then
	\eqref{eq:gamma-after-update} gives
	$
	\gamma_{k+1}
	\geq
	\eta_\delta/(4M_\gamma A_0)
	\geq
	\underline\gamma.
	$
	Therefore,
	$
	\gamma_k\geq\underline\gamma
	$
	for all $k\geq0$. This proves~(ii).
	
	Finally, since $\widehat\gamma_t\downarrow0$, there are only
	finitely many indices $t$ such that
	$\widehat\gamma_t\geq\underline\gamma$.
	Since every penalty update increases $t_k$ by one, only finitely
	many penalty updates can occur. Hence $t_k$, and therefore
	$\gamma_k$, is eventually constant.
\end{proof}
The dependence of Lemma~\ref{lem:penalty-stabilization} on the
relaxation level is worth noting. Although ENNAMCQ may hold for every
fixed $\delta>0$, the margin $\eta_\delta$ need not remain bounded
away from zero as $\delta\downarrow0$.

\begin{example}[Deterioration of the ENNAMCQ margin]
	\label{ex:eta-delta-deterioration}
	Let $X$ be any nonempty compact convex set,
	$Y=[-1,1]$, and
	$
	g(x,y)=\frac12y^2.
	$
	Then $g$ is independent of $x$, and for every $\rho>0$,
	\[
	z_\rho(y)=\frac{y}{1+\rho},
	\qquad
	H_\rho(y)
	=
	\frac{\rho}{2(1+\rho)}y^2.
	\]
	Thus, for
	$0<\delta<\rho/(2(1+\rho))$, the region
	$R(y)=H_\rho(y)-\delta\geq0$ is characterized by
	$
	|y|
	\geq
	\sqrt{\frac{2(1+\rho)\delta}{\rho}}.
	$
	Moreover,
	$
	\nabla H_\rho(y)=\frac{\rho}{1+\rho}y.
	$
	Since the two points satisfying
	$
	|y|
	=
	\sqrt{\frac{2(1+\rho)\delta}{\rho}}
	$
	lie in the interior of $Y$ due to the choice of $\delta$, their normal cones are $\{0\}$.
	Consequently, the minimum in
	Lemma~\ref{lem:ennamcq-margin} is attained at these points, and
	$
	\eta_\delta
	=
	\sqrt{\frac{2\rho}{1+\rho}}\sqrt{\delta}.
	$
	Hence
	$
	\eta_\delta=\Theta(\sqrt{\delta})\to0
	$
	as $\delta\downarrow0$.
\end{example}

Since the lower-level objective in
Example~\ref{ex:eta-delta-deterioration} is strictly convex,
ENNAMCQ holds for every fixed $\delta>0$ by
Proposition~\ref{prop:ennamcq-flat-affine}. Nevertheless, the corresponding $\eta_\delta$ may deteriorate as $\delta\downarrow0$. Consequently, the
lower bound $\underline\gamma$ in
Lemma~\ref{lem:penalty-stabilization} may approach zero, and the
stabilization index and constants entering the subsequent KKT
estimates may also deteriorate. Accordingly, the complexity results
below are stated for a fixed positive relaxation level and are not
uniform in $\delta$ as $\delta\downarrow0$.

\subsection{Theorem~\ref{thm:complexity} and
	Corollary~\ref{cor:overall-complexity}: Complexity guarantees} 
\label{app:complexity-proof}
We first establish the stationarity and feasibility estimates used
in the complexity analysis.

\begin{lemma}[Stationarity estimate]
	\label{lem:stationarity-estimate}
	Let
	$
	E_k^0\coloneqq E_{\mu_k,\rho}(\zeta_k)$ and $E_k^+\coloneqq E_{\mu_k,\rho}(\zeta_{k+1}),
	$
	and define
	\begin{equation}
			\widetilde\lambda_{k+1}
		\coloneqq
		\gamma_k^{-1}
		\phi_{\mu_k}'(\widetilde R_k^+).
		\label{eq:multiplier}
	\end{equation}
	Then, for every $k\geq0$,
	\begin{equation}
		\Phi_\rho^{\rm s}
		(w^{k+1},\widetilde\lambda_{k+1})
		\leq
		\gamma_k^{-1}
		\left(
		M_1\omega_{k+1}+E_k^0+E_k^+
		\right),
		\label{eq:stationarity-estimate}
	\end{equation}
	where $\overline L$ is given by \eqref{eq:overline-L},
	$
	M_1
	\coloneqq
	\widehat\gamma_0\mu_0L_f+\overline L + L_G,
	$ and $	\Phi_\rho^{\rm s}$ is given in Definition~\ref{def:approx-KKT}.
\end{lemma}

\begin{proof}
	From the proof of Lemma~\ref{lem:penalty-stabilization} in Appendix~\ref{app:penalty-stabilization}, there exists $n^{k+1}\in N_C(w^{k+1})$ such that \eqref{eq:accepted-optimality-penalty} holds. Then using the definition of $\widetilde\lambda_{k+1}$ in \eqref{eq:multiplier}, 
\begin{align}
	&
	\gamma_k\nabla f(w^{k+1})
	+
	\gamma_k\widetilde\lambda_{k+1}
	\nabla R(w^{k+1})
	+n^{k+1}
	\notag \\
	&\quad=
	\gamma_k
	\bigl(\nabla f(w^{k+1})-\nabla f(w^k)\bigr)
	-\frac{L_k}{\mu_k}d^k +
	\left[
	\phi_{\mu_k}'(\widetilde R_k^+)
	\nabla R(w^{k+1})
	-\nabla G_{\mu_k}(w^{k+1})
	\right] \notag 
	\\
	&\qquad+
	\left[
	\nabla G_{\mu_k}(w^{k+1})
	-\nabla G_{\mu_k}(w^k)
	\right]
	+
	\left[
	\nabla G_{\mu_k}(w^k)
	-\widetilde\nabla G_k^0
	\right].\label{eq:penalty-stationarity-condition}
\end{align}
For the first bracketed term, since
$
\nabla G_{\mu_k}(w^{k+1})
=
\phi_{\mu_k}'(R(w^{k+1}))
\nabla R(w^{k+1}),
$
we have
\begin{align}
	\left\|
	\phi_{\mu_k}'(\widetilde R_k^+)
	\nabla R(w^{k+1})
	-\nabla G_{\mu_k}(w^{k+1})
	\right\| & =
	\left|
	\phi_{\mu_k}'(\widetilde R_k^+)
	-\phi_{\mu_k}'(R(w^{k+1}))
	\right|
	\|\nabla R(w^{k+1})\|
	\notag \\
	&\quad\leq
	\frac{b_\phi B_R}{\mu_k}
	\left|
	\widetilde R_k^+-R(w^{k+1})
	\right|
	\notag \\
	&\quad\leq
	\frac{b_\phi B_R}{\mu_k}\zeta_{k+1}
	\leq
	E_k^+,\label{eq:first-bracket-upperbound}
\end{align}
where we used~\eqref{eq:phi-derivative-bound}, \eqref{eq:LR-BR}, and
\eqref{eq:app-residual-error}.

Taking norms in \eqref{eq:penalty-stationarity-condition}, applying the triangle
inequality, and using \eqref{eq:first-bracket-upperbound},
\begin{align*}
	\left\|
	\gamma_k\nabla f(w^{k+1})
	+
	\gamma_k\widetilde\lambda_{k+1}
	\nabla R(w^{k+1})
	+n^{k+1}
	\right\|&\leq
	\gamma_kL_f\|d^k\|
	+\frac{L_k}{\mu_k}\|d^k\|
	+E_k^+
	+\frac{L_G}{\mu_k}\|d^k\|
	+E_k^0
	\\
	&\quad=
	\gamma_kL_f\|d^k\|
	+\frac{L_k+L_G}{\mu_k}\|d^k\|
	+E_k^0+E_k^+.
\end{align*}
where we also used $L_f$-Lipschitz continuity of $\nabla f$,
Lemma~\ref{lem:G-lip}, and
Lemma~\ref{lem:inexact-errors}(iv). Since $\|d^k\|=\mu_k\omega_{k+1}$,
	$\gamma_k\leq\widehat\gamma_0$,
	$\mu_k\leq\mu_0$, and $L_k\leq\overline L$, we obtain
	\[
	\left\|
	\gamma_k\nabla f(w^{k+1})
	+
	\gamma_k\widetilde\lambda_{k+1}
	\nabla R{}(w^{k+1})
	+n^{k+1}
	\right\|
	\leq
	M_1\omega_{k+1}+E_k^0+E_k^+.
	\]
	Dividing by $\gamma_k$, noting that $\frac{n^{k+1}}{\gamma_k}\in N_C(w^{k+1})$ and by the definition of the distance function, we obtain~\eqref{eq:stationarity-estimate}.
\end{proof}

\begin{lemma}[Feasibility and complementarity]
	\label{lem:feasibility-estimate}
	Suppose that, at iteration $k$,
	no penalty update occurs but the second condition in \eqref{eq:inexact-PU} is satisfied. Then
	\begin{align}
		\Phi_\rho^{\rm f}(w^{k+1})
		&\leq
		2a_\phi\mu_k+\zeta_{k+1},
		\label{eq:feasibility-estimate}\\
		\Phi_\rho^{\rm c}
		(w^{k+1},\widetilde\lambda_{k+1})
		&\leq
		\gamma_k^{-1}
		\left(2a_\phi\mu_k+\zeta_{k+1}\right),
		\label{eq:complementarity-estimate}
	\end{align}
	where $	\Phi_\rho^{\rm f}$ and $	\Phi_\rho^{\rm c}$ are given in Definition~\ref{def:approx-KKT}
\end{lemma}

\begin{proof}
	By the hypothesis, we have  $\widetilde R_k^+\leq 2a_\phi\mu_k.$
	Thus, by \eqref{eq:app-residual-error}, we have 
	\[
	R{}(w^{k+1})
	\leq
	\widetilde R_k^++\zeta_{k+1} \leq 2a_\phi\mu_k + \zeta_{k+1}.
	\]
	Consequently, $$\Phi_\rho^{\rm f}(w^{k+1}) = [R(w^{k+1})]_+ \leq [2a_\phi \mu_k + \zeta_{k+1}]_+= 2a_\phi \mu_k+ \zeta_{k+1},$$ proving \eqref{eq:feasibility-estimate}.
	
	For complementarity, write
	$
	R{}(w^{k+1})
	=
	\widetilde R_k^++e_k^+,$ where 
	$0\leq e_k^+\leq\zeta_{k+1}.
	$
	Then
	\begin{align*}
		\left|
		\widetilde\lambda_{k+1}
		R{}(w^{k+1})
		\right| &\leq
		\gamma_k^{-1}
		\left|
		\phi_{\mu_k}'(\widetilde R_k^+)
		\widetilde R_k^+
		\right|
		+
		\gamma_k^{-1}
		\phi_{\mu_k}'(\widetilde R_k^+)e_k^+ \\
		& \leq \gamma_k^{-1}
		\left|
		\phi_{\mu_k}'(\widetilde R_k^+)
		\widetilde R_k^+
		\right|
		+
		\gamma_k^{-1}
		\zeta_{k+1} \\
		& \leq \gamma_k^{-1} (2a_\phi\mu_k) + 	\gamma_k^{-1}
		\zeta_{k+1},
	\end{align*}
	where the second inequality holds by \eqref{eq:phi-derivative-bound}, and the last inequality holds by noting that $\widetilde R_k^+\leq 2a_\phi\mu_k$ and invoking Lemma~\ref{lem:phi-products}(ii) with $\theta=2$. By the definition of $\Phi_\rho^{\rm c}$, we get~\eqref{eq:complementarity-estimate}.
\end{proof}

\begin{lemma}[KKT estimates at a non-update iteration]
	\label{lem:nonupdate-kkt}
	Suppose
	$
	\Xi_\infty
	\coloneqq
	\sum_{k=0}^{\infty}\tau_k
	<\infty,
	$
	and let $\underline\gamma>0$ be the lower bound from
	Lemma~\ref{lem:penalty-stabilization}. Define
	\begin{equation}
		t_*
		\coloneqq
		\max\{t\geq0:
		\widehat\gamma_t\geq\underline\gamma\},
		\notag 
	\end{equation}
	For $K\geq0$, let $m\coloneqq\lceil K/2\rceil$ and define
	\begin{align}
			S_K
		& \coloneqq
		\sum_{k=m}^{K}\mu_k,
		\qquad
		\widetilde S_K
		\coloneqq
		S_K-t_*\mu_m.
		\label{eq:modified-SK} \\ 
		\mathcal E_K
		&\coloneqq
		\max_{m\leq j\leq K}
		(E_j^0+E_j^+),
		\label{eq:window-E}\\
		\mathcal Z_K
		&\coloneqq
		\max_{m\leq j\leq K}
		\zeta_{j+1}.
		\label{eq:window-zeta}
	\end{align}
	If
	\begin{equation}
		\widetilde S_K>0
		\qquad \text{and} \qquad 
		\widetilde S_K^{-1}
		\leq
		\frac{(c_2\underline\gamma)^2}{M_2},
		\label{eq:window-small-step-condition}
	\end{equation}
	where $	M_2
	\coloneqq
	\frac{2(Q_0+a_1\mu_0+\Xi_\infty)}{c_1},$
	then there exists an index $\widehat k\in[m,K]$ at which no
	penalty update occurs and such that
	\begin{align}
		\Phi_\rho^{\rm s}
		(w^{\widehat k+1},
		\widetilde\lambda_{\widehat k+1})
		&\leq
		\underline\gamma^{-1}
		\left(
		M_1\sqrt{\frac{M_2}{\widetilde S_K}}
		+\mathcal E_K
		\right),
		\label{eq:window-stationarity}\\
		\Phi_\rho^{\rm f}(w^{\widehat k+1})
		&\leq
		2a_\phi\mu_m+\mathcal Z_K,
		\label{eq:window-feasibility}\\
		\Phi_\rho^{\rm c}
		(w^{\widehat k+1},
		\widetilde\lambda_{\widehat k+1})
		&\leq
		\underline\gamma^{-1}
		\left(
		2a_\phi\mu_m+\mathcal Z_K
		\right).
		\label{eq:window-complementarity}
	\end{align}
\end{lemma}

\begin{proof}
	Since $\gamma_k\geq\underline\gamma$ for every $k$ and
	$\widehat\gamma_t\downarrow0$, the integer $t_*$ is finite.
	Moreover, $t_0=0$ and every penalty update increases $t_k$ by one.
	Since $\gamma_k=\widehat\gamma_{t_k}\geq\underline\gamma$, we have
	$t_k\leq t_*$ for every $k$. Hence, the total number of penalty
	updates is at most $t_*$.
	
	Let
	$
	\mathcal I
	\coloneqq
	\{k\geq0:t_{k+1}=t_k+1\}
	$
	be the set of penalty-update iterations, and define
	$
	\mathcal N_K
	\coloneqq
	\{m,\ldots,K\}\setminus\mathcal I.
	$	Thus, $\mathcal N_K$ consists of the non-update iterations in the
	second half of the first $K+1$ iterations.

	Summing \eqref{eq:Lyapunov-descent-app} from $k=m$ to $K$ gives
	\begin{equation}
		\frac{c_1}{2}
		\sum_{k=m}^{K}\mu_k\omega_{k+1}^2
		\leq
		Q_m-Q_{K+1}
		+
		\sum_{k=m}^{K}\tau_k.
		\label{eq:second-half-descent}
	\end{equation}
	On the other hand, summing \eqref{eq:Lyapunov-descent-app}
	$k=0$ to $m-1$ yields 
	\begin{equation}
			Q_m
		\leq
		Q_0+\sum_{k=0}^{m-1}\tau_k -\frac{c_1}{2}
		\sum_{k=0}^{m-1}\mu_k\omega_{k+1}^2 \leq 	Q_0+\sum_{k=0}^{m-1}\tau_k.
		\label{eq:first-half-descent}
	\end{equation}
	Meanwhile, because
	$f(w)-f_{\inf}\geq0$,
	$G_\mu(w)\geq-a_1\mu$ by \eqref{eq:phi-approx} and  $\{\mu_k\}$ is nonincreasing, we have $
		Q_k \geq -a_1\mu_0$ for all  $k\geq 0$. 
	Thus, combining \eqref{eq:first-half-descent} and
	\eqref{eq:second-half-descent} yields
	\[
	\frac{c_1}{2}
	\sum_{k=m}^{K}\mu_k\omega_{k+1}^2
	\leq
	Q_0+a_1\mu_0+\sum_{k=0}^{K}\tau_k
	\leq
	Q_0+a_1\mu_0+\Xi_\infty.
	\]
	Hence, by the definition of $M_2$, we have
	\begin{equation}
		\sum_{k\in\mathcal N_K}
		\mu_k\omega_{k+1}^2
		\leq
		M_2.
		\label{eq:nonupdate-step-sum}
	\end{equation}
	
	Since $\{\mu_k\}$ is nonincreasing and
	$|\mathcal I|\leq t_*$, we have from the definition of $S_K$ and $\widetilde S_K$ in~\eqref{eq:modified-SK} 
	\begin{align}
		\sum_{k\in\mathcal N_K}\mu_k
		&=
		S_K-
		\sum_{k\in\mathcal I\cap\{m,\ldots,K\}}\mu_k
		\notag\\
		&\geq
		S_K-
		|\mathcal I\cap\{m,\ldots,K\}|\mu_m
		\notag\\
		&\geq
		S_K-t_*\mu_m
		=
		\widetilde S_K.
		\label{eq:nonupdate-weight}
	\end{align}
	Since $\widetilde S_K>0$ by \eqref{eq:window-small-step-condition}, the set $\mathcal N_K$ is nonempty.
	Combining~\eqref{eq:nonupdate-step-sum} and
	\eqref{eq:nonupdate-weight}, there exists
	$\widehat k\in\mathcal N_K$ such that
	\begin{equation}
		\omega_{\widehat k+1}^2
		\leq
		\frac{
			\sum_{k\in\mathcal N_K}
			\mu_k\omega_{k+1}^2
		}{
			\sum_{k\in\mathcal N_K}\mu_k
		}
		\leq
		\frac{M_2}{\widetilde S_K}.
		\label{eq:nonupdate-small-step}
	\end{equation}
	By~\eqref{eq:window-small-step-condition} and
	$\gamma_{\widehat k}\geq\underline\gamma$,
	\[
	\omega_{\widehat k+1}
	\leq
	\sqrt{\frac{M_2}{\widetilde S_K}}
	\leq
	c_2\underline\gamma
	\leq
	c_2\gamma_{\widehat k}.
	\]
	Since $\omega_{\widehat k+1} = \frac{\norm{d^{\widehat k}}}{\mu_{\widehat k}}$, it follows from the above inequality that the second condition in \eqref{eq:inexact-PU} holds. Meanwhile,
	$\widehat k\in\mathcal N_K$ so that no penalty update occurs at iteration $\widehat k$. Hence, Lemma~\ref{lem:feasibility-estimate}  implies that 
	\[
	\Phi_\rho^{\rm f}(w^{\widehat k+1})
	\leq
	2a_\phi\mu_{\widehat k}+\zeta_{\widehat k+1}.
	\]
	Since $\widehat k\geq m$, the monotonicity of $\{\mu_k\}$ and the
	definition of $\mathcal Z_K$ in \eqref{eq:window-zeta} imply
	$
	\mu_{\widehat k}\leq\mu_m$ and
	$
	\zeta_{\widehat k+1}\leq\mathcal Z_K
	$. Hence,~\eqref{eq:window-feasibility} immediately follows. By a similar argument and noting that
	$\gamma_{\widehat k}\geq\underline\gamma$, we obtain~\eqref{eq:window-complementarity}.
	
	Finally, Lemma~\ref{lem:stationarity-estimate}, \eqref{eq:nonupdate-small-step} and \eqref{eq:window-E} yield
	\begin{align*}
		\Phi_\rho^{\rm s}
		(w^{\widehat k+1},
		\widetilde\lambda_{\widehat k+1})
		&\leq
		\gamma_{\widehat k}^{-1}
		\left(
		M_1\omega_{\widehat k+1}
		+E_{\widehat k}^0+E_{\widehat k}^+
		\right)\\
		&\leq
		\underline\gamma^{-1}
		\left(
		M_1\sqrt{\frac{M_2}{\widetilde S_K}}
		+\mathcal E_K
		\right),
	\end{align*}
	which proves~\eqref{eq:window-stationarity}.
\end{proof}

To convert the preceding estimate into explicit polynomial rates, we record the following elementary consequences of the \alert{smoothing condition}.
\begin{lemma}[Polynomial smoothing estimates]
	\label{lem:polynomial-smoothing-estimates}
	\alert{Let $\{\mu_k\}$ be positive and nonincreasing and satisfy
		$
		\mu_k=\Theta((1+k)^{-r}),
		$
		with $r\in(0,1)$,}
	and define $S_K$ as in \eqref{eq:modified-SK} with $m=\lceil K/2\rceil$.
	Then
	\begin{equation}
		S_K=\Theta(K^{1-r}),
		\qquad
		\mu_{m}=\Theta(K^{-r}).
		\label{eq:SK-rate}
	\end{equation}
	Consequently, for $\widetilde S_K$ given in \eqref{eq:modified-SK} with any fixed $t_*\geq 0$, we have
	\begin{equation}
		\widetilde S_K=\Theta(K^{1-r}).
		\label{eq:modified-SK-rate}
	\end{equation}
	In particular, $\widetilde S_K>0$ for all sufficiently large $K$.
\end{lemma}

\begin{proof}
	Since
	$K-m+1=\lfloor K/2\rfloor+1$, for $K\geq2$,
	\begin{equation}
		\frac{K}{2}
		\leq
		K-m+1
		\leq
		K\label{eq:K-m+1}
	\end{equation}
	\alert{By $\mu_k=\Theta((1+k)^{-r})$, there exist constants
		$c_\mu,C_\mu>0$ such that, for all sufficiently large $k$,
		\[
		c_\mu(1+k)^{-r}
		\leq
		\mu_k
		\leq
		C_\mu(1+k)^{-r}.
		\]
		Hence, for all sufficiently large $K$ and $k=m,\ldots,K$,
		\[
		c_\mu(K+1)^{-r}
		\leq
		\mu_k
		\leq
		C_\mu(m+1)^{-r}.
		\]}
	Hence, summing from $k=m$ to $k=K$ and noting \eqref{eq:K-m+1}, 
	\[
	\alert{\frac{K}{2}c_\mu(K+1)^{-r}}
	\leq
	S_K
	\leq
	\alert{KC_\mu(m+1)^{-r}}.
	\]
	Using $K+1\leq2K$ and $m+1\geq K/2$ gives
	\begin{equation}
		\alert{\frac{c_\mu}{2^{r+1}}}K^{1-r}
		\leq
		S_K
		\leq
		\alert{2^rC_\mu}K^{1-r},
		\label{eq:SK-squeeze}
	\end{equation}
	and therefore $S_K=\Theta(K^{1-r})$. Similarly,
	\alert{since
		$
		c_\mu(m+1)^{-r}
		\leq
		\mu_m
		\leq
		C_\mu(m+1)^{-r},
		$}
	and $K/2\leq m+1\leq K$ for $K\geq2$, we have
	\begin{equation}
		\alert{c_\mu}K^{-r}
		\leq
		\mu_m
		\leq
		\alert{2^rC_\mu}K^{-r}.
		\label{eq:mu-squeeze}
	\end{equation}
	Thus, $\mu_{\lceil K/2\rceil}=\Theta(K^{-r})$. Hence, both the estimates in \eqref{eq:SK-rate} hold. 
	
	Finally, using \eqref{eq:SK-squeeze} and \eqref{eq:mu-squeeze},
	\begin{equation}
		\frac{t_*\mu_m}{S_K}
		\leq
		\frac{
			\alert{t_*2^rC_\mu}K^{-r}
		}{
			\alert{(c_\mu/2^{r+1})}K^{1-r}
		}
		=
		\alert{\frac{2^{2r+1}t_*C_\mu}{c_\mu K}},
		\label{eq:tstar-mu-SK}
	\end{equation}
	Since
	\[
	\frac{\widetilde S_K}{S_K}
	=
	1-\frac{t_*\mu_m}{S_K}
	\longrightarrow 1
	\]
	by~\eqref{eq:tstar-mu-SK}, we have
	$\widetilde S_K\sim S_K$. In particular,
	$\widetilde S_K=\Theta(S_K)$, and since
	$S_K=\Theta(K^{1-r})$, \eqref{eq:modified-SK-rate} holds.
\end{proof}
\alert{For instance, for any fixed integer $B\geq1$, the blockwise schedule
$\mu_k=\mu_0(1+\lfloor k/B\rfloor)^{-r}$ satisfies
$\mu_k=\Theta((1+k)^{-r})$; the standard choice
$\mu_k=\mu_0(1+k)^{-r}$ corresponds to $B=1$. }

We now give the complete statement of
Theorem~\ref{thm:complexity}, including the explicit estimates from
which the rates in the main text follow.

\medskip
\noindent\textbf{Theorem~\ref{thm:complexity}
	(complete statement).}
\begin{itshape}
	Suppose Assumption~\ref{assume:A} and ENNAMCQ hold, and suppose
	that the inverse-penalty grid satisfies~\eqref{eq:grid-ratio}. Let $\{\mu_k\}$ be positive and
	nonincreasing with \alert{$\mu_k=\Theta((1+k)^{-r})$} for some $r\in(0,1)$,
	and let $\{\zeta_k\}$ satisfy
	$\zeta_k\leq\bar\zeta\mu_k^q$ for some $\bar\zeta>0$ and $q>1/r$.
	Let $\{w^k\}$ be generated by Algorithm~\ref{alg:main} and, for each
	$k\geq0$, set
	$\widetilde\lambda_{k+1}
	\coloneqq
	\gamma_k^{-1}\phi_{\mu_k}'(\widetilde R^{k+1})$.
	Then, for every sufficiently large $K$, there exists
	$\widehat k\in[\lceil K/2\rceil,K]$ such that
	\begin{align}
		\Phi_\rho^{\rm s}
		(w^{\widehat k+1},\widetilde\lambda_{\widehat k+1})
		&=
		O\!\left(K^{-(1-r)/2}\right),
		\label{eq:main-rate-stationarity-app}\\
		\Phi_\rho^{\rm f}(w^{\widehat k+1})
		&=
		O(K^{-r}),
		\label{eq:main-rate-feasibility-app}\\
		\Phi_\rho^{\rm c}
		(w^{\widehat k+1},\widetilde\lambda_{\widehat k+1})
		&=
		O(K^{-r}).
		\label{eq:main-rate-complementarity-app}
	\end{align}
	Consequently, an
	$(\varepsilon_s,\varepsilon_f,\varepsilon_c)$-KKT point is
	obtained within
	\begin{equation}
			K
		=
		O\!\left(
		\max\left\{
		\varepsilon_s^{-2/(1-r)},
		\varepsilon_f^{-1/r},
		\varepsilon_c^{-1/r}
		\right\}
		\right)
		\label{eq:K-outer}
	\end{equation}
	outer iterations.
\end{itshape}

\begin{proof}
	We verify the conditions required to invoke Lemma~\ref{lem:nonupdate-kkt}. Since
	$\zeta_k\leq\bar\zeta\mu_k^q$ with $q>1/r>1$, we have
	$\zeta_k/\mu_k\leq\bar\zeta\mu_k^{q-1}\to0$.
	Hence $\zeta_k=o(\mu_k)$, and
	Lemma~\ref{lem:penalty-stabilization} yields
	$\gamma_k\geq\underline\gamma>0$ for all $k\geq0$ and finite
	stabilization of the penalty parameter.
	
Next, we verify the summability condition on $\{\tau_k\}$.
	From \eqref{eq:constants-for-error-bounds}, we have
	$
	E_k^0
	\leq
	c_D\mu_k^{q/2}+c_E\mu_k^{q-1},
	$ where 	$
	c_D\coloneqq
	L_{a,\rho}\sqrt{2\bar\zeta/m_\rho}
	$
	and
	$
	c_E\coloneqq b_\phi B_R\bar\zeta.
	$
	Moreover, since $\mu_{k+1}\leq\mu_k$,
	$\zeta_{k+1}\leq\bar\zeta\mu_{k+1}^q
	\leq\bar\zeta\mu_k^q$, and therefore the same bound holds for
	$E_k^+$:
	\begin{equation}
		E_k^0+E_k^+
		\leq
		2c_D\mu_k^{q/2}
		+
		2c_E\mu_k^{q-1}.
		\label{eq:E-poly}
	\end{equation}
	Using $(a+b)^2\leq2a^2+2b^2$, 
	$\zeta_{k+1}\leq\bar\zeta\mu_k^q$, and $\vartheta_k\geq\underline\vartheta$ by definition, we obtain
	\begin{align*}
		\tau_k = \tau_{k,L_k}
		&=
		\frac{\mu_k}{2\vartheta_k}(E_k^0)^2
		+\zeta_k+\zeta_{k+1}\\
		&\leq
		\frac{c_D^2}{\underline\vartheta}\mu_k^{q+1}
		+
		\frac{c_E^2}{\underline\vartheta}\mu_k^{2q-1}
		+
		2\bar\zeta\mu_k^q.
	\end{align*}
	Since $rq>1$, we also have
	$r(q+1)>1$ and $r(2q-1)=2rq-r>2-r>1$.
	Hence all three series on the above right-hand side are summable, and therefore
	$\Xi_\infty\coloneqq\sum_{k=0}^\infty\tau_k<\infty$.
	In particular, the constant $M_2$ in
	Lemma~\ref{lem:nonupdate-kkt} is finite.
	
	Let $m=\lceil K/2\rceil$. By
	Lemma~\ref{lem:polynomial-smoothing-estimates},
	$
	S_K=\Theta(K^{1-r}),
	$
	$
	\mu_m=\Theta(K^{-r}),
	$
	and
	$
	\widetilde S_K=\Theta(K^{1-r}).
	$
	Thus $\widetilde S_K>0$ for all sufficiently large $K$ and
	$\widetilde S_K^{-1}=O(K^{-(1-r)})\to0$.
	Hence~\eqref{eq:window-small-step-condition} holds for all
	sufficiently large $K$. At this point, for all sufficiently large $K$, the conditions of
	Lemma~\ref{lem:nonupdate-kkt} are satisfied.
	
	We derive the complexity of the right-hand sides of \eqref{eq:window-stationarity}--\eqref{eq:window-complementarity} in Lemma~\ref{lem:nonupdate-kkt}. For $m\leq j\leq K$,
	\eqref{eq:E-poly} and the monotonicity of $\mu_j$ give
	\[
	E_j^0+E_j^+
	\leq
	2c_D\mu_m^{q/2}
	+
	2c_E\mu_m^{q-1}.
	\]
	Using $\mu_m=\Theta(K^{-r})$ from
	\eqref{eq:SK-rate}, we obtain
	\begin{equation}
			\mathcal E_K
		=
		O\!\left(
		K^{-rq/2}+K^{-r(q-1)}
		\right).
		\label{eq:mathcalE_K-rate}
	\end{equation}
	Because $rq>1$, we have
	$
	rq/2>1/2>(1-r)/2
	$
	and
	$
	r(q-1)=rq-r>1-r>(1-r)/2.
	$
	Therefore, together with \eqref{eq:mathcalE_K-rate}, 
	\begin{equation}
		\mathcal E_K
		=
		o\!\left(K^{-(1-r)/2}\right).
		\label{eq:E-window-poly}
	\end{equation}
	
	Similarly, for $m\leq j\leq K$,
	$\zeta_{j+1}\leq\bar\zeta\mu_{j+1}^q
	\leq\bar\zeta\mu_m^q$, and thus
	$
	\mathcal Z_K
	\leq\bar\zeta\mu_m^q
	=
	O(K^{-rq}).
	$
	Since $q>1$,
	\begin{equation}
		\mathcal Z_K=o(K^{-r}).
		\label{eq:zeta-window-poly}
	\end{equation}
	
	Lemma~\ref{lem:nonupdate-kkt} therefore yields, for every
	sufficiently large $K$, an index
	$\widehat k\in[m,K]$ satisfying
	\eqref{eq:window-stationarity}--%
	\eqref{eq:window-complementarity}.
Since
$
\widetilde S_K^{-1/2}
=
O(K^{-(1-r)/2})
$ from~\eqref{eq:modified-SK-rate},
\eqref{eq:E-window-poly} shows that the right-hand side of
\eqref{eq:window-stationarity} is
\[
O(K^{-(1-r)/2})
+
o(K^{-(1-r)/2})
=
O(K^{-(1-r)/2}).
\]
Likewise,
$\mu_m=\Theta(K^{-r})$ and
\eqref{eq:zeta-window-poly} show that the right-hand sides of
\eqref{eq:window-feasibility} and
\eqref{eq:window-complementarity} are
\[
O(K^{-r})+o(K^{-r})=O(K^{-r}).
\]
This proves
\eqref{eq:main-rate-stationarity-app}--%
\eqref{eq:main-rate-complementarity-app}. Finally, solving the three rate bounds for $K$ gives
\eqref{eq:K-outer},
	which proves the claimed outer-iteration complexity.
\end{proof}
%

\begin{remark}[Lower-level interpretation]
	\label{rem:delta-stationarity}
	The feasibility estimate \eqref{eq:main-rate-feasibility-app} in Theorem~\ref{thm:complexity},
	together with Corollary~\ref{cor:kkt-lower-level-stationarity},
	shows that the returned iterate
	$w^{\widehat k+1}=(x^{\widehat k+1},y^{\widehat k+1})$
	satisfies
	$
	H_\rho(w^{\widehat k+1})
	\leq
	\delta+O(K^{-r}),
	$
	and consequently, its associated Moreau proximal point
	$z_\rho(w^{\widehat k+1})$ satisfies
	\[
	\|y^{\widehat k+1}-z_\rho(w^{\widehat k+1})\|
	=
	O\!\left(\sqrt{\delta+K^{-r}}\right),
	\]
	and
	\[
	\operatorname{dist}\!\left(
	0,\,
	\nabla_y g\!\left(
	x^{\widehat k+1},z_\rho(w^{\widehat k+1})
	\right)
	+
	N_Y\!\left(z_\rho(w^{\widehat k+1})\right)
	\right)
	=
	O\!\left(\sqrt{\delta+K^{-r}}\right).
	\]
	Thus, the returned $y$-iterate is close to a point that is
	near-stationary for the lower-level problem. In particular, with
	$r=1/3$ and $K=O(\varepsilon^{-3})$, both errors are
$O(\sqrt{\delta+\varepsilon})$.
\end{remark}

We next give the explicit oracle complexity obtained when the Moreau
subproblems are solved by Algorithm~\ref{alg:ME-evaluation}.
The argument is unchanged for any first-order inner method that computes
a $\zeta$-accurate Moreau evaluation in $O(\log(1/\zeta))$ iterations
uniformly over the proximal subproblems; only the problem-dependent
constant in the complexity bound changes.

\medskip
\noindent\textbf{Corollary~\ref{cor:overall-complexity}
	(Overall first-order complexity).}
\textit{Suppose the conditions of
	Theorem~\ref{thm:complexity} hold and, in addition, choose
	$\zeta_k=\Theta(\mu_k^q)$. Suppose
	Algorithm~\ref{alg:ME-evaluation} is used for each required
	Moreau evaluation. Let
	$B_{\rm ls}$ denote a uniform bound on the number of trial Moreau
	evaluations per outer iteration, with $B_{\rm ls}=1$ for Strategy I.
	Then the overall first-order/projection complexity of IVSP through
	outer iteration $K$ is
	\begin{equation}
		\label{eq:overall-complexity-app}
		O\!\left(
		B_{\rm ls}\kappa_{\psi,\rho}\,
		K\log K
		\right).
	\end{equation}
	In particular, for the balanced choice $r=1/3$, IVSP computes an
	$\varepsilon$-KKT point with overall first-order/projection complexity
	\begin{equation}
		\label{eq:balanced-overall-complexity-app}
		O\!\left(
		B_{\rm ls}\kappa_{\psi,\rho}\,
		\varepsilon^{-3}
		\log\frac{1}{\varepsilon}
		\right)
		=
		\widetilde O(\varepsilon^{-3}).
	\end{equation}
}

\begin{proof}
	By Lemma~\ref{lem:ME-residual-complexity}, each call to
	Algorithm~\ref{alg:ME-evaluation} with tolerance $\zeta$ requires
	$
	O\!\left(
	\kappa_{\psi,\rho}
	\log\frac{1}{\zeta}
	\right)
	$
	first-order/projection iterations for the Moreau subproblem. The
	implicit constant is uniform over the outer iterates and
	line-search trials.
	
	At outer iteration $k$, every new trial-point Moreau evaluation is
	required to be $\zeta_{k+1}$-accurate. Since $\zeta_k=\Theta(\mu_k^q)$ and
$\mu_k=\Theta((1+k)^{-r})$, we have
$
\zeta_{k+1}
=
\Theta((k+2)^{-rq}),
$
	and hence
	$
	\log(1/\zeta_{k+1})
	=
	O(\log(k+2)).
	$
	Thus, each new trial-point proximal solve at outer iteration $k$
	requires
	$
	O\!\left(
	\kappa_{\psi,\rho}\log(k+2)
	\right)
	$
	first-order/projection iterations.
	
	After the initial $\zeta_0$-accurate Moreau evaluation, Strategy I
	requires one new proximal solve per outer iteration. Under
	Strategy II, every line-search trial requires one new proximal
	solve, and Lemma~\ref{lem:descent} bounds the number of trials per
	outer iteration by $B_{\rm ls}$. The Moreau evaluation associated
	with an accepted trial is retained for the next outer iteration,
	so no additional current-point proximal solve is required.
	Therefore, the total first-order/projection work used by the
	Moreau subproblems through outer iteration $K$ is
	\begin{align*}
		&
		O\!\left(
		\kappa_{\psi,\rho}
		\log\frac{1}{\zeta_0}
		\right)
		+
		O\!\left(
		B_{\rm ls}\kappa_{\psi,\rho}
		\sum_{k=0}^{K-1}\log(k+2)
		\right)	=
		O\!\left(
		B_{\rm ls}\kappa_{\psi,\rho}\,
		K\log K
		\right).
	\end{align*}
	
	The outer method itself requires only a constant number of
	first-order evaluations per accepted iteration and at most
	$B_{\rm ls}$ projections onto $C$ per iteration. Its
	first-order/projection work is therefore $O(B_{\rm ls}K)$, which
	is dominated by \eqref{eq:overall-complexity-app}. This proves the
	first claim.
	
	For the balanced choice $r=1/3$,
	Theorem~\ref{thm:complexity} gives
	$K=O(\varepsilon^{-3})$. Substituting this into
	\eqref{eq:overall-complexity-app} gives
	\eqref{eq:balanced-overall-complexity-app}.
\end{proof}

\section{Details on the complexity comparison}
\label{app:complexity-comparison}

Table~\ref{tab:complexity-comparison} compares representative
deterministic first-order bilevel methods under a common norm-based
$\varepsilon$ convention whenever possible. Since the methods target
different reformulations and solution concepts, the reported rates
should be interpreted together with the ``Guarantee'' column rather
than as complexities for an identical stopping criterion.

\paragraph{Solution criteria.}
For methods assuming a unique lower-level solution,
\emph{hyperobjective stationarity} refers to
$\|\nabla\Phi(x)\|\leq\varepsilon$, where
$\Phi(x)=f(x,y^\star(x))$. \emph{Penalty stationarity} refers to
$\varepsilon$-stationarity of the penalized or smoothed surrogate used
by the corresponding method. A KKT guarantee controls stationarity,
constraint feasibility, and complementarity for the relevant
constrained reformulation. BOME's KKT$^\ast$ criterion is instead based
on the value-gap residual
\[
\min_{\lambda\geq0}
\|\nabla f+\lambda\nabla q\|^2+q,
\qquad q=g-g^\star,
\]
and is not the classical KKT residual. The weak-KKT notion of
Lu--Mei~\citep{lumei} combines stationarity of their minimax
reformulation with lower-level value-gap feasibility.

\paragraph{Rate conventions.}
When a method controls several residuals, the complexity displayed in
Table~\ref{tab:complexity-comparison} is the number of iterations
required for all quantities indicated in the ``Guarantee'' column to
be $O(\varepsilon)$. Rates originally stated using squared residuals
are converted to this norm-based convention. For V-PBGD~\citep{VPBGD}, the reported
$\widetilde O(\varepsilon^{-2})$ rate treats its prescribed
penalty coefficient as fixed. For SLM~\citep{SLM}, we
report the $O(\varepsilon^{-2})$ rate with its relaxation parameter
treated as fixed, consistent with the fixed-relaxation setting of
IVSP; the hidden constant depends polynomially on that relaxation
level. For MEHA~\citep{MEHA}, stationarity of a fixed penalized problem
admits the $\mathcal O(\nu\epsilon^{-2})$ rate reported in
\citet{lu2025tsp}, whereas requiring both
$\mathcal O(\epsilon)$ stationarity and
$\mathcal O(\epsilon)$ unrelaxed Moreau-gap feasibility in its
increasing-penalty analysis gives $\mathcal O(\epsilon^{-4})$.
PNGBiO~\citep{PNGBiO} obtains an
$\mathcal O(T^{-1/4})$ rate for both its stationarity measure and
unrelaxed Moreau-gap violation, yielding the same
$\mathcal O(\epsilon^{-4})$ joint criterion.

For IVSP, the relaxation level $\delta>0$ is prescribed and fixed.
The $\mathcal O(\epsilon^{-3})$ outer-iteration bound controls
stationarity, feasibility, and complementarity for the KKT system of
the relaxed constraint
$
g(x,y) - g_\rho(x,y) - \delta\leq 0.
$
Thus, the IVSP rate and the above MEHA/PNGBiO rates do not correspond
to an identical feasibility target: the latter drive the unrelaxed
Moreau-gap violation to $\mathcal O(\epsilon)$, whereas IVSP drives
$[g(x,y) - g_\rho(x,y)-\delta]_+$ to $\mathcal O(\epsilon)$ for fixed $\delta$.
Including the logarithmic work required by the inexact Moreau
subproblems gives
$\widetilde{\mathcal O}(\epsilon^{-3})$ first-order/projection
complexity for this fixed-relaxation KKT target. Similar to SLM \citep{SLM}, the constants in the IVSP bound are not claimed to be uniform as
$\delta\downarrow0$; see Appendix~\ref{app:moreau-relaxation}.

\paragraph{Per-iteration oracle costs.}
The complexity rates in Table~\ref{tab:complexity-comparison} should
not be interpreted as assuming identical per-iteration oracle costs.
In particular, the iterations of MEHA~\citep{MEHA} and
PNGBiO~\citep{PNGBiO} are formed from first-order information without
requiring lower-level objective-value evaluations. IVSP additionally
requires such values to estimate the Moreau-gap residual. Under
Strategy~I, each outer iteration requires one new inexact Moreau solve
and, after this solve, two evaluations of $g$ to form
$\widetilde R_{k+1}$. Thus, although its lower-level
first-order/projection complexity is
$O(\varepsilon^{-3}\log(1/\varepsilon))$, it also incurs
$O(\varepsilon^{-3})$ lower-level value evaluations.

Under Strategy~II, if $B_k$ backtracking trials are performed at
iteration $k$, each trial requires a new inexact Moreau solve, two
lower-level value evaluations to form the trial residual, and an
upper-level value evaluation for the Armijo test. Hence these
trial-point costs scale by $B_k\leq B_{\rm ls}$, whereas
current-point first-order quantities are computed only once or
retained from the preceding accepted iteration. Consequently, the
corresponding value-oracle costs are
$O(B_{\rm ls}\varepsilon^{-3})$, while Corollary~\ref{cor:overall-complexity}
accounts separately for the
$O(B_{\rm ls}\varepsilon^{-3}\log(1/\varepsilon))$
lower-level first-order/projection work.

\section{Experimental Details}
\label{app:experimental-details}

This section provides the complete problem formulations and
implementation details for the numerical experiments in
Section~\ref{sec:experiments}.

We conducted all experiments on an Ubuntu 20.04 workstation equipped
with an AMD EPYC 7413 24-Core Processor, 512 GB of system memory,
and an NVIDIA RTX A5000 GPU with 24 GB of GPU memory.
We used Python 3.10.12 and PyTorch 2.6.0 with CUDA 12.4.

All optimization runs are deterministic after fixing the problem instances,
data splits, and initializations. In the real-world experiments, we use fixed
training, validation, and test sets for all methods, and no stochastic
mini-batching is used during optimization. The test set is used only for
evaluation.
All compared methods are run under the same wall-clock budget within
each problem: 1 second for the synthetic experiment, 60 seconds for both
MNIST and FashionMNIST hyper-cleaning, and
600 seconds for Omniglot few-shot learning.
Wall-clock time includes all optimization computations, including
inner-loop computations, but excludes data loading and offline metric
evaluation.
The main hyperparameters of all compared methods are tuned separately
for each problem using method-specific candidate values.
For the real-world tasks, we use the same validation-based tuning
protocol and comparable tuning budgets for all methods within each
problem. 
Test performance is not used for hyperparameter
selection.

\paragraph{Implementation details.}
We make two practical implementation choices for
Algorithm~\ref{alg:inexact-penalty-smoothing}: a finite-precision
relaxation of the inexact Armijo test and a safeguarded
Barzilai--Borwein initialization of the backtracking parameter.

First, in finite-precision arithmetic, we use the following numerically
relaxed version of the inexact Armijo condition~\eqref{eq:robust-armijo}:
\begin{align}
    &\gamma_k f(w_L^k)
    +\phi_{\mu_k}
    \bigl(\widetilde R_{k,L}^+ + \zeta_{k+1}\bigr)\\
    &\qquad \leq
    \gamma_k f(w^k)
    +\phi_{\mu_k}(\widetilde R^k)
    -\frac{c_1}{2\mu_k}\|w_L^k-w^k\|^2
    +\tau_{k,L}
    +10^{-12}.
    \label{eq:implemented-robust-armijo}
\end{align}
The additive tolerance $10^{-12}$ is used only to prevent numerical
stalling due to floating-point roundoff when the two sides of
\eqref{eq:robust-armijo} are nearly identical, and is not part of the
theoretical algorithm or convergence analysis.

Second, for the backtracking variant of IVSP, we initialize the
line-search parameter using a safeguarded Barzilai--Borwein (BB)
estimate~\citep{barzilai1988two}, similarly to the implementation of
\citet{xu2026smoothing}.
Specifically, define
\[
\Delta w_k:=w^k-w^{k-1},
\qquad
\Delta v_k
:=
\gamma_k
\bigl(
\nabla f(w^k)-\nabla f(w^{k-1})
\bigr)
+
\widetilde{\nabla}G_k
-
\widetilde{\nabla}G_{k-1},
\qquad k\geq1.
\]
We form
\begin{equation}
\label{eq:bb-initialization}
L_k^{\mathrm{BB}}
=
\mu_k
\frac{\|\Delta v_k\|^2}
     {|\langle \Delta w_k,\Delta v_k\rangle|}.
\end{equation}
We set $L_{0,0}=L_{\mathrm{init}}$. For $k\geq1$, if
$|\langle \Delta w_k,\Delta v_k\rangle|>10^{-24}$ and
$L_k^{\mathrm{BB}}\in[L_{\min},L_{\max}]$, we set
$L_{k,0}=L_k^{\mathrm{BB}}$; otherwise,
$L_{k,0}=\min\left\{L_{\max},\max\left\{L_{\min},L_{k-1,0}/2\right\}\right\}.$
Starting from $L=L_{k,0}$, we repeatedly replace $L$ by $2L$
until \eqref{eq:implemented-robust-armijo} is satisfied.

\paragraph{Evaluation of the Moreau-gap violation.}
For reporting the estimated Moreau-gap constraint violation, we
compute an independent approximation $\widehat z$ of the proximal
point $z_\rho(x,y)$ by applying projected gradient with monotone
backtracking to the full lower-level Moreau subproblem, initialized
at $z=y$. For each task, we use the same proximal parameter $\rho$ as
IVSP and apply it to all compared methods. The proximal solve is
terminated when the norm of the projected-gradient mapping is at most
$10^{-3}$. Using the resulting $\widehat z$, we compute
$
\widehat H_\rho(x,y)
=
g(x,y)-g(x,\widehat z)
-\frac{1}{2\rho}\|\widehat z-y\|^2
$
and report
$
\widehat\Phi_\rho^{\rm f}(x,y)
:=
[\widehat H_\rho(x,y)-\delta]_+,
$
which estimates the feasibility residual $\Phi_\rho^{\rm f}$ in
Definition~\ref{def:approx-KKT}. This independent proximal solve is
used only for evaluation.

\paragraph{Common IVSP settings.}
For IVSP, following \citet{xu2026smoothing}, we use the centered-Huber
smoothing of the plus function,
\begin{equation}
\label{eq:centered-huber}
\phi_\mu(t)
=
\begin{cases}
0, & t\leq-\mu/2,\\[1mm]
\dfrac{(t+\mu/2)^2}{2\mu}, & |t|<\mu/2,\\[2mm]
t, & t\geq\mu/2.
\end{cases}
\end{equation}
This is the optimal inner $1/\mu$-smoothing of $[\cdot]_+$ in the
sense of \citet{samakhoana2026optimal}, with uniform approximation
error $\mu/8$. Accordingly, the constants in
\eqref{eq:a-phi} and \eqref{eq:phi-derivative-lip} are
$a_\phi=1/8$ and $b_\phi=1$, respectively.

Across all experiments, we use
$\mu_k=\mu_0(\lfloor k/B\rfloor+1)^{-r}$,
$\widehat{\gamma}_t=\gamma_0(t+1)^{-1/2}$, and
$\zeta_k=\bar{\zeta}\mu_k^q$,
with $r=1/3$ and $q=1/r+0.01$.
The common parameter settings are
$L_{\min}=10^{-11}$, $L_{\max}=10^{11}$,
$\delta=10^{-3}$, $c_1=10^{-4}$, $c_2=1$,
$\underline{\vartheta}=10^{-24}$,
$\hat{\vartheta}=1.99$, and $B_R=10^3$.

\paragraph{Synthetic nonconvex problem.}
For the synthetic experiment, we use the nonconvex bilevel problem
described in Section~\ref{sec:synthetic-nonconvex}, with
$n=1000$, $a=2$, $c_i=2$, $x^0=-6$, and $y^0=0$.

For IVSP, we additionally use
$\sigma=1$, $L_{\rm init}=1$, $L_{xy}=\sqrt{1000}$, and $L_{yy}=1$.
The main method-specific hyperparameters are summarized in
Table~\ref{tab:synthetic-hyperparameters}.

\begin{table}[t]
\centering
\caption{Main hyperparameter settings for the synthetic nonconvex experiment.}
\label{tab:synthetic-hyperparameters}
\small
\setlength{\tabcolsep}{4pt}
\begin{tabular}{lp{0.76\linewidth}}
\toprule
Method & Hyperparameters \\
\midrule

IVSP &
$\mu_0=9\times10^{-3}$,
$B=1$,
$\gamma_0=0.9$,
$\bar{\zeta}=3$,
$\rho=0.5$.
\\[1mm]

PNGBiO &
$\alpha_k=0.70(k+1)^{-1/2}$,
$\beta_k=0.20(k+1)^{-0.05}$,
$\eta_k=0.015$,
$\gamma=0.5$,
$c_k=0.004(k+1)^{0.21}$.
\\[1mm]

MEHA &
$\alpha_k=0.004$,
$\beta_k=0.02$,
$\eta_k=0.05$,
$\gamma=0.5$,
$c_k=0.1(k+1)^{0.28}$.
\\[1mm]

SLM &
$\eta=1.5\times10^{-3}$,
$\alpha=4.0\times10^{-3}$,
$\gamma=0.3$,
$\tau=10^{-8}$,
$T_r=10$.
\\[1mm]

F$^2$SA &
$\alpha_k=0.01$,
$\gamma_k=0.1$,
$\lambda_0=0.0195$,
$T=1$.
\\[1mm]

V-PBGD &
$\alpha=0.003$,
$\beta=0.05$,
$\gamma_k=0.0025(k+1)^{0.6}$,
$T=1$.
\\[1mm]

BOME &
$\xi=0.05$,
$\alpha=0.05$,
$\eta=0.01$,
$T=1$.
\\

\bottomrule
\end{tabular}
\end{table}

\paragraph{Data hyper-cleaning.}
For data hyper-cleaning~\citep{ren2018learning, MEHA, PNGBiO}, let
$\mathcal D_{\rm tr}=\{(u_i,\widetilde q_i)\}_{i=1}^{n}$
denote the corrupted training set and let
$\mathcal D_{\rm val}$ denote the clean validation set.
The UL variable $x\in[-100,100]^n$ assigns the weight
$\omega(x_i)$ to training example $i$, where
$\omega(t)=1/(1+e^{-t})$ is the sigmoid function.

The bilevel formulation is
\begin{equation}
\label{eq:app-hypercleaning}
\begin{aligned}
    &\min_{\substack{x\in[-100,100]^{5000}\\ \, y\in[-100,100]^{25450}}}
    \quad
    \frac{1}{|\mathcal D_{\rm val}|}
    \sum_{(u,q)\in\mathcal D_{\rm val}}
    \ell(h_y(u),q)
    \\
    \text{s.t.}\quad
    y\in
    &\argmin_{\widetilde y\in[-100,100]^{25450}}
    \frac{1}{|\mathcal D_{\rm tr}|}
    \sum_{i=1}^{n}
    \omega(x_i)
    \ell(h_{\widetilde y}(u_i),\widetilde q_i).
\end{aligned}
\end{equation}

The classifier $h_y$ is a two-layer fully connected neural network
with 784 input units, 32 sigmoid hidden units, and 10 output units.
It therefore contains
$784\times32+32+32\times10+10=25{,}450$ parameters.
The sigmoid activation makes the LL objective smooth, while the
neural-network parameterization makes it nonconvex. This problem is an instance of the sigmoid-reweighted classification
model in Corollary~\ref{cor:ennamcq-hypercleaning}.
Indeed, with $M=B_y=100$, $d_h=32$, and $m=10$,
\[
\frac{2B_y(d_h+1)+\log m}{1+e^M}
< 2.5\times10^{-40}
<10^{-3}=\delta.
\]
Hence, ENNAMCQ holds for both hyper-cleaning experiments.

For both MNIST~\citep{MNIST} and FashionMNIST~\citep{FashionMNIST}, we use 5,000 training, 5,000 validation, and 10,000 test examples.
We randomly corrupt 50\% of the training labels by replacing each
selected label with an incorrect class label, while the validation and test sets remain clean.
The same corrupted training set is used for all compared methods.
The variables are initialized with $x^0=0$ and
$y^0=0.01\,\varepsilon$, where $\varepsilon$ has independent
standard normal entries.
We report validation loss, test error, and the Moreau-gap constraint violation as functions of wall-clock time.

We additionally use
$L_{\rm init}=1$, $L_{xy}=1$, $\sigma=1$ for both hyper-cleaning tasks,
with $L_{yy}=1$ on MNIST and $L_{yy}=20$ on FashionMNIST.
The main method-specific hyperparameters are summarized in
Table~\ref{tab:hypercleaning-settings}.

\begin{table}[t]
\centering
\caption{Main hyperparameter settings for data hyper-cleaning.}
\label{tab:hypercleaning-settings}
\scriptsize
\setlength{\tabcolsep}{3pt}
\begin{tabular}{lp{0.39\linewidth}p{0.39\linewidth}}
\toprule
Method & MNIST & FashionMNIST \\
\midrule

IVSP &
$\mu_0=0.02$,
$B=20$,
$\gamma_0=10$,
$\bar{\zeta}=9$,
$\rho=0.05$
&
$\mu_0=10$,
$B=2$,
$\gamma_0=1$,
$\bar{\zeta}=5\times10^{-3}$,
$\rho=0.1$
\\[1mm]

PNGBiO &
$\alpha_k=\beta_k=(k+1)^{-1/2}$,
$\eta_k=0.05$,
$\gamma=0.1$,
$c_k=0.2(k+1)^{0.25}$
&
$\alpha_k=\beta_k=(k+1)^{-1/2}$,
$\eta_k=0.05$,
$\gamma=0.4$,
$c_k=0.2(k+1)^{0.25}$
\\[1mm]

MEHA &
$\alpha_k=9$,
$\beta_k=(k+1)^{-0.6}$,
$\eta_k=0.05$,
$\gamma=5$,
$c_k=0.2(k+1)^{0.005}$
&
$\alpha_k=10$,
$\beta_k=(k+1)^{-0.6}$,
$\eta_k=0.03$,
$\gamma=12$,
$c_k=0.1(k+1)^{0.005}$
\\[1mm]

SLM &
$\eta=2$,
$\alpha=2$,
$\gamma=0.2$,
$\tau=0.01$,
$T_r=1$
&
$\eta=2$,
$\alpha=2$,
$\gamma=0.2$,
$\tau=0.01$,
$T_r=1$
\\[1mm]

F$^2$SA &
$\alpha_k=3(k+1)^{-1/3}$,
$\gamma_k=0.2$,
$\lambda_0=0.1$,
$T=10$
&
$\alpha_k=4(k+1)^{-1/3}$,
$\gamma_k=0.2$,
$\lambda_0=0.1$,
$T=10$
\\[1mm]

V-PBGD &
$\alpha=2$,
$\beta=1$,
$\gamma=1$,
$T=1$
&
$\alpha=1$,
$\beta=1$,
$\gamma=1$,
$T=1$
\\[1mm]

BOME &
$\xi=3$,
$\alpha=1$,
$\eta=0.5$,
$T=1$
&
$\xi=2$,
$\alpha=1$,
$\eta=0.5$,
$T=1$
\\

\bottomrule
\end{tabular}
\end{table}

\paragraph{Few-shot learning.}

\begin{table}[t]
\centering
\caption{Main hyperparameter settings for few-shot learning on Omniglot.}
\label{tab:fewshot-settings}
\small
\setlength{\tabcolsep}{4pt}
\begin{tabular}{lp{0.76\linewidth}}
\toprule
Method & Hyperparameters \\
\midrule

IVSP &
$\mu_0=1.5$,
$B=20$,
$\gamma_0=0.035$,
$\bar{\zeta}=1$,
$\rho=0.1$.
\\[1mm]

PNGBiO &
$\alpha_k=0.35(k+1)^{-1/2}$,
$\beta_k=0.05(k+1)^{-1/2}$,
$\eta_k=10^{-3}$,
$\gamma=100$,
$c_k=0.067(k+1)^{0.08}$.
\\[1mm]

MEHA &
$\alpha_k=0.08(k+1)^{-1/2}$,
$\beta_k=0.05(k+1)^{-1/2}$,
$\eta_k=10^{-3}$,
$\gamma=100$,
$c_k=0.008(k+1)^{0.08}$.
\\[1mm]

SLM &
$\eta=55$,
$\alpha=10$,
$\gamma=0.01$,
$\tau=0.01$,
$T_r=10$.
\\[1mm]

F$^2$SA &
$\alpha_k=5(k+1)^{-1/3}$,
$\gamma_k=0.5$,
$\xi=0.15$,
$\lambda_0=1$,
$T=5$.
\\[1mm]

V-PBGD &
$\alpha=1$,
$\beta=0.01$,
$\gamma=10$,
$T=10$.
\\[1mm]

BOME &
$\xi=3$,
$\alpha=1$,
$\eta=0.01$,
$T=10$.
\\

\bottomrule
\end{tabular}
\end{table}

We consider episodic few-shot classification on Omniglot
\citep{Omniglot}, following the bilevel meta-learning formulation used
in \citet{MEHA}. The 1,623 character classes are randomly partitioned
into 1,100 training, 100 validation, and 423 test classes using a fixed
class split. The three class sets are mutually disjoint. We consider the
10-way 1-shot setting with five query examples per class. We generate
128 fixed training episodes, 100 validation episodes, and 100 test
episodes.

For each training episode $i\in\{1,\ldots,T\}$, where $T=128$, let
$\mathcal S_i$ and $\mathcal Q_i$ denote its support and query sets,
respectively. Each support set contains one example from each of 10
classes, and each query set contains five examples from each class.

The UL variable $x$ parameterizes a shared four-block convolutional
feature extractor. Each block consists of a $3\times3$ convolution with
64 output channels, ReLU activation, $2\times2$ max pooling, and batch
normalization. For Omniglot, the successive spatial resolutions are
$28\to14\to7\to3\to1$, resulting in a 64-dimensional feature vector.
Batch-normalization statistics are computed separately for each task
from the image set being processed.

For an episode dataset
$\mathcal D=\{(u_j,q_j)\}_{j=1}^{|\mathcal D|}$, let
$F_{x,\mathcal D}(u_j)\in\mathbb R^{64}$ denote the feature vector of
$u_j$ produced by the shared feature extractor, with the
batch-normalization statistics computed from $\mathcal D$.
For each training episode $i$, the LL variable is a linear softmax
classifier
$y_i=(W_i,b_i)$, where
$W_i\in\mathbb R^{10\times64}$ and $b_i\in\mathbb R^{10}$.
We define
\[
\mathcal L_{\mathcal D}(x,y_i)
=
\frac{1}{|\mathcal D|}
\sum_{j=1}^{|\mathcal D|}
\ell\!\left(
W_iF_{x,\mathcal D}(u_j)+b_i,q_j
\right),
\]
where $\ell$ denotes the 10-class cross-entropy loss. The corresponding
bilevel problem is
\begin{equation}
\label{eq:app-fewshot}
\begin{aligned}
    \min_{\substack{x\in[-5,5]^{111936}\\
                    y\in[-20,20]^{83200}}}
    \quad&
    \frac{1}{T}
    \sum_{i=1}^{T}
    \mathcal L_{\mathcal Q_i}(x,y_i)
    \\
    \text{s.t.}\quad
    y\in\argmin_{\widetilde y\in[-20,20]^{83200}}
    \quad&
    \frac{1}{T}
    \sum_{i=1}^{T}
    \mathcal L_{\mathcal S_i}(x,\widetilde y_i).
\end{aligned}
\end{equation}

For fixed $x$, the LL objective in~\eqref{eq:app-fewshot} is convex in
$y$, since each task-specific model is a linear softmax classifier.
It is not strongly convex: in particular, the softmax cross-entropy is
invariant under a common shift of all class logits, which induces
nontrivial flat directions in the LL objective. As established in
Proposition~\ref{prop:ennamcq-fewshot}, this structure nevertheless
satisfies ENNAMCQ for every $\delta>0$. Thus, this experiment provides
a large-scale convex but non-strongly-convex LL instance for which the
constraint qualification required by our analysis holds.

For each validation and test episode, we initialize a fresh linear
classifier at zero and adapt it for 10 gradient steps on the support
set with stepsize $0.1$, while keeping the feature extractor fixed.
We then evaluate query loss and classification error using the
adapted classifier.

All training episodes are fixed before optimization and are shared
across all compared methods. The optimization is therefore deterministic
after fixing the dataset, episode construction, and initialization.
The convolutional parameters are initialized according to the standard
PyTorch convolutional initialization, the batch-normalization scale and
bias are initialized to one and zero, respectively, and all task-specific
linear classifiers are initialized at zero. We use data seed 42,
initialization seed 43, and class-split seed 0.

For IVSP, we additionally use
$\sigma=0$, $L_{\rm init}=0.1$, $L_{xy}=1$, and $L_{yy}=15$. The main method-specific hyperparameters are summarized in Table~\ref{tab:fewshot-settings}.

\ifarxiv
\end{appendices}
\setlength{\bibsep}{1pt} 	
\bibliographystyle{plainnat}
\bibliography{ivsp_bibfile}
\else
\fi
\end{document}